\documentclass[11pt]{article}
\usepackage[top=27mm, bottom=27mm, right=30mm,left=30mm]{geometry}

\usepackage[english]{babel}
\usepackage[utf8]{inputenc}
\usepackage[svgnames]{xcolor}
\usepackage[T1]{fontenc}
\usepackage{wrapfig} 

\usepackage{fancyhdr} 
\usepackage{lastpage} 
\usepackage{here}
\usepackage{color}
\usepackage{ragged2e}
\usepackage{array}
\usepackage{indentfirst} 
\usepackage{authblk} 

\usepackage{setspace}
\usepackage{float}
\usepackage[hidelinks]{hyperref} 
\usepackage{amssymb} 
\usepackage{amsmath,amsfonts} 
\usepackage{amsthm}
\usepackage{amstext}

\usepackage{stmaryrd} 
\usepackage{mathtools} 
\usepackage{dsfont} 
\usepackage{mathrsfs}
\usepackage{enumitem} 
\usepackage{subcaption} 
\usepackage{multicol} 
\usepackage{multirow} 

\numberwithin{equation}{section}
\numberwithin{figure}{section}
\counterwithin{table}{section}

\newtheoremstyle{notation}
  {}
  {}
  {}
  {}
  {\bfseries}
  {}
  {7pt}
  {\thmname{#1}\thmnumber{ #2}.\textnormal{\thmnote{ (#3)}}}

\theoremstyle{plain}
\newtheorem{theo}{Theorem}[section]
\newtheorem*{theo*}{Theorem}

\newtheorem{lemme}[theo]{Lemma}
\newtheorem*{lemme*}{Lemma}

\newtheorem*{sublemme*}{Sublemma}
\newtheorem{prop}[theo]{Proposition}
\newtheorem*{prop*}{Proposition}
\newtheorem{propri}[theo]{Property}
\newtheorem{coro}[theo]{Corollary}
\newtheorem{defi}[theo]{Definition}
\newtheorem*{defi*}{Definition}

\theoremstyle{definition}
\newtheorem{rem}[theo]{Remark}

\theoremstyle{notation}
\newtheorem*{nota}{Notation}

\newcommand{\partent}[1]{\left\lfloor #1 \right\rfloor} 

\newcommand{\N}{\mathbb{N}} 
\newcommand{\Z}{\mathbb{Z}} 
\newcommand{\Q}{\mathbb{Q}} 
\newcommand{\R}{\mathbb{R}} 

\newcommand{\dimH}{\textup{dim}_H\!} 

\newcommand\numberthis{\addtocounter{equation}{1}\tag{\theequation}} 

\allowdisplaybreaks 

\newcommand{\DocumentTitle}{Traces of functions in Besov spaces in Gibbs environment}
\newcommand{\DocumentTitleShort}{Traces in inhomogeneous Besov spaces}
\author[1]{Quentin Rible}
\author[1]{St\'ephane Seuret}
\affil[1]{Univ Paris Est Creteil, Univ Gustave Eiffel, CNRS, LAMA UMR8050, F-94010 Creteil, France}
\date{}

\title{\vspace{-1.5cm} \DocumentTitle \vspace{-0.4cm}}

\begin{document}

\setlength{\abovedisplayskip}{5pt}
\setlength{\belowdisplayskip}{5pt}

\maketitle

\begin{abstract}
    This paper investigates the traces of functions belonging to the inhomogeneous Besov spaces $B^{\xi}_{p,q}$, where $\xi$ is a product of capacities defined as power of Gibbs measures. We first establish that the traces of functions in $B^{\xi}_{p,q}$ along affine hyperplanes belong to another inhomogeneous Besov space. Furthermore, we derive an upper bound for the singularity spectrum of the traces of all functions in $B^{\xi}_{\infty,q}$. This bound is then refined for a prevalent set of functions in $B^{\xi}_{\infty,q}$, for which we explicitly compute the singularity spectrum of their traces. Notably, our analysis reveals that the regularity properties of these affine traces are highly sensitive to the choice of the hyperplane along which the trace is taken.
\end{abstract}

\section{Introduction}

In the context of turbulence, the study of local variations in the velocity of turbulent fluids led to the introduction of multifractal descriptions of signals and functions. Turbulent flows are indeed far from being spatially homogeneous, and the pointwise regularity of the velocity may vary significantly from one point to another; this is typical of multifractal functions and measures. A fundamental step was then made by Frisch and Parisi, who uncovered a striking connection between multifractality and the global scaling properties of the velocity \cite{Frisch-Parisi:1985:Multifractal-turbulence}. This connection, now known as the \emph{multifractal formalism}, provided a unifying framework that subsequently stimulated the development of multifractal analysis in a wide range of mathematical settings, including geometric measure theory, real analysis \cite{Chamizo-Ubis:2014:Multifractal_behavior_polynomial_Fourier_Series, Jaffard:1996:Riemann_function_singularities, Seuret-Ubis:2017:Riemann_series}, dynamical systems and ergodic theory, and Diophantine approximation \cite{Feng:2007:Gibbs_properties_conformal_measures, Shmerkin:2005:multifractal_formalism_self-similar_measures, Shmerkin-Solomyak:2016:absolute_continuity_measures}, among others.
In parallel, the development of numerical methods for the estimation of multifractal characteristics led to significant advances in signal and image processing \cite{Abry-Jaffard-Wendt:2009:wavelet_leader_texture_classification, Abry-Jaffard-Roux-Wendt:2009:Wavelet_Bootstrap}, and in other scientific fields \cite{Abry-Ciuciu-Dumeur-Jaffard-Saes:2024:Multifract_analysis_neuroscience, Billat-Jaffard-Nasr-Saes:2023:Marathon-physiological-analysis}.

It is important to emphasise that, in practice, measurements of turbulent velocity fields are typically performed along one-dimensional sets, rather than over full three-dimensional regions of $\R^3$. As a consequence, the multifractal properties of turbulent flows are often inferred from one-dimensional traces, which naturally raises the question of how accurately such traces reflect the multifractal structure of the underlying three-dimensional flow. This problem is in fact quite general and applies to all three-dimensional phenomena for which only partial measurements are available.

Our main objective is to rigorously examine the validity of this inference procedure of multifractality for functions in inhomogeneous Besov spaces $B^{\xi}_{p,q}$. We establish precise relationships between the regularity properties of functions belonging to $B^{\xi}_{p,q}$ and those of its traces restricted to affine subspaces of dimension $d < D$, see \eqref{eq:def_trace} below for the precise definition of traces.

Let us now give now precise definitions. 
Multifractal analysis consists in describing the pointwise behavior of functions $f:\R^D\to\R$, computing the singularity spectrum $\sigma_f$, and, when possible, relating $\sigma_f$ to the scaling function associated with $f$. Given $x_0\in\R^D$, $H\in\R_+$, a locally bounded function $f:\R^D\to\R$ belongs to $\mathcal{C}^{H}(x_0)$ if there exist a polynomial $P_{x_0}$ of degree less than $\partent{H}$, a constant $C>0$, and a neighborhood $V$ of $x_0$ such that
\begin{equation*} 
    \forall x\in V, \quad |f(x)-P_{x_0}(x-x_0)| \leq C\,|x-x_0|^H.
\end{equation*}
The \textit{pointwise H{\"o}lder exponent} of $f\in L^{\infty}_{\operatorname{loc}}(\R^D)$ at $x_0$ is then	
\begin{equation}\label{eq:hold_exp}
    h_f(x_0)=\sup\big\{H\in\R_+ : f\in\mathcal{C}^H(x_0)\big\}.
\end{equation}
The associated \textit{singularity} (or \textit{multifractal}) \textit{spectrum} $\sigma_f$ of $f$ is the mapping
\begin{equation}\label{eq:mf_spect}
    \sigma_f: h \mapsto \dim_H\!\big(E_f(h)\big), \quad \text{where } E_f(h) := \big\{x\in \R^D:h_f(x)=h\big\} .
\end{equation}
Here, $\dim_{H}$ stands for the Hausdorff dimension with the convention $\dimH(\emptyset)=-\infty$ (note that $E_f(h)=\emptyset$ for $h<0$).
The singularity spectrum provides a geometrical description of the distribution of the singularities of $f$. 

Let us emphasise that for large classes of functions $f$ (similarly, of Borel measures), the singularity spectrum of $f$ satisfies a so-called {\em multifractal formalism}, meaning that $\sigma_f$ coincides with the Legendre transform $ \tau_{f}^*$ of a scaling function $\tau_{f}: \R \to \R$ associated with $f$, i.e.
\begin{equation}\label{eq:mf_form}
    \sigma_f = \tau^*_{f}(h) := \inf_{q\in\R} (hq - \tau_{f}(q)).
\end{equation}
The scaling function $\tau_{f}$, defined using wavelets (see below and \cite{Jaffard:2004:Wavelet_techniques, Jaffard:2007:Wavelet_leaders}), describes the statistical scaling properties of $f$. When \eqref{eq:mf_form} holds, one says that the \textit{multifractal formalism} holds for $f$. This relationship is key, since when properly defined, $\tau_f$ has proved to be numerically estimable on real data, see for instance \cite{Jaffard:2007:Wavelet_leaders}. 

Results on the pointwise regularity of functions belonging to classical spaces such as Sobolev $W^{s,p}(\R^D)$ or Besov $B^{s}_{p,q}(\R^D)$ spaces have been obtained in \cite{Aubry-Bastin-Dispa:2007:prevalence_sobolev, Fraysse-Jaffard:2006:Almost_all_function_sobolev,Jaffard:2000:Frisch-Parisi-conjecture, Jaffard:1997:multifractal_formalism_all_functions, Jaffard-Meyer:2000:regularity_critical_besov_spaces}. They are of two types: first, identify a universal upper bound for the singularity spectrum of all functions in a given space, and second, when possible, compute the explicit value for the spectrum of a generic set of functions in the sense of prevalence or Baire categories. 

A remarkable feature of generic functions in $W^{s,p}(\R^D)$ and $B^{s}_{p,q}(\R^D)$ is that their singularity spectrum is invariably affine increasing \cite{Jaffard:2000:Frisch-Parisi-conjecture, Fraysse:2007:prevalent_validity_multifractal_formalism}. This stands in stark contrast to the fact that spectra estimated from real-world data via the multifractal formalism consistently exhibit a strictly concave profile, featuring both an increasing and a decreasing part \cite{Abry-Jaffard-Wendt:2009:wavelet_leader_texture_classification, Abry-Ciuciu-Dumeur-Jaffard-Saes:2024:Multifract_analysis_neuroscience,Billat-Jaffard-Nasr-Saes:2023:Marathon-physiological-analysis}.
Numerous generalisations of these classical Sobolev and Besov spaces have been studied extensively \cite{Ansorena-Blasco:1995:weighted_besov_spaces, Farkas-Leopold:2006:characterisations_generalised_smoothness_Besov, Moura:2007:characterisation_generalised_smoothness_besov, Loosveldt-Nicolay:2019:equivalent_definition_generalised_smoothness_besov, Kreit-Nicolay:2013:characterisation_generalised_Holder-Zygmund_spaces, Diening-Harjulehto-Hasto-Ruzicka:2011:variable_exponent_lebesgue_sobolev, Gaczkowski-Gorka-Pons:2016:variable_exponent_sobolev, Triebel:1983:theory_function_spaces_1, Jaffard:2000:Frisch-Parisi-conjecture, Fraysse:2007:prevalent_validity_multifractal_formalism}; for further details, see also \cite{Rible:2025:Prevalence_inhomogeneous_Besov}. However, in none of these cases do generic functions exhibit a concave singularity spectrum with both increasing and decreasing components.

Motivated by these observations, Barral and Seuret \cite{Barral-Seuret:2023:Besov_Space_Part_2} introduced inhomogeneous Besov spaces $B^{\xi}_{p,q}(\R^D)$, where Baire-typical functions may admit any prescribed admissible continuous concave singularity spectrum. The precise definition of these spaces will be addressed later. The inhomogeneity in these spaces is governed by a Hölder capacity $\xi$, which can be interpreted as a measure with strongly location-dependent mass distribution. This framework naturally captures the multifractal behavior observed in empirical data.

This paper deals with the multifractal analysis of traces of functions in inhomogeneous Besov spaces $B^{\xi}_{p,q}(\R^D)$ on $d$-dimensional affine subspaces of $\R^D$, when $\xi$ is the product of Gibbs measures on $\R^D$.
Let $0 < d < D$ be two fixed integers and $d':= D-d$. For $a\in\R^{d'}$, we denote $\mathcal{C}_a:=\{(x,a) :x\in\R^{d} \}$ the horizontal $d$-dimensional affine subspace of $\R^D$ passing by $(0^d,a)$.
The trace of a function $f$ on $\mathcal{C}_a$ is defined by
\begin{equation}\label{eq:def_trace}
    f_a := f\vert_{\mathcal{C}_a} : 
    \begin{array}{ccc}
        \R^d & \longrightarrow & \R \\
        x & \longmapsto & f(x,a).
    \end{array}
\end{equation}

This notion has been considered for functions in standard Besov spaces by Aubry-Maman-Seuret and Jaffard \cite{Aubry-Maman-Seuret:2013:Traces_besov_results, Maman:2013:Genericity_Thesis, Jaffard:1995:traces_negative_dim, Caetano-Haroske:2016:traces_besov_fractal} and weighted Besov spaces by Besoy-Haroske-Triebel \cite{Besoy-Haroske-Triebel:2022:traces_weighted_fct_spaces} for example.
Universal and generic results in the sense of prevalence and Baire categories on the pointwise regularity of traces of functions belonging to classical spaces have been obtained by Aubry, Maman and Seuret \cite{Aubry-Maman-Seuret:2013:Traces_besov_results, Maman:2013:Genericity_Thesis}.
This article presents two main results: first, for any function $f$ in an inhomogeneous Besov space $B^{\xi}_{p,q}(\R^D)$, we establish that all its traces $f_a$ belong to another specific inhomogeneous Besov space, which we characterise explicitly, while simultaneously demonstrating that most traces exhibit higher regularity than anticipated (Theorem \ref{theo:up_bound_spect_all_fct}). Second, we prove that for a prevalent set of functions in $B^{\xi}_{p,q}(\R^D)$, the multifractal properties of the traces $f_a$ can be determined (Theorem \ref{theo:prev_spect_inh_besov_inf_q}). Notably, unlike in standard Besov spaces—where a unique singularity spectrum is typically observed for most traces $f_a$—inhomogeneous spaces admit a richer structure: depending on the choice of $a$ (within large parameter sets that we describe), the traces may display distinct multifractal behaviors.

The present article describes two kinds of result: first, for any function $f$ in an inhomogeneous Besov spaces, we show that all its traces $f_a$ belong to an inhomogeneous Besov spaces that we identify, and simultaneously that most traces are more regular than expected (Theorem \ref{theo:up_bound_spect_all_fct}
). Second, we prove that for a prevalent set of functions, the multifractal properties of the traces can be explicitly determined (Theorem \ref{theo:prev_spect_inh_besov_inf_q}). In particular, contrarily to standard Besov spaces where there is only one well-identified singularity spectrum for most of the traces $f_a$, in inhomogeneous spaces, depending of the choice of $a$ (in large sets of parameters that we describe), the traces may have distinct multifractal behaviors.

Let us now give precise statements. Let $\mathcal{B}(\R^D)$ stand for the Borel sets in $\R^D$.
\begin{defi}\label{defi:hold_capa}
    Let $D\geq1$.
    The set of H{\"o}lder set functions on $\R^D$ is 
    \begin{equation*}
        \mathcal{H}(\R^D) 
        = \big\{ \xi:\mathcal{B}(\R^D)\to~\R_+ \cup \{+\infty\} : \exists C, s > 0,~\forall E,F \in \mathcal{B}(\R^D),
            \xi(E) \leq C |E|^s \big\}.
    \end{equation*} 
    The set of H{\"o}lder capacities is then
    \begin{equation*}
        \mathcal{C}(\R^D) 
        = \big\{\xi:\mathcal{H}(\R^D)\to~\R_+ \cup \{+\infty\} : \forall E,F \in \mathcal{B}(\R^D), E\subset F \Rightarrow \xi(E)\leq\xi(F) \big\}.
    \end{equation*} 
\end{defi}

When $\xi(E) \leq C |E|^s $ for all $E$, $\xi$ is said to be $s$-H\"older. The topological support $\operatorname{supp}(\nu)$ of $\nu \in \mathcal{H}(\R^D)$ is the set of points $x\in \R^D$ for which $\nu(B(x,r)) > 0$ for every $r > 0$. A capacity $\nu$ is fully supported on $A$ when $\operatorname{supp}(\nu)=A$.
We mainly look at capacities supported on $[0,1]^D$, but our results can be extended to $\R^D$.

For $\xi\in\mathcal{B}(\R^D)$, $s>0$, and $E \in\mathcal{B}(\R^D)$, define the set functions $\xi^{s}$, $\xi^{(+s)}$ and $\xi^{(-s)}$ by
\begin{equation*}
    \xi^{s}(E)=\xi(E)^s ~~\text{and}~~\xi^{(+s)}(E)=\xi(E)|E|^s,
\end{equation*}
and when $\xi$ is $s_0$-H\"older, set for all $s\in(0,s_0)$, 
$
    \xi^{(-s)}(E)=
    \begin{cases}
        0               &\text{if}~|E|=0, \\
        \xi(E)|E|^{-s} &\text{if}~0<|E|<\infty.
    \end{cases}
$
Observe that if $\xi\in\mathcal{C}(\R^D)$, then $\xi^{s}$ and $\xi^{(+s)}$ are still H{\"o}lder capacities.

Recall that $\xi\in\mathcal{C}(\R^D)$ is doubling if for some $C_{\xi}\geq 1$,
for every $x\in \R^D$ and $r>0$,  
\begin{equation}\label{eq:doub_meas_ball}
    \xi(B(x,2r))\leq C_{\xi} \ \xi(B(x,r)).
\end{equation}

\begin{rem}
    A doubling $s$-H{\"o}lder capacity $\xi\in \mathcal{C}([0,1]^d)$ satisfies, for $C>0$, for all $r>0$,
    \begin{equation}\label{eq:doubling_induce_P1}
        C^{-1} r^{\log_2(C_{\xi})} \leq \xi(B(x,r)) \leq C r^{s}.
    \end{equation}
\end{rem}

The multifractal properties of a set function are based on local dimensions. The convention $0^q=0$ is adopted.
\begin{defi}\label{defi:multi_frac_quantity_capacity}
    Let $\xi\in\mathcal{H}(\R^D)$. The \textit{lower local dimension} of $\xi$ at $x\in \operatorname{supp}(\xi)$ is 
    \begin{equation}\label{eq:h_mu_ball}
        \underline{h}_{\xi}(x) = \liminf_{r \to 0^+} \frac{\log \xi(B(x,r))}{\log(r)}.
    \end{equation}
    The \textit{singularity spectrum} of $\xi$ is given by
    \begin{equation*}
        \sigma_{\xi}\ : \ h\in\R \mapsto \dim_H \underline{E}_{\xi}(h) \quad \text{where} \quad \underline{E}_{\xi}(h)=\{x\in\operatorname{supp}(\xi) \ : \ \underline{h}_{\xi}(x)=h\}.
    \end{equation*}
 
    For every integer $j\geq 1$, let $\Lambda_j^D$ be the set of dyadic cubes of side length $2^{-j}$ in $[0,1]^D$. 
    
    The scaling function of $\xi\in\mathcal{H}(\R^D)$ with $\operatorname{supp}(\xi)\cap[0,1]^D \neq \emptyset$ is defined by
    \begin{equation*}
        \tau_\xi(q)=\liminf_{j\to+\infty} \frac{1}{-j}\log_2 \sum_{\lambda\in\Lambda_j^D} .\xi(\lambda)^q,
    \end{equation*}
  
\end{defi}
\noindent The minimal and maximal local dimension and the dimension of $\xi$ are defined by
\begin{align}
    h_{\xi}^{\min} &= \min\{\underline{h}_{\xi}(x) : x\in \operatorname{supp}(\xi)\}, \quad 
    h_{\xi}^{\max} = \max\{\underline{h}_{\xi}(x) : x\in \operatorname{supp}(\xi)\}, \label{eq:def_exp_extreme}\\
    \dim(\xi) &= \inf\{ \dimH(E) : \xi(E) =1 \}. \label{eq:def_dim_meas}
\end{align}

This paper considers capacities $\mu$ and $\nu$ defined as powers of Gibbs measures. Let $\mathcal{L}^D$ stand for the $D$-dimensional Lebesgue measure. Let us recall how they are defined, see \cite{Pesin-Weiss:1997:multi-frac-anal-gibbs-measures}.
\begin{defi}\label{defi:Gibbs_measure}
    Let $\varphi: \R^D \to \R$ be a $\Z^D$-periodic H{\"o}lder continuous function (called a potential). The sequence of measures defined on $\R^D$ whose densities are given by
    \begin{equation*}
        \nu_n(dx) = \frac{\exp(S_n \varphi(x))}{\int_{[0,1]^D} \exp(S_n \varphi(t)) \mathcal{L}^D (dt)}\mathcal{L}^D(dx), \quad \text{where } S_n \varphi(x) = \sum_{k=0}^{n-1} \varphi(2^k x),
    \end{equation*}
    converges weakly to the $\Z^D$-invariant Gibbs measure $\nu_{\varphi}$ associated with $\varphi$.
\end{defi}
These measures have been extensively studied \cite{Jaerisch-Sumi:2020:multi-frac-formalism-gibbs-measures,Parry-pollicott:1990:zeta_function,Patzschke:1997:self-conformal-measures,Pesin-Weiss:1997:multi-frac-anal-gibbs-measures,Pesin:1997:dimension_theory,Seuret:2018:Inhomogeneous-random-covering-Markov-Shifts}. We recall in Section \ref{sec:capacity_and_product_capa} their main multifractal properties. For the moment, we need the following ones: Given a fixed potential $\varphi: \R^D \to \R$:
\begin{enumerate}
    \item The Gibbs measure $\nu_{\varphi}$ is doubling, and $\operatorname{supp}(\nu_{\varphi})=\R^D$.
    \item For every $h\geq0$, $ 
        \sigma_{\nu_{\varphi}}(h)=\tau_{\nu_{\varphi}}^*(h) := \inf_{r\in \R}(rh-\tau_{\nu_{\varphi}}(r)). $
    \item For every $r\in\R$, setting
    \begin{equation}\label{eq:h_nu_r_def}
        h^{r}_{\nu_{\varphi}} = \tau'_{\nu_{\varphi}}(r),
    \end{equation}
    one has $\sigma_{\nu_{\varphi}}(h^{r}_{\nu_{\varphi}}) = r h^{r}_{\nu_{\varphi}} - \tau_{\nu_{\varphi}}(r)$. Also, $h^{\min}_{\nu_{\varphi}}=h^{+\infty}_{\nu_{\varphi}}$ and $h^{\max}_{\nu_{\varphi}}=h^{-\infty}_{\nu_{\varphi}}$.
    \item For every $r\in\R$, the Gibbs measure $\nu_{r\varphi}$ is called the auxiliary measure associated with $\varphi$ and $r$ (the auxiliary measure is associated with the potential $r\varphi$), and enjoys multifractal properties similar to those of $\nu_{\varphi}$. Then 
    \begin{equation}\label{eq:nur}
    \nu_{r\varphi}(\underline{E}_{\nu_{\varphi}}(h^{r}_{\nu_{\varphi}}))=1 \mbox{, \ \ \ i.e. for $\nu_{r\varphi}$-a.e. $x$, \ \ $\underline{h}_{\nu_{\varphi}}(x)=h^{r}_{\nu_{\varphi}}$}.
    \end{equation}
    
\end{enumerate}

Item (2) is referred to as the multifractal formalism for measures, and is analog to the formalism \eqref{eq:mf_form} for functions.
Item (3) follows from \eqref{eq:h_nu_r_def} and simple considerations on the Legendre transform. Item (4) states that the auxiliary measure $\nu_{r\varphi}$ is concentrated on those points $x$ whose local dimension for $\nu_{\varphi}$ equals $h^{r}_{\nu_{\varphi}}$, see Theorem \ref{theo:auxiliary_measure} for details on auxiliary measures. Since $\nu_{\varphi}$ is $\Z^D$-periodic, we identify $\nu_{\varphi}$ with its restriction to $[0,1]^D$.  
\begin{defi}
    The set of Gibbs capacities $\mathscr{E}_D\subset \mathcal{C}(\R^D)$ is 
    \begin{equation*}
        \mathscr{E}_D = \{\nu^s : \nu \text{ a Gibbs measure}, s>0\}.
    \end{equation*}
\end{defi}

To state our results, we recall the definition of inhomogeneous Besov spaces.
Write $L^D=\{0,1\}^D\setminus\{0^{D}\}$, and consider the family $\{\phi,\psi_\lambda:=\psi^l(2^j \cdot -k)\}_{\lambda=(j,k,l) \in \Z\times \Z^D\times L^D }$ generated by a collection of wavelets $(\phi,\psi^l)_{l\in L^D}$ (see Section \ref{sec:mr_wlt_ana} for details), such that every $f\in L^2(\R^D)$ can be written
\begin{equation*}
    f \overset{L^2}{=} \sum_{\lambda \in \Z\times \Z^D\times L^D} c_{\lambda} \psi_{\lambda} \quad \text{ with }\quad c_{\lambda}=2^{Dj} \int_{\R^D} f(x)\psi_{\lambda} (x)dx.
\end{equation*}
For $k=(k_1,\ldots,k_D)\in\Z^D$, $j\in\Z$, we often identify $\lambda=(j,k,l)$ with the dyadic cube $\prod_{i=1}^{D} [k_i 2^{-j}, (k_i+1) 2^{-j})$, and in particular we write $\xi(\lambda)$ for $\xi(\prod_{i=1}^{D} [k_i 2^{-j}, (k_i+1) 2^{-j}))$.

For any $f\in L^p(\R^D)$, $1\leq p \leq \infty$, $f$ can be written
\begin{equation}\label{eq:wlt_dec_on_R}
    f = \sum_{k\in\Z^D} \beta(k) \phi(\cdot-k) + \sum_{j\geq 0} \sum _{\lambda\in \{j\}\times\Z^D\times L^D} c_{\lambda} \psi_{\lambda}, \quad\textnormal{with}\quad \beta(k)=\int_{\R^D} f(x)\phi(x-k)dx.
\end{equation}
Without loss of generality, we consider functions supported in $[0,1]^D$ and write
\begin{equation}\label{eq:wlt_dec_on_01}
    f =  \beta(0) \phi + \sum_{j\geq 0} \sum _{\lambda\in \Lambda^D_j \times L^D} c_{\lambda} \psi_{\lambda}.
\end{equation}

Indeed, any function $f\in L^2(\R^{D})$ can be written as a sum $f = g+\sum_{\ell\in \Z^D}f_\ell(\cdot +\ell)$, where $g$ is a function as regular as the wavelets and $f_\ell$ can be written as \eqref{eq:wlt_dec_on_01}.
Since we are interested only in local and multifractal properties, and using that the Hausdorff dimension of a countable union is the supremum of the dimension of all sets, the singularity spectrum of a locally bounded function $f\in L^2(\R^{D})$ is obtained as the supremum of those of the $f_\ell$.

\smallskip

\textbf{So, from now on, we will work with functions defined on $[0,1]^D$ written as \eqref{eq:wlt_dec_on_01} and with a set function $\xi$ with $\operatorname{supp}(\xi)=[0,1]^D$.}
By a slight abuse of notation, we write $\mathcal{H}([0,1]^D) = \{\xi \in \mathcal{H}(\R^D):\operatorname{supp}(\xi)=[0,1]^D\}$ and $\mathcal{C}([0,1]^D) =\mathcal{H}([0,1]^D) \cap \mathcal{C}(\R^D)$.

Barral and Seuret introduced the following inhomogeneous Besov spaces.

\begin{defi}\label{defi:inh_besov_wlt}
    Let $\xi\in\mathcal{H}([0,1]^D)$ satisfying that for some $C>0$ and $0<s_1<s_2$,
    \begin{equation}\label{eq:P1_ball}
        \mbox{for all $x\in [0,1]^D$,~~for all $r>0$, } ~~~C^{-1} r^{s_2} \leq \xi(B(x,r)) \leq C r^{s_1}.
    \end{equation}
    Let $p\in[1,+\infty]$, consider an integer $n \geq \lfloor {s_2+ {D}/{p}}+1\rfloor$ and wavelets $(\phi,\psi^l)_{l\in L^{D}} \in C^n(\R^d)$ compactly supported, having at least $n$ vanishing moments (see Section \ref{sec:mr_wlt_ana}). A function $f$ written as \eqref{eq:wlt_dec_on_01} belongs to $B^{\xi}_{p,q}([0,1]^d)$ if and only if $f\in L^p([0,1]^d)$ and $|f|_{\xi,p,q}<+\infty$, where 
    \begin{equation}\label{eq:snorm_inh_besov}
        |f|_{\xi,p,q}=\big\|(\varepsilon^{\xi,p}_{j})_{j\in\N}\big\|_{\ell^q(\N)}, \quad \text{with } \quad \varepsilon^{\xi,p}_{j}=\Big\|\Big( \frac{c_{\lambda}}{\xi(\lambda)} \Big)_{\lambda\in \Lambda^D_j \times L^D}\Big\|_{\ell^p(\Lambda^D_j \times L^D)}.
    \end{equation}
    
    Then define 
    \begin{equation}\label{eq:besov_tilde}
        \widetilde{B}^{\xi}_{p,q}([0,1]^D) = \bigcap_{0 < \varepsilon < \min(1,s_1)} B^{\xi^{(-\varepsilon)}}_{p,q}([0,1]^D).
    \end{equation}
\end{defi}
Formula \eqref{eq:snorm_inh_besov} implies that $|c_\lambda | \leq C \,\xi(\lambda)$, for every cube $\lambda$. Recalling that the H\"older spaces $C^\alpha([0,1]^{D})$ are characterised by the fact that $|c_\lambda | \leq C\,|\lambda|^\alpha$ for every $\lambda$, the spaces $ B^{\xi}_{p,q}([0,1]^{D})$ can be understood as an inhomogeneous generalisation of $C^\alpha([0,1]^{D})$, since the value of $\xi(\lambda)$ varies a lot depending on $\lambda$. This interpretation justifies the term environment for $\xi$.

\smallskip

Taking $\xi = (\mathcal{L}^D)^{\frac{s}{D}-\frac{1}{p}}$ with $s > D/p$, it implies that $B^{\xi}_{p,q}([0,1]^D) = B^{s}_{p,q}([0,1]^D)$. 
Also, when $p=q=\infty$, $B^{\xi,\psi}_{\infty,\infty}([0,1]^{D})$ is the space of those functions $f$ for which 
\begin{equation*}
    C=\sup_{\lambda\in \Z\times \Z^{D}\times L^{D}} \frac{|c_{\lambda}|}{\xi(\lambda)} <+\infty.
\end{equation*}

It is proved in \cite{Barral-Seuret:2023:Besov_Space_Part_2} that the definition of $B^{\xi}_{p,q}([0,1]^D)$ and $\widetilde{B}^{\xi}_{p,q}([0,1]^D)$ is independent of the choice of $\psi$ as soon as $\xi$ is a doubling, and that the norms $\|\cdot\|_{L^p}+|\cdot|_{\xi,p,q}$ are equivalent for any choice of $\psi$. Since this will be our situation, we omit the dependence on $\psi$ in the notations.

\begin{rem}\label{rem:incl_besov_xi_eps}
    By construction, the family of Banach spaces $\big\{ B^{\xi^{(-\varepsilon)}}_{p,q}([0,1]^D) \big\}_{0<\varepsilon<\min(1,s_1)}$ satisfies that, for $0<\varepsilon<\varepsilon'<\min(1,s_1)$, $B^{\xi^{(-\varepsilon)}}_{p,q}([0,1]^D) \hookrightarrow B^{\xi^{(-\varepsilon')}}_{p,q}([0,1]^D)$.
    Hence, in \eqref{eq:besov_tilde}, the intersection can be taken on any interval of the form $(0,\varepsilon)$ with $\varepsilon <\min(1,s_1)$.
    
    Also, $\widetilde{B}^{\xi}_{p,q}([0,1]^d)$ is a Fréchet space \cite{Barral-Seuret:2023:Besov_Space_Part_2}, and a translation-invariant complete metric inducing the same topology as $\big\{ \|\cdot\|_{B^{\xi^{(-\varepsilon)}}_{p,q}} := \|\cdot\|_{L^p}+|\cdot|_{\xi^{(-\varepsilon)},p,q} \big\}_{0<\varepsilon<\min(1,s_1)}$ is
    \begin{equation}\label{eq:metric_B_xi_tilde}
        \forall f,g \in \widetilde{B}^{\xi}_{p,q}([0,1]^d), \quad \mathscr{D}(f,g)=\sum_{n >\max(1,s_1^{-1})} 2^{-n} \frac{\|f-g\|_{B^{\xi^{(-1/n)}}_{p,q}}}{1+\|f-g\|_{B^{\xi^{(-1/n)}}_{p,q}}}.
    \end{equation}
\end{rem}

\noindent We recall the multifractal properties of prevalent functions in $\widetilde{B}^{\xi}_{\infty,q}([0,1]^D)$ \cite{Rible:2025:Prevalence_inhomogeneous_Besov}. We will come back soon to the notion of prevalence, for the moment the reader shall have in mind that this is a natural substitute for "almost every element" in infinite dimensional spaces. 

\begin{defi}\label{defi:SMF}
    The capacity $\xi\in\mathcal{C}([0,1]^D)$ is said to satisfy the strong multifractal formalism (SMF) if for every $h\geq 0$, $\sigma_\xi(h)=\tau_\xi^*(h)=\dim_H E_{\xi}(h)$, where $E_{\xi}(h)$ is the set of points $x$ where the local dimension \eqref{eq:h_mu_ball} is a limit and equals $h$.
\end{defi}

\begin{theo}[\cite{Rible:2025:Prevalence_inhomogeneous_Besov}]\label{theo:Rible_prevalence} 
    If $\xi\in \mathcal{C}([0,1]^D)$ is doubling and satisfies the SMF, then there exists a prevalent set of functions $f\in\widetilde{B}^{\xi}_{\infty,q}([0,1]^D)$ such that $\sigma_f=\tau_\xi^*=\sigma_\xi$.
\end{theo}

Our objective is to study the multifractal properties of the traces of $f\in\widetilde{B}^{\xi}_{p,q}(\R^D)$ when $\xi$ is a product of Gibbs measures. More precisely, if $D=d+d'$, for $\mu_{\varphi} \in \mathscr{E}_d$ (resp. $\nu_{\widetilde{\varphi}} \in\mathscr{E}_{d'}$) associated to a potential $\varphi:\R^d\to\R$, (resp. $\widetilde{\varphi}:\R^{d'}\to\R$), we consider the capacity on $\R^D$
\begin{equation*}
    \xi= \mu_{\varphi} \times \nu_{\widetilde{\varphi}}
\end{equation*}
i.e. for $A\in \mathcal{B}([0,1]^d)$, $B\in \mathcal{B}([0,1]^{d'})$, $\xi(A \times B) := \mu_{\varphi}(A) \cdot \nu_{\widetilde{\varphi}}(B)$. We omit the mention to $\varphi$ and $\widetilde{\varphi}$, and write simply $\mu$, $\nu$, and also $\nu_r$ for the auxiliary measures associated with $\nu$.

\begin{figure} 
    \centering
    \includegraphics[page=3,width=.65\textwidth]{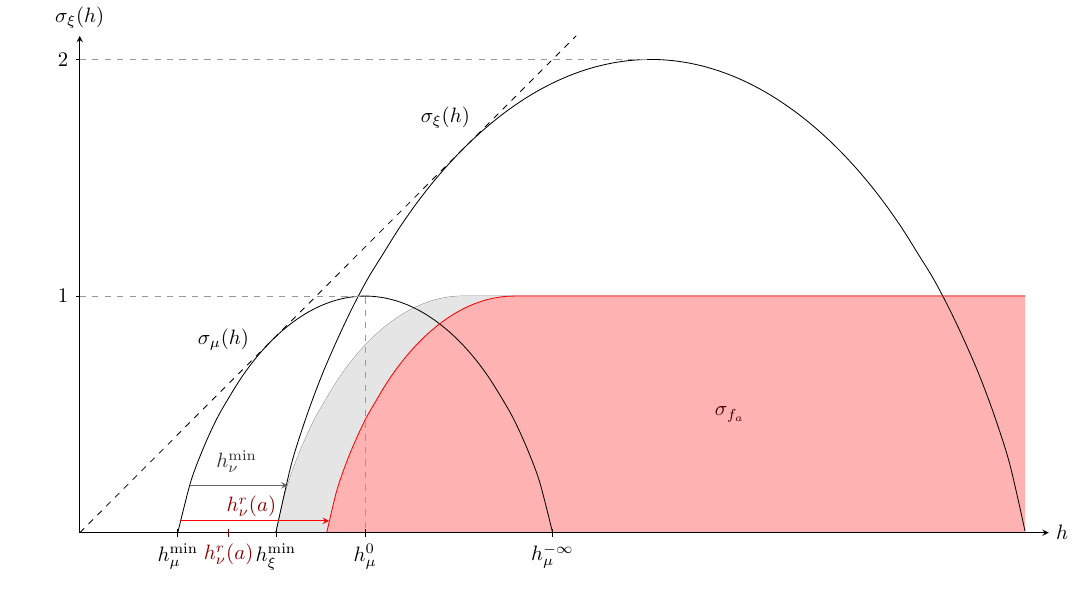}\vspace{-2mm}
    \caption{Upper-bound for $\sigma_{f_a}(h)$ for all $a\in[0,1]^{d'}$ in grey, for $\nu_r$-almost all $a\in[0,1]^{d'}$ in red.}
    \label{fig:up_bound_spect_trace}
\end{figure}

Theorem \ref{theo:Rible_prevalence} applies to $\xi=\mu\times\nu$: there is a prevalent set of functions $f\in\widetilde{B}^{\xi}_{p,q}([0,1]^D)$ for which $\sigma_f=\tau_\xi^*$. The spectrum of $\xi$ is determined in terms of $\mu$ and $\nu$, see Section \ref{sec:product_capa}.
\begin{prop}\label{prop:spect_prod_xi}
    Let $\mu\in\mathscr{E}_d$ and $\nu\in\mathscr{E}_{d'}$ with $D=d+d'$. Let $\xi=\mu\times \nu$be the product capacity on $[0,1]^D$.
    The product capacity $\xi$ satisfies the SMF and verifies $\tau_\xi=\tau_\mu+\tau_\nu$.
\end{prop}

For standard Besov spaces, the trace operator $f\mapsto f_a$ maps $B^{s}_{p,q}(\R^D)$ to $B^{s-d'/p}_{p,q}(\R^d)$, reflecting a loss of regularity when taking traces. However, for $\mathcal{L}^{d'}$-a.e. $a$, $f_a \in B^{s}_{p,q}(\R^d)$.
Our first result shows that the same phenomenon occurs for inhomogeneous spaces.
\begin{theo}\label{theo:incl_trace}
    Let $p,q\in[1,+\infty]$. Let $\mu\in\mathscr{E}_d$ and $\nu\in\mathscr{E}_{d'}$ with $D=d+d'$. Let $\xi$ be the product capacity on $[0,1]^D$ defined by $\xi = \mu \times \nu$.
    
    For all $f\in \widetilde{B}^{\xi}_{p,q}([0,1]^{D})$ and all $r\in\R$, for $\nu_{r}$-almost all $a\in[0,1]^{d'}$, one has 
    \begin{equation*}
        f_a \in \widetilde{B}^{\mu^{(+\Gamma_{\nu,r})}}_{p,1}([0,1]^d) \quad \text{with} \ \Gamma^p_{\nu,r} = h^{r}_{\nu} + {\dim(\nu_{r})}/{p}.
    \end{equation*}
\end{theo}

By convention $\alpha/{\infty}=0$. Observe that $f_a$ belongs to a space whose regularity depends on the local regularity of $\nu$ at $a$, since by \eqref{eq:nur}, $\nu_r$-a.e. points $a$ are such that $h_{\nu}(a)=h_\nu^r$. 

For the next results, \textbf{we focus on the case $p=\infty$.} First, an upper-bound for the singularity spectrum of traces of all functions $f\in \widetilde{B}^{\xi}_{\infty,q}([0,1]^{D})$ is obtained.
\begin{theo}\label{theo:up_bound_spect_all_fct}
    For all $f\in \widetilde{B}^{\xi}_{\infty,q}$ and $a\in[0,1]^{d'}$, one has for every $h\in\R_+$
    \begin{equation}\label{eq:theo_up_b_spect:all_a}
        \sigma_{f_a}(h) \leq 
        \begin{cases}
            \sigma_{\mu}(h-h^{\min}_{\nu}) &\text{if } h\leq h^{0}_{\mu} +h^{\min}_{\nu}, \\
           \hspace{3mm} d &\text{if } h> h^{0}_{\mu} +h^{\min}_{\nu}.
        \end{cases}
    \end{equation}

\noindent For all $f\in \widetilde{B}^{\xi}_{\infty,q}$, for every $r\in\R$ and $\nu_r$-almost all $a\in[0,1]^{d'}$, one has for every $h\in\R_+$
    \begin{equation}\label{eq:theo_up_b_spect:almost_a}
        \sigma_{f_a}(h) \leq 
        \begin{cases}
            \sigma_{\mu}(h-h^{r}_{\nu}) &\text{if } h\leq h^{0}_{\mu} +h^{r}_{\nu}, \\
           \hspace{3mm}      d &\text{if } h>h^{0}_{\mu}+h^{r}_{\nu}.
        \end{cases}
    \end{equation}
\end{theo}

These upper bounds can be sharpened and turned into equalities for prevalent sets of functions in $\widetilde{B}^{\xi}_{\infty,q}$. As quickly mentioned above, prevalence theory is used to supersede the notion of Lebesgue measure for Borel sets in any real or complex topological vector space $E$ endowed with its Borel $\sigma$-algebra $\mathcal{B}(E)$. Christensen \cite{Christensen:1972:sets_Haar_measure_zero} and Hunt, Sauer and York \cite{Hunt-Sauer-Yorke:1992:prevalence} proposed this notion independently.
In our article, all Borel measures $\vartheta$ on $(E, \mathcal{B}(E))$ will be automatically completed, i.e. for every set $A\subset E$, one sets $\vartheta(A) := \vartheta(B)$ if $B \in \mathcal{B}(E)$ and the symmetric difference $A\Delta B$ is included in some $D\in\mathcal{B}(E)$ with $\vartheta(D) = 0$. 
\begin{defi}
    A set $A\subset E$ is universally measurable if it is measurable for any (completed) Borel measure on $E$.
    A universally measurable set $A\subset E$ is Haar-null if there exists a finite completed measure $\vartheta$ that is positive on some compact subset $K$ of $E$ and such that 
    \begin{equation}
        \text{for every } x\in E, \quad \vartheta(A+x)=0.
    \end{equation}
    The measure $\vartheta$ used to show that some subset is Haar-null is called \textit{transverse}. 
    More generally, a set that is included in a Haar-null universally measurable set is also called Haar-null. 
    
    Finally, the complement in $E$ of a Haar-null subset $S$, denoted $S^{\complement}$, is called prevalent. 
\end{defi}

Hunt pioneered the use of prevalence in function spaces \cite{Hunt:1994:prevalence_continuous_functions}. More recently, Bayart and Heurteaux \cite{Bayart-Heurteaux:2013:Hausdorff_Dimension_Graphs_Prevalent_Continuous_Functions} and Balka, Darji and Elekes \cite{Balka-Darji-Elekes:2016:dimension_graph_continous_maps} studied the dimension of images of continuous function, Fraysse, Jaffard and Kahane \cite{Fraysse-Jaffard-Kahane:2005:propriete_analyse,Fraysse-Jaffard:2006:Almost_all_function_sobolev} and then Aubry, Maman and Seuret \cite{Aubry-Maman-Seuret:2013:Traces_besov_results} studied, respectively, the prevalent regularity of functions and the trace function $f_a$ of function $f$ in Besov spaces. Olsen studied the multifractal and $L^q$-dimension of prevalent measure on a compact of $\R^d$ \cite{Olsen:2010:Lq_dimension_prevalent_measure,Olsen:2010:multifractal_dimension_prevalent_measure}.

Our theorem regarding the prevalent multifractal behavior in $\widetilde{B}^{\xi}_{\infty,q}([0,1]^D)$ is the following. It depends on the existence of a wavelet satisfying a property that we call $(R)$, we come back to this after the statement.

\begin{theo}\label{theo:prev_spect_inh_besov_inf_q}
    Let $\mu\in\mathscr{E}_d$, $\nu\in\mathscr{E}_{d'}$. Consider the product measure $\xi=\mu \times \nu \in \mathcal{C}([0,1]^D)$, with $D=d+d'$.
    Call $C,s_1,s_2$ the real numbers such that \eqref{eq:P1_ball} is satisfied.
    
    Let $n \geq \partent{s_2+\frac{D}{p}}+1$. 
    Assume that there exists a compactly supported wavelet $\psi\in C^n(\R)$ having (at least) $n$ vanishing moments and satisfying the property $(R)$. 
    
    Let $(r_n)_{n\geq \N}$ be a countable sequence of real numbers.  
    There exists a prevalent set of function $f\in \widetilde{B}^{\xi}_{\infty,q}([0,1]^D)$ such that for every $n\in\N$, for $\nu_{r_n}$-almost all $a\in(0,1]^{d'}$, 
    \begin{equation*}
        \sigma_{f_a}(h) = \begin{cases}
            \sigma_{\mu}(h-h^{r_n}_{\nu}) & \text{if } h\in \operatorname{supp}(\sigma_{\mu})+h^{r_n}_{\nu} \\
            -\infty & \text{else}.
        \end{cases}
    \end{equation*}
\end{theo}

\begin{figure} 
    \centering
    \includegraphics[page=4,width=.65\textwidth]{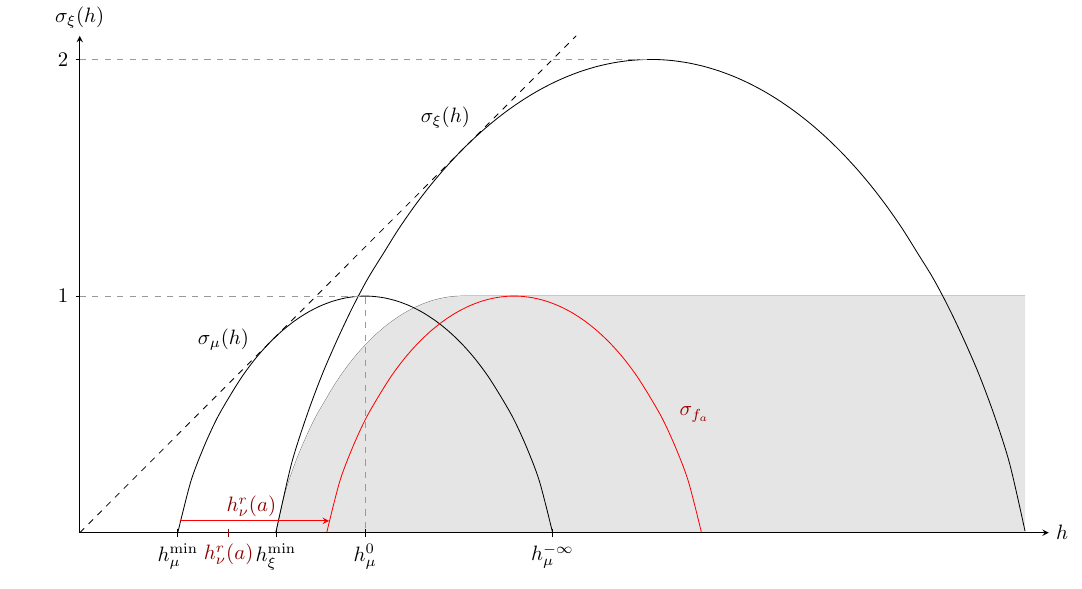}\vspace{-2mm}
    \caption{Spectrum $\sigma_{f_a}(h)$ for almost all $f$ and $\nu_r$-almost all $a\in[0,1]^{d'}$.}
    \label{fig:spect_trace_prev}
\end{figure}
Recalling that for $r=0$ the auxiliary measure $\nu_r$ is the Lebesgue measure $\mathcal{L}^{d'}$, Theorem \ref{theo:prev_spect_inh_besov_inf_q}
 gives that for $\mathcal{L}^{d'}$-almost every $a$, the singularity spectrum of $f_ a$ is $\sigma(\cdot-h_\nu^0)$. This is exactly similar to the result of \cite{Aubry-Maman-Seuret:2013:Traces_besov_results} for standard Besov spaces. Here, in addition to what happens in standard Besov spaces, the traces do not all exhibit the same multifractal properties, and we are able to compute the singularity spectrum of many other traces: $r$ being fixed, for $\nu_r$-almost every $a$, the singularity spectrum $\sigma_{f_a}$ of the trace $f_a$ coincides with that of $\mu$ shifted by $h^{r}_{\nu}$.
This is somehow expected since the environment $\xi$ is not spatially homogeneous, but it remains surprising nonetheless to be able to precisely describe this variety of regularity.
Our work follows in the tradition of standard theorems in geometric measure theory, such as those of Marstrand--Mattila \cite{Mattila99},
where one studies the dimension of the slices of a set $E \subset \R^D$ of dimension $\alpha$ along affine subsets of dimension \emph{strictly less than} $D$.
In most of these results, there is a single almost-sure behavior, and the set of exceptional behaviors is often small, or even of dimension $0$.
Here, by focusing on traces, we undertake a comparable analysis, but the inhomogeneity of the initial function implies that the regularity of the traces is highly dependent on the specific trace considered.

It is not entirely satisfactory that the result is proved only for a countable sequence $(r_n)_{n\in \N}$, that can be dense in $\R$. We conjecture that Theorem \ref{theo:prev_spect_inh_besov_inf_q} holds for every $r\in \R$ simultaneously on a prevalent set. Similarly, it is a natural question to investigate the situation where $p<\infty$.

Theorem \ref{theo:prev_spect_inh_besov_inf_q} depends on the existence of wavelets enjoying the following property.
\begin{defi}\label{defi:prop_R}
    A function $\psi:\R\to\R $ satisfies the property $(R)$ if:
    \begin{enumerate}[label=$(R_{\arabic*})$]
        \item $\psi\in\mathcal{C}^1(\R)$ and there exists $K_{\psi}\in\N^*$ such that $\overline{\operatorname{supp}(\psi)} = [0,K_{\psi}]$.
        \item The cardinality of the set $\mathcal{Z}(0,0):=\psi^{-1}(\{0\})\bigcap[0,K_{\psi}]$ of zeros of $\psi$ is finite.
        \item\label{item:prop_R:S} For all $x\in[0,1]$, $S(x) := \sum_{p=0}^{K_{\psi}-1} |\psi(x+p)| >0$.
    \end{enumerate}
\end{defi}

This property holds for Daubechies' wavelets up to order 45; see \cite{Rible-Seuret-1}. Since db45 $\in C^{10}(\R)$, Theorem~\ref{theo:prev_spect_inh_besov_inf_q} applies whenever the local dimensions of $\xi$ do not exceed 10. Proving that broad wavelet classes, including Daubechies', satisfy (R) is a relevant question, which we intend to address in future work. Property~(R) will be crucial for estimating the pointwise regularity of prevalent traces, especially when proving Theorem \ref{theo:low_bound_1per_wlt} and the lower bounds for pointwise Hölder exponents derived therein.

\medskip
 
Section \ref{sec:capacity_and_product_capa} provides the necessary material to work with capacities that are product of Gibbs measures, as well as the notions related to wavelet decompositions. 
Section \ref{sec:traces_in_besov} contains the proof of Theorem \ref{theo:incl_trace} and compares it to those in standard Besov spaces.
In Section \ref{sec:upper_bound_spectrum}, we prove Theorem \ref{theo:up_bound_spect_all_fct}. , the upper bounds for the singularity spectrum of functions in $\widetilde{B}^{\xi}_{\infty,q}([0,1]^D)$. In Section \ref{sec:sat_fct}, we introduce and study the regularity properties of a saturating function in $\widetilde{B}^{\xi}_{\infty,q}([0,1]^D)$. Intuitively, such a function has the largest possible wavelet coefficients in $\widetilde{B}^{\xi}_{\infty,q}([0,1]^D)$. This function is the cornerstone of the rest of the paper.
Finally, Theorem \ref{theo:prev_spect_inh_besov_inf_q} is proved in section \ref{sec:proof_prev_spect}. We construct a prevalent set on which the multifractal behavior of functions is controlled by finding suitable measures and probe sets based on the analysis of the saturating functions. One of the delicate issues consists in establishing the analyticity of the sets involved in our proofs.

\section{Capacities, wavelets and multifractal formalism}\label{sec:capacity_and_product_capa}

\subsection{Capacities and multifractal formalism}\label{subsec:capacity_and_formalism}

For $j\in \N$, recall that $\Lambda^D_j$ is the set of dyadic cube of side length $2^{-j}$ on $[0,1]^D$, i.e. the dyadic cube $\prod_{i=1}^{D} [k_i 2^{-j}, (k_i+1) 2^{-j})$ where $k=(k_1,\ldots,k_D)\in\{0,\ldots,2^j-1\}^D$. We denote $\Lambda^D=\bigcup_{j\in\N} \Lambda^D_j$ the set of dyadic cubes in $[0,1]^D$.

For $j\in\N$, $\lambda\in \Lambda^D_j$ and $N\in\N^*$, we denote $N\lambda$ the cube with the same centre as $\lambda$ and radius equal to $N\cdot 2^{-j-1}$ in $(\R^D,\|\cdot\|_{\infty})$. For instance, $3\lambda$ is the union of those $\lambda'\in \Lambda^D_j$ such that $\partial\lambda \cap \partial \lambda' \neq \emptyset$ where $\partial \lambda$ is the boundary of the cube $\lambda$.

For $x\in\R^d$, denote $\lambda^D_j(x)$ the dyadic cubes of side length $2^{-j}$ containing $x$. 

For any Borel set $A\subset\R^D$ and capacity $\nu$ with $\nu(A)>0$, we will denote $\nu_{\vert_{A}} = \nu(.\cap A)$.

In addition to the lower and limit local dimension and already defined in Definition \ref{defi:multi_frac_quantity_capacity}, one also needs the upper local dimension of set functions.

\begin{defi}\label{defi:local_dimension_iso_set_and_spectrum}
    Let $\nu\in\mathcal{H}([0,1]^D)$. For $x\in \operatorname{supp}(\nu)$, the lower and upper local dimension of $\nu$ at $x$ are respectively defined by
    \begin{equation*}
        \underline{h}_{\nu}(x)= \liminf_{r \to 0^+} \frac{\log \nu(B(x,r))}{\log(r)} \quad \text{and} \quad \overline{h}_{\nu}(x)= \limsup_{r \to 0^+} \frac{\log \nu(B(x,r))}{\log(r)}.
    \end{equation*}
    Whenever $\underline{h}_{\nu}(x) = \overline{h}_{\nu}(x)$, the common limit is called $h_{\nu}(x)$. Then, for $h\in\R$, 
    \begin{equation*}
    \begin{array}{c}
        \underline{E}_{\nu}(h)=\{x\in\operatorname{supp}(\nu) \ : \ \underline{h}_{\nu}(x)=h\}, \quad \overline{E}_{\nu}(h)=\{x\in\operatorname{supp}(\nu) \ : \ \overline{h}_{\nu}(x)=h\} \\ 
        \text{and} \quad E_{\nu}(h) = \underline{E}_{\nu}(h)\cap \overline{E}_{\nu}(\alpha).
    \end{array}
    \end{equation*}
\end{defi}

\begin{rem}\label{rem:different_local_dimension_def_ball_dyadic}
    Note that the local dimensions are sometimes defined as 
    \begin{align*}
        &\underline{h}_{\nu}(x)= \liminf_{j\to +\infty} \frac{\log_2 \nu(\lambda_j(x))}{-j} \quad \text{and} \quad \overline{h}_{\nu}(x)= \limsup_{j\to +\infty} \frac{\log_2 \nu(\lambda_j(x))}{-j}, \\
        \text{or}\quad &\underline{h}_{\nu}(x)= \liminf_{j\to +\infty} \frac{\log_2 \nu(3\lambda_j(x))}{-j} \quad \text{and} \quad \overline{h}_{\nu}(x)= \limsup_{j\to +\infty} \frac{\log_2 \nu(3\lambda_j(x))}{-j}.
    \end{align*}
    However, this paper only considers doubling capacities for which all previous notions of local dimensions, level sets, singularity spectrum and scaling function do not depend on whether dyadic cubes or centred balls are considered.
\end{rem}
Let us introduce the notation for $\varepsilon>0$, $h\in\R$ and $I=[a,b]$ an interval
\begin{equation}\label{eq:definition_nota_plus_minus_set}
    h \pm \varepsilon = [h-\varepsilon, h+\varepsilon] \quad \text{and} \quad I \pm \varepsilon = [a-\varepsilon,b+\varepsilon].
\end{equation}
\begin{defi}\label{defi:large_deviation_spectra}
    For $\nu\in\mathcal{H}([0,1]^{D})$, $I\subset \R$ and $j\in\N^{*}$, set
    \begin{equation*}
        \mathcal{D}_{\nu}(j,I) = \Big\{ \lambda \in \Lambda^D_j \ : \ \frac{\log_2 \nu(\lambda)}{-j}\in I \Big\}.
    \end{equation*}
\end{defi}

\noindent By Remark \ref{rem:different_local_dimension_def_ball_dyadic} for the almost doubling capacities, the level sets $E_{\nu}(h)$ can be rewritten as 
\begin{equation}\label{eq:proof:E_nu_of_h_is_borel}
    E_{\nu}(h) = \bigcap_{m\in\N^*} \bigcup_{J\in\N} \bigcap_{j\geq J} \bigcup_{\lambda \in \mathcal{D}_{\nu}(j,h \pm 1/m)} \lambda
\end{equation}
which are Borel sets as, for all $m\in\N^*$,
\begin{equation*}
     \mathcal{D}_{\nu}(j,h \pm 1/m) = \Big\{ \lambda \in \Lambda^D_j \ : \ \frac{\log_2 \nu(\lambda)}{-j}\in h \pm 1/m \Big\}
\end{equation*}
is a finite set of dyadic cubes at generation $j$.
 
A classical result on capacities states that for every $\nu\in\mathbb{C}([0,1]^D)$, one always has $\sigma_{\nu}(h) \leq \tau_{\nu}^*(h)$ (see \cite{Brown_Michon_Peyriere:1992:Multifractal-analysis-measures, Jaffard:2000:Frisch-Parisi-conjecture, LevyVehel-Vojak:1998:Choquet-capacities}). 
Proposition \ref{prop:prop_capa_SMF} can be deduced from \cite{Barral:2015:inverse-problems, Barral-Ben_Nasr-Peyriere:2003:multifract-formalism, Brown_Michon_Peyriere:1992:Multifractal-analysis-measures, LevyVehel-Vojak:1998:Choquet-capacities, Olsen:1995:Multifractal-formalism-measure, Barral-Seuret:2023:Besov_Space_Part_2}.

\begin{prop}\label{prop:prop_capa_SMF}
    Let $\nu\in\mathcal{C}([0,1]^D)$ with $supp(\nu) \neq \emptyset$, such that $\nu$ obeys the SMF. Then:
    \begin{enumerate}[label=$({\roman*})$]
        \item\label{item:sigma_is_tau_star} For every $h\in\R^{+}$, one has $\sigma_{\nu}(h) = \dimH\big(\underline{E}_{\nu}(h)\big) = \dimH\big(\overline{E}_{\nu}(h)\big) = \tau_{\nu}^{*}(h)$.
        \item\label{item:dim_E_leq} For every $h\leq \tau'_{\nu}(0^{+})$, $\dim_H \underline{E}^{\leq}_{\mu}(h) = \tau_{\mu}^{*}(h)$.
        \item\label{item:dim_E_geq} For every $h\geq \tau'_{\nu}(0^{-})$, $\dim_H \overline{E}^{\geq}_{\mu}(h) = \tau_{\mu}^{*}(h)$.
        \item\label{item:sigma_is_d} For every $\tau'_{\nu}(0^{+}) \leq h \leq \tau'_{\mu}(0^{-})$, $\sigma_{\mu}(h) = D$.
        \item\label{item:bound_hmin_hmax} Denote $\operatorname{dom}(\tau^{*}_{\nu})=\{\alpha\in\R~:~\tau^{*}_{\nu}(\alpha)\geq0\}$. If $\operatorname{dom}(\tau^{*}_{\nu})$ is compact, then $\operatorname{dom}(\tau^{*}_{\nu})=[h_{\nu}^{+\infty},h_{\nu}^{-\infty}]$ and there exists a positive decreasing sequence $(\varepsilon_j)_{j\geq 0}$ tending to 0 when $j$ tends to infinity, such that for all $j\in\N$ and $\lambda\in\Lambda^D_j$,
        \begin{equation}\label{eq:bound_hmin_hmax}
            h_{\nu}^{+\infty} -\eta_j \leq \frac{\log_2 \nu(\lambda)}{-j} \leq h_{\nu}^{-\infty} + \eta_j.
        \end{equation}
    \end{enumerate}
\end{prop}

\begin{rem}\label{rem:power_meas_prop}
    \begin{enumerate}[label=(\arabic*)]
        \item Direct computations show that for any $s>0$, $\nu\in\mathcal{C}([0,1]^D)$, $t\in\R$, $ \tau_{\nu^s}(t)= \tau_{\nu}(st)$, and by Legendre transform, for $h\in\R$, $\tau^*_{\nu^s}(h)= \tau^*_{\nu}\big( {h}/{s}\big).$
        \item Let $\nu \in \mathcal{C}([0,1]^D)$ satisfying \eqref{eq:P1_ball}.
        One can check that the set functions $\nu^s$, $\nu^{(+s)}$ and, when $s$ is small enough, $\nu^{(-s)}$, satisfy the same property with different exponents, and
        \begin{itemize}
            \item $\underline{E}_{\nu^s}(\alpha) = \underline{E}_{\nu}\big( {\alpha}/{s} \big)$, $  
              \overline{E}_{\nu^s}(\alpha) = \overline{E}_{\nu}\big( {\alpha}/{s} \big), $ and $ E_{\nu^s}(\alpha) = E_{\nu}\big( {\alpha}/{s} \big),$
            \item the singularity spectrum of $\nu^s$ is $h\mapsto \sigma_{\nu} \big( {h}/{s} \big)$.
        \end{itemize}
    \end{enumerate}
\end{rem}

\subsection{Gibbs measure and related capacity}

\noindent We gather classical properties satisfied by the Gibbs capacities $\mathscr{E}_D$ (see \cite{Jaerisch-Sumi:2020:multi-frac-formalism-gibbs-measures, Patzschke:1997:self-conformal-measures, Parry-pollicott:1990:zeta_function,Pesin-Weiss:1997:multi-frac-anal-gibbs-measures, Pesin:1997:dimension_theory, Seuret:2018:Inhomogeneous-random-covering-Markov-Shifts}).

\begin{propri}\label{propri:Gibbs_meas_prop}
    Let $\nu$ be a Gibbs measure. Then for every $s>0$:
    \begin{enumerate}[label=(\roman*)]
        \item The capacity $\nu^s\in\mathscr{E}_D$ is doubling \eqref{eq:doubling_induce_P1} and satisfies the SMF of Definition \ref{defi:SMF}.
        \item The dimension of $\nu^s$ \eqref{eq:def_dim_meas} equals $\tau_{\nu}'(1)=\frac{1}{s}\tau_{\nu^s}'(1/s)$. Hence, $\dim(\nu^s)=\dim(\nu)$.
        \item\label{item:gibbs_tau_analytic} The mapping $\tau_{\nu^s}$ is strictly concave, real-analytic.
        \item\label{item:gibbs_tau_star_analytic} The mapping $\tau_{\nu^s}^{*}$ has $\operatorname{dom}(\tau^{*}_{\nu^s})=[h_{\nu^s}^{+\infty},h_{\nu^s}^{-\infty}]\subset \R^{*}_{+}$ and is strictly concave, real-analytic on $(h_{\nu}^{+\infty},h_{\nu}^{-\infty})$.
        \item The capacity $\nu^s$ satisfies a quasi-Bernoulli property: there is a constant $C\geq1$ such that for every $j_1,j_2\in\N$, for every dyadic cube $\lambda_{j_1,k_1}^D \in\Lambda^D$ and $\lambda_{j_2,k_2}^D \in\Lambda^D$,  
        \begin{equation}\label{eq:quasi_bernoulli}
            \frac{1}{C}~\mu(\lambda_{j_1,k_1}^D)~\mu(\lambda_{j_2,k_2}^D) \leq \mu(\lambda_{j_1+j_2,(2^{j_2}k_1+k_2)}^D) \leq C~\mu(\lambda_{j_1,k_1}^D)~\mu(\lambda_{j_2,k_2}^D)
        \end{equation}
    \end{enumerate}
\end{propri}
In \eqref{eq:quasi_bernoulli}, the cube $\lambda_{j_1+j_2,(2^{j_2}k_1+k_2)}^D$ is the image of $\lambda_{j_2,k_2}^D$ by the canonical affine mapping that sends $[0,1]^D$ to $\lambda_{j_1,k_1}^D$. This quasi-Bernoulli property means that $\nu^s$ is almost multiplicative.
  
\begin{figure}
    \centering
    \includegraphics[scale=0.8]{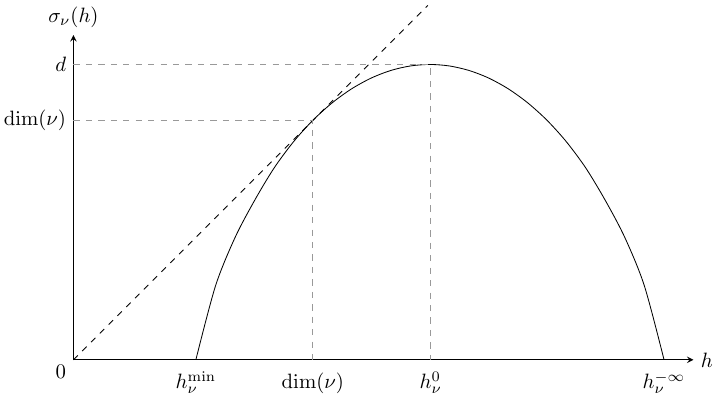}\vspace{-2mm}
    \caption{Spectrum $\sigma_{\nu}$ for a Gibbs measure $\nu$.}
    \label{fig:spectrum_nu}
\end{figure}

Figure \ref{fig:spectrum_nu} illustrates the typical shape of $\sigma_{\nu}$ when $\nu$ is a Gibbs measure. 
Observe that the smallest local dimension of $\nu$ is $h_{\nu}^{\min} = h^{+\infty}_{\nu}$ and the largest one is $h_{\nu}^{\max} = h^{-\infty}_{\nu}$.
 Using item \ref{item:gibbs_tau_analytic}, one has that $h^{0}_{\nu}=\tau_{\nu}'(0^+)=\tau_{\nu}'(0^-)$ is the unique maximum of the spectrum.

As explained earlier, auxiliary measures of Gibbs measures play a key role in our analysis. Let us regroup some important results about them, see \cite{Frisch-Parisi:1985:Multifractal-turbulence, Halsey:1987:fractal-measure-singularities, Brown_Michon_Peyriere:1992:Multifractal-analysis-measures, Olsen:1995:Multifractal-formalism-measure, Heurteaux:2007:Dimension-measure}.
    
\begin{theo}[Auxiliary measure]\label{theo:auxiliary_measure}
    Let $\nu$ be a Gibbs measure associated to a potential $\varphi$ on $\R^D$ and $r\in\R$. There exists a probability measure $\nu_r$, called \textnormal{auxiliary measure}, satisfying for a constant $C\geq1$ and for all cubes $I \in \Lambda^D$ 
    \begin{equation}\label{eq:theo:axiliary_measure}
        \frac{1}{C}\ \nu(I)^r |I|^{-\tau_{\nu}(r)} \leq \nu_r(I) \leq C\ \nu(I)^r |I|^{-\tau_{\nu}(r)}.
    \end{equation}
    The measure $\nu_r$ has the following properties (recall the definition \eqref{eq:h_nu_r_def} of $h^r_\nu$):  
        \begin{enumerate}[label=(\roman*)]
        \item\label{theo:auxiliary_measure:nu_r_gibbs_smf} $\nu_r$ is a Gibbs measure associated to the potential $r\varphi$ and satisfies the SMF,
        \item\label{theo:auxiliary_measure:nu_r_measure_1} One has $h_{\nu_r} = \sigma_\nu(r^r_\nu)$ and $\nu_r\big( E_{\nu_r}(h_{\nu_r} ) \big)= \nu(E_{\nu}(h_\nu^r))=1$, 
        \item\label{theo:auxiliary_measure:dim_nu_r} $\sigma_{\nu}(h^{r}_{\nu})=\dim(\nu_r)=rh^{r}_{\nu}-\tau_{\nu}(r) \leq h^{r}_{\nu}$.
    \end{enumerate}
\end{theo}

\begin{rem}
    If we consider a capacity $\nu$ satisfying the SMF, one sees from the definition that $E_{\nu}(h^{r}_{\nu}) = E_{\nu_r}(\sigma_{\nu}(h^{r}_{\nu}))$. Indeed, using \eqref{eq:theo:axiliary_measure}, 
    \begin{equation*}
        h_{\nu_r}(x) = \lim_{j\to+\infty} \frac{\log_2 \nu_r(\lambda_j(x))}{-j} = r \ \lim_{j\to+\infty}\frac{\log_2 \nu(\lambda_j(x))}{-j} - \tau_{\nu}(r) = rh_{\nu}^r-\tau_\nu(r) = \sigma_{\nu}(h^r_{\nu}).
    \end{equation*}
\end{rem}

\begin{rem}
    Consider a Gibbs measure $\nu$ and $\mu:=\nu^s \in \mathscr{E}_d$ with $s>0$. Then using Remark \ref{rem:power_meas_prop}, the \textnormal{auxiliary measure} $\mu_r=\nu_{rs}$ satisfying \eqref{eq:theo:axiliary_measure} and the properties of Theorem \ref{theo:auxiliary_measure}:
    \begin{enumerate}[label=(\roman*)]
        \item $\mu_r=\nu_{rs}$ is a Gibbs measure and satisfies the SMF,
        \item $\mu_r\big( E_{\mu_r}(\sigma_{\mu}(\tau_{\mu}'(r))) \big)=\nu_{rs}\big( E_{\nu_{rs}}(\sigma_{\nu}(\tau_{\nu}'(rs))) \big)=1$,
        \item $\dim(\mu_r)=\dim(\nu_{rs})=\sigma_{\nu}\big( \tau_{\nu}'(rs) \big)=\sigma_{\mu} \big(\tau_{\mu}'(r) \big)$.
    \end{enumerate}
\end{rem}

The following result introduces quantities important for what follows. It states that, up to a restriction onto some set $A^{\nu_r}_{n,m}$ of large $\nu_r$-measure, the $\nu$-mass and $\nu_r$-mass of the dyadic cubes intersecting the neighborhood of $A^{\nu_r}_{n,m}$ can be very precisely controlled.
\begin{prop}\label{prop:bound_meas_dyadics}
    Let $\nu\in \mathscr{E}_{D}$, $n,m\in\N^*$ and $r\in\R$. 
    For any fixed constant $K \in\N^*$, there exist an integer $J^{\nu_r}_{n,m}\in\N$ and a set $A^{\nu_r}_{n,m}$ satisfying the following properties :
    \begin{enumerate}
        \item $\nu_r(A^{\nu_r}_{n,m}) \geq 1- {1}/{n}$,
        \item for every $a \in A^{\nu_r}_{n,m}$, for every $j \geq J^{\nu_r}_{n,m}$ and for $\lambda^{D}\in\Lambda_j^{D}$ such that $a \in K \lambda^{D}$,
        \begin{align}
            2^{-j(h^{r}_{\nu}+1/m)} \leq \nu&\big(\lambda^{D} \big) \leq 2^{-j(h^{r}_{\nu}-1/m)}, \label{eq:bound_meas_alm_all_dyad} \\
            2^{-j(\dim(\nu_r)+1/m)} \leq \nu_r&\big(\lambda^{D} \big) \leq 2^{-j(\dim(\nu_r)-1/m)}. \label{eq:bound_auxmeas_alm_all_dyad}
        \end{align}
    \end{enumerate}
\end{prop}
\begin{proof} 
   Consider $\nu\in\mathscr{E}_{D}$ and its associated auxiliary measure $\nu_r$. Using Definition \ref{defi:large_deviation_spectra}, for all $\widetilde{m}\in\N^*$, $\mathcal{D}_{\nu}(j,h \pm 1/\widetilde{m}) = \Big\{ \lambda\subset [0,1]^{D}, \lambda \in \Lambda^D_j : \frac{\log_2 \nu(\lambda)}{-j}\in h \pm 1/\widetilde{m} \Big\}$ is a finite set of dyadic cubes at generation $j$. Applying \eqref{eq:proof:E_nu_of_h_is_borel}, $\nu$ and $\nu_r$ being doubling, one has
    \begin{equation*}
        E_{\nu}(h) = \bigcap_{\widetilde{m}\in\N^*} \bigcup_{J\in\N} \bigcap_{j\geq J} \bigcup_{\lambda \in \mathcal{D}_{\nu}(j,h \pm 1/\widetilde{m})} \lambda.
    \end{equation*}
    and from Theorem \ref{theo:auxiliary_measure}, $\nu_r\big( E_{\nu_r}(\sigma_{\nu}(h^{r}_{\nu})) \big)=\nu_r\big( E_{\nu}(h^{r}_{\nu}) \big)=1$. 
    One deduces that for every $\widetilde{m}\in\N^*$, 
    \begin{equation*}
        \nu_r\Big ( \bigcup_{J\in\N} \bigcap_{j\geq J} \bigcup_{\lambda \in \mathcal{D}_{\nu}(j,h^{r}_{\nu} \pm 1/\widetilde{m})} \lambda \Big)=1.
    \end{equation*}
    The fact that $\big (\bigcap_{j\geq J} \bigcup_{\lambda \in \mathcal{D}_{\nu}(j,h^{r}_{\nu} \pm 1/\widetilde{m})} \lambda \big)_{J\geq 1}$ is an increasing sequence of Borel sets implies that $\lim_{J\to+\infty} \nu_r\big(\bigcap_{j\geq J} \bigcup_{\lambda \in \mathcal{D}_{\nu}(j,h^{r}_{\nu} \pm 1/\widetilde{m})} \lambda \big)=1$, so there is $J_{n,\widetilde{m}}\in\N$ such that the Borel set 
    \begin{equation}\label{eq:proof:def_A_n_m_nu_r}
        A^{\nu_r}_{n,\widetilde{m}}:= \bigcap_{j\geq J_{n,\widetilde{m}}} \bigcup_{\lambda \in \mathcal{D}_{\nu}(j,h^{r}_{\nu} \pm 1/\widetilde{m})} \lambda
    \end{equation}
    satisfies $\nu_r(A^{\nu_r}_{n,\widetilde{m}}) \geq 1- {1}/{n}$. Using item \ref{theo:auxiliary_measure:dim_nu_r} of Theorem \ref{theo:auxiliary_measure}, $\dim(\nu_r) = \sigma_{\nu}(h^{r}_{\nu}) \leq h^{r}_{\nu}$. 
    From \eqref{eq:proof:def_A_n_m_nu_r}, for $n, \widetilde{m} \in\N^*$, there is an integer $J_{n,\widetilde{m}}$ such that for every $a\in A^{\nu_r}_{n,\widetilde{m}}$ and $j\geq J_{n,\widetilde{m}}$,
    \begin{equation}\label{eq:bound_meas_proof} 
        2^{-j(h^{r}_{\nu}+1/\widetilde{m})} \leq \nu\big(\lambda_j^{D}(a) \big) \leq 2^{-j(h^{r}_{\nu}-1/\widetilde{m})}
    \end{equation}
    where we recall that $\lambda_j^{D}(a)$ is the dyadic cube $\lambda_j^{D}\in\Lambda_j^{D}$ such that $a \in \lambda_j^{D}$.
    
 Fix $K \in\N^*$.
   For $\lambda_j^{D}\in\Lambda_j^{D}$ such that $a \in K \lambda_j^{D}$, lcall$x_{j,a}$ the centre of $\lambda_j^{D}(a)\in \Lambda_j^{D}$ and $\widetilde{x}$ the center of $\lambda_j^{D}$. From the definition of $K \lambda^{D}$, we have that $\|a-\widetilde{x}\|_{\infty} \leq 2^{-j-1+\log_2(K)}$, so
    \begin{align*}
        \lambda_j^{D} \subset B\big(x_{j,a},2^{-j+\partent{\log_2(K)}+1} \big) \ \mbox{ and } \ 
        \lambda_j^{D}(a) \subset B\big(\widetilde{x},2^{-j+\partent{\log_2(K)}+1} \big).
    \end{align*}
    Using the doubling property \eqref{eq:doub_meas_ball} of $\nu$, one gets
    \begin{equation*}
        \frac{1}{C_{\nu}^{\partent{\log_2(K)}+1}} \nu(\lambda_j^{D}(a)) \leq \nu(\lambda_j^{D}) \leq C_{\nu}^{\partent{\log_2(K)}+1} \nu(\lambda_j^{D}(a))
    \end{equation*}
    and the same result holds for $\nu_r$ as it is also doubling. 
    Hence, one has from \eqref{eq:bound_meas_proof}
    \begin{equation*}
        2^{-j(h^{r}_{\nu}+1/\widetilde{m_j})} \leq \nu\big(\lambda_j^{D} \big) \leq 2^{-j(h^{r}_{\nu}-1/\widetilde{m_j})}
    \end{equation*}
    with $ \frac{1}{\widetilde{m_j}} = \frac{1}{\widetilde{m}}+\frac{(\partent{\log_2(K)}+2)~\log_2(C_{\nu})}{j}$. Moreover, by \eqref{eq:theo:axiliary_measure}, for every $r\in\R$, there is $C_r\geq1$ such that
    \begin{equation*}
        \frac{1}{C_r}~\nu\big(\lambda_j^{D} \big)^r\big|\lambda_j^{D}\big|^{-\tau_{\nu}(r)} \leq~\nu_r\big(\lambda_j^{D}\big) \leq C_r~\nu\big(\lambda_j^{D} \big)^r\big|\lambda_j^{D}\big|^{-\tau_{\nu}(r)}.
    \end{equation*}
    One deduces that, for every $\lambda_j^{D}\in\Lambda_j^{D}$ such that $a \in K \lambda_j^{D}$,
    \begin{align*}
        2^{-j (rh^{r}_{\nu}-\tau_{\nu}(r) + 1/\widetilde{m_j}+\frac{\log_2(C_r)}{j} )} \leq~&\nu_r (\lambda_j^{D} ) \leq 2^{-j (rh^{r}_{\nu}-\tau_{\nu}(r)-1/\widetilde{m_j}-\frac{\log_2(C_r)}{j} )}\\
        2^{-j (\dim(\nu_r) + 1/\widetilde{m_j}+\frac{\log_2(C_r)}{j} )} \leq~&\nu_r (\lambda_j^{D} ) \leq 2^{-j (\dim(\nu_r) - 1/\widetilde{m_j}-\frac{\log_2(C_r)}{j} )}. \numberthis \label{eq:bound_auxmeas_proof}
    \end{align*}
    One concludes that, for $ \frac{1}{m_j} = \frac{1}{\widetilde{m}}+\frac{(\partent{\log_2(K)}+2)~\log_2(C_{\nu})}{j} + \frac{\log_2(C_r)}{j}$, by \eqref{eq:bound_meas_proof} and \eqref{eq:bound_auxmeas_proof},
    \begin{align*}
        2^{-j(h^{r}_{\nu}+1/m_j)} \leq~\nu&\big(\lambda_j^{D} \big) \leq 2^{-j(h^{r}_{\nu}-1/m_j)} \\
        2^{-j \big(\dim(\nu_r) + 1/m_j \big)} \leq~\nu_r&\big(\lambda_j^{D} \big) \leq 2^{-j \big(\dim(\nu_r) - 1/m_j \big)}.
    \end{align*}
    For $m\in\N^*$, there exists an integer $J_m\in\N$ large enough such that for all $j\geq J_m$, 
    \begin{equation*}
        \frac{(\partent{\log_2(K)}+2)~\log_2(C_{\nu})}{j} + \frac{\log_2(C_r)}{j} \leq \frac{1}{2m}
    \end{equation*}
    and considering $J^{\nu_r}_{n,m} := \max(J_{n,2m},J_m)$ (with $\widetilde{m}=2m$), one sees that $1/\widetilde{m_j}\leq 1/m$ and $1/m_j \leq 1/m$, inducing the second item.
\end{proof}
 
Proposition \ref{prop:bound_meas_dyadics} reads as: if $K \lambda^{d'}\cap A^{\nu_r}_{n,m} \neq \emptyset$, then \eqref{eq:bound_meas_alm_all_dyad} and \eqref{eq:bound_auxmeas_alm_all_dyad} hold true.

Also, $m$ being fixed, one sees that $(A^{\nu_r}_{n,m})_{n\in\N^*}$ is an increasing sequence.

\subsection{Product capacity on \texorpdfstring{$[0,1]^D$}{[0,1] hat D}}\label{sec:product_capa}

Fix $D= d+d'$. Let $\varphi: \R^d \to \R$ and $\widetilde{\varphi}: \R^{d'} \to \R$ be respectively a $\Z^d$-invariant and $\Z^{d'}$-invariant real valued H{\"o}lder function. Then $\varphi + \widetilde{\varphi}: (x^d,x^{d'}) \in \R^d \times \R^{d'} \mapsto \varphi(x^d)+\widetilde{\varphi}(x^{d'}) \in\R$ is a $\Z^D$-invariant real valued H{\"o}lder function. Next lemma is immediate.
\begin{lemme}
    Let $\mu$ and $\nu$ be Gibbs measures respectively on $[0,1]^d$ associated to $\varphi$ and on $[0,1]^{d'}$ associated to $\widetilde{\varphi}$ as above. The product measure $\xi=\mu \times \nu$ is a Gibbs measure on $[0,1]^{D}$ associated to the H{\"o}lder continuous potential $\varphi + \widetilde{\varphi}$. 
\end{lemme}

The H{\"o}lder index of $\xi$ is deduced from that of $\mu$ and $\nu$.

\begin{lemme}\label{lemma:B_xi_s-holder}
    Let $\mu\in\mathscr{E}_{d}$ on $[0,1]^d$ be $s_{\mu}$-H{\"o}lder and $\nu\in\mathscr{E}_{d'}$ on $[0,1]^{d'}$ be $s_{\nu}$-H{\"o}lder with $s_{\mu}+s_{\nu} > 0$.
    Then $\xi=\mu \times \nu$ is $(s_{\mu}+s_{\nu})$-H{\"o}lder.
\end{lemme}
\begin{proof}
    Since, $\mu$ is $s_{\mu}$-H{\"o}lder, there is $C_{\mu}>0$ such that $\mu(E^{d})\leq C_{\mu} |E^{d}|^{s_{\mu}}$ for all $E^{d}\in\mathcal{B}([0,1]^d)$. Similarly, $\nu$ is $s_{\nu}$-H{\"o}lder, and $\nu(E^{d'})\leq C_{\nu} |E^{d'}|^{s_{\nu}}$ for all $E^{d'}\in\mathcal{B}([0,1]^{d'})$ and some $C_\nu>0$.

    Let $E^{D}\in\mathcal{B}([0,1]^{D})$. There exist a $x\in\R^d$ and $x'\in\R^{d'}$ such that $E^{D} \subset B_{d}(x,|E^{D}|) \times B_{d'}(x',|E^{D}|)$.
    One obtains that for some constant $C>0$, 
    \begin{align*}
        \xi(E^{D}) 
        &\leq \xi( B_{d}(x,|E^{D}|) \times B_{d'}(x',|E^{D}|)) 
        = \mu(B_{d}(x,|E^{D}|)) \cdot \nu(B_{d'}(x',|E^{D}|))\\
        &\leq C~|E^{D}|^{s_{\mu}} \cdot |E^{D}|^{s_{\nu}},
    \end{align*}
    proving that $\xi$ is $(s_{\mu}+s_{\nu})$-H{\"o}lder.
\end{proof}

From the previous lemma and Remark \ref{rem:incl_inh_besov}, one sees that, for $s:=s_{\mu}+s_{\nu}$,
\begin{equation}\label{eq:embedding_besov_xi_product_measure}
    B^{\xi}_{p,q}([0,1]^D) \hookrightarrow B^{s+\frac{D}{p}}_{p,q}([0,1]^D) \hookrightarrow B^{s+\frac{D}{p}-\frac{D}{p}}_{\infty,q}([0,1]^D) = C^{s}([0,1]^D).
\end{equation}

\begin{lemme}
    Let $\mu$ and $\nu$ be doubling capacities in $\mathcal{C}([0,1]^{d})$ and $\mathcal{C}([0,1]^{d'})$, respectively.
    The product capacity $\xi=\mu \times \nu$ is doubling.
\end{lemme}
\begin{proof}
    Let us prove \eqref{eq:doub_meas_ball} for $\xi$. For $x^D=(x^d,x^{d'})\in (0,1)^d \times (0,1)^{d'} = (0,1)^D$, for all $r>0$,
    \begin{equation*}
         B_{D}(x^D,2 r) \subset B_{d}(x^d,2 r) \times B_{d'}(x^{d'},2 r) ~~\text{and}~~ B_{d}(x^d,r/2) \times B_{d'}(x^{d'},r/2) \subset B_{D}(x^D,r).
    \end{equation*}
    So, using \eqref{eq:doub_meas_ball} for $\mu$ and $\nu$ (with the constant $C_{\mu}$ and $C_{\nu}$ respectively),
    \begin{align*}
        \xi(B_{D}(x^D,2 r)) 
        &\leq \mu(B_{d}(x^d,2 r)) \cdot \nu(B_{d'}(x^{d'},2 r))\\
        &\leq C_{\mu}^2 C_{\nu}^2~\mu(B_{d}(x^d,r/2)) \cdot \nu(B_{d'}(x^{d'},r/2)),
    \end{align*}
    hence $\xi \big(B_{D}(x^D,2 r)\big) \leq C_{\xi}~\xi (B_{D}(x^D,r) )$ and the lemma is proved.
\end{proof}

\begin{prop}
    Let $\mu\in\mathscr{E}_{d}$ and $\nu\in\mathscr{E}_{d'}$.
    The capacity $\xi=\mu \times \nu$ verifies the SMF and
    \begin{equation}\label{eq:spect_xi}
        \text{for all}~r\in\R, \quad \sigma_{\xi} \big(h_{\mu}^{r} + h_{\nu}^{r} \big)= \sigma_{\mu} \big(h_{\mu}^{r}\big) + \sigma_{\nu} \big(h_{\nu}^{r} \big).
    \end{equation}
\end{prop}
\begin{proof}
    We begin with the lower bound $\sigma_\xi\big(h_{\mu}^{r} + h_{\nu}^{r} \big)\geq \sigma_{\mu} \big(h_{\mu}^{r}\big) + \sigma_{\nu} \big(h_{\nu}^{r}\big)$. Observe that for every $0\leq \alpha \leq h$, one has $E_{\mu}(\alpha) \times E_{\nu}(h-\alpha) \subset E_{\xi}(h) \subset \underline{E}_{\xi}(h)$. Indeed, for $x^D=(x^d,x^{d'})\in[0,1]^D$, one has $h_{\mu}(x^d) + h_{\nu}(x^{d'}) = h_{\xi}(x^D)$ using that
    \begin{equation*}
        \lim_{j\to +\infty} \frac{\log \mu(\lambda^d_j(x^d))}{\log2^{-j}} + \frac{\log \nu(\lambda^{d'}_j(x^{d'}))}{\log 2^{-j}} 
        = \lim_{j\to +\infty} \frac{\log \xi(\lambda^D_j(x^D))}{\log2^{-j}} 
    \end{equation*}
    since all the limits exist. Hence $\bigcup_{0 \leq \alpha \leq h} E_{\mu}(\alpha) \times E_{\nu}(h-\alpha) \subset E_{\xi}(h)$, from which we deduce
    \begin{align*}
        \sigma_{\xi}(h) 
        &= \dimH\big(\underline{E}_{\xi}(h)\big)
        \geq \max_{0 \leq \alpha \leq h} \dimH\big(E_{\mu}(\alpha) \times E_{\nu}(h-\alpha)\big) \\
        &\geq \max_{0 \leq \alpha \leq h} \dimH\big(E_{\mu}(\alpha)\big) + \dimH\big(E_{\nu}(h-\alpha)\big) = \max_{0 \leq \alpha \leq h} \sigma_{\mu}(\alpha) + \sigma_{\nu}(h-\alpha),
    \end{align*}
    the last equality coming from the fact that $\mu$ and $\nu$ satisfy the SMF. Taking $h=h_{\mu}^{r} + h_{\nu}^{r}$ for some $r\in\R$ and $\alpha = h_{\mu}^{r}$, we get $\sigma_{\xi} \big(h_{\mu}^{r} + h_{\nu}^{r} \big) \geq \sigma_{\mu} \big(h_{\mu}^{r}\big) + \sigma_{\nu} \big(h_{\nu}^{r} \big)$.
    
    Observe also that when $r$ ranges in $[-\infty,+\infty]$, $h_{\mu}^{r} + h_{\nu}^{r}$ covers the whole interval $[h_{\mu}^{+\infty}+h_{\nu}^{+\infty}, h_{\mu}^{-\infty}+h_{\nu}^{+\infty}]$, which contains all possible values of the local dimensions of $\xi$, hence formula \eqref{eq:spect_xi} describes entirely the values of $\sigma_\xi$.
    
    We finish by the converse inequality. It is known that for every $h\in \R$, $\sigma_{\xi}(h) \leq \tau_\xi^*(h)$. Since $\mu \in \mathscr{E}_{d}$ and $\nu \in \mathscr{E}_{d'}$ satisfy \eqref{eq:quasi_bernoulli}, it is standard that the scaling function of $\mu$ and $\nu$ is a limit by Fekete's Subadditive Lemma, as used in \cite{Brown_Michon_Peyriere:1992:Multifractal-analysis-measures}.
    As a Cartesian product of such objects, $\xi$ is quasi-Bernoulli as well, and $\tau_\xi(q)=\lim_{j\to+\infty} \frac{1}{-j} \log_2 \sum_{I\in\Lambda^D _j} \xi(I)^q$.
    One then computes that
    \begin{align*}
        \tau_\xi(q)
        &= \lim_{j\to+\infty} \frac{\log \sum_{\lambda^d \times \lambda^{d'}\in\Lambda^{D}_j} \xi(\lambda^d \times \lambda^{d'})^q }{\log 2^{-j}} \\
        &= \lim_{j\to+\infty} \frac{\log \sum_{\lambda^d \in\Lambda^{d}_j} \mu(\lambda^d)^q \sum_{\lambda^{d'} \in\Lambda^{d'}_j}\nu(\lambda^{d'})^q }{\log 2^{-j}} \\
        &= \lim_{j\to+\infty} \frac{\log \sum_{\lambda^d \in\Lambda^{d}_j} \mu(\lambda^d)^q }{\log(2^{-j})}+ \lim_{j\to+\infty} \frac{\log \sum_{\lambda^{d'} \in\Lambda^{d'}_j}\nu(\lambda^{d'})^q }{\log 2^{-j}} = \tau_\mu(q) + \tau_\nu(q).
    \end{align*}
     
    Observe then that, for a fixed $r\in\R$, $\tau_{\xi}'(r) = \tau'_\mu(r) + \tau'_\nu(r) = h_{\mu}^{r} + h_{\nu}^{r}$. Also, one knows by the properties of the Legendre transform of concave functions that the infinimum $\inf_{t\in\R} (t \tau_{\xi}'(r) - \tau_{\xi}(t))$ is reached at $t=r$, hence, using the SMF for $\mu$ and $\nu$,
    \begin{align*}
        \tau_{\xi}^{*}(\tau_{\xi}'(r)) 
        &= r \tau_{\xi}'(r) - \tau_{\xi}(r) 
        = r h_{\mu}^{r} - \tau_{\mu}(r) + h_{\nu}^{r} - \tau_{\nu}(r) 
        = \tau_{\mu}^{*} (h_{\mu}^{r} ) + \tau_{\nu}^{*}\big(h_{\nu}^{r}\big) \\
        &= \sigma_{\mu} (h_{\mu}^{r} ) + \sigma_{\nu}\big(h_{\nu}\big).
    \end{align*}
    One deduces that $\sigma_{\xi} (h_{\mu}^{r} + h_{\nu}^{r} ) \leq \tau_{\xi}^* (h_{\mu}^{r} + h_{\nu}^{r} ) = \sigma_{\mu} (h_{\mu}^{r} ) + \sigma_{\nu} (h_{\nu} )$, hence the result.
\end{proof}
    
\subsection{Multi-resolution wavelet analysis}\label{sec:mr_wlt_ana}

Let $D= d+d'\in\N$ and $1\leq p, q \leq +\infty$. A mapping $\psi:\R^D\to \R$ has $n+1$ vanishing moments when for every multi-index $\alpha\in\N^D$ such that $\sum_{i=1}^{D} \alpha_i \leq n$, $\displaystyle \int_{\R^D} x_1^{\alpha_1}\ldots x_D^{\alpha_D} \psi(x)dx=0$. 

For an arbitrary integer $N\geq 1$, one can construct compactly supported functions $\psi^{0} \in C^{N}(\R)$ (called the scaling function) and $\psi^{1} \in C^{N}(\R)$ (called the mother wavelet), with $\psi^{1}$ having at least $N+1$ vanishing moments (see \cite{Daubechies:1988:compactly_supported_wavelets, Daubechies-Lagarias:1991:regularity_multiresolution_analysis}) and such that the set of function $\psi^{1}_{j,k}:x\mapsto \psi^{1}(2^j x-k)$ for $j\in\Z$ and $k\in\Z$ form an orthogonal basis of $L^2(\R)$. In this case, the wavelet is said to be $N$-regular.
We set $0^D := (0,\ldots,0)$, $1^D := (1,\ldots,1)$, $L^D = \{0,1\}^D\backslash 0^D$.

\noindent 
An orthogonal basis of $L^2(\R^D)$ is obtained by tensorisation as follows (see \cite[Chapters 2-3]{Meyer:1990:Ondelette_operateur} for a general construction):
for $\lambda^D = (j,k,l)\in \Z \times \Z^D \times \{0,1\}^D$, define the tensorised wavelet $\psi_{\lambda^D} (x_1,...,x_D)= \prod_{i=1}^{D} \psi_{j,k_i}^{l_i}(x_i)$ with $k=(k_1,\ldots,k_D)$ and $l=(l_1,\ldots,l_D)\in\{0,1\}^D$.

Similar notations are used with $d$ and $d'$ (e.g. $L^d = \{0,1\}^d\backslash 0^d$).

For $\lambda^d = (j,k,l) \in \Lambda^{d}_j\times \{0,1\}^{d}$ and $\lambda^{d'} = (j,k',l') \in \Lambda^{d'}_j\times \{0,1\}^{d'}$, we will write $\lambda^D = \lambda^d\times\lambda^{d'}$ for $\lambda^D = (j,(k,k'),(l,l')) \in \Lambda^{D}_j\times \{0,1\}^{D}$. We will also alternatively use $c_{\lambda^{d},\lambda^{d'}} = c_{\lambda^{D}}$ depending on the need to make the notation simpler.

Thus, with these notations, and recalling \eqref{eq:wlt_dec_on_01}, we will consider functions written as
\begin{equation*}
    f = c_{0,0^D,0^D} \psi_{0,0^D,0^D}~+ \sum_{\lambda^D \in \Lambda^D \times L^D} c_{\lambda^D} \psi_{\lambda^D} 
    = c_{0,0^D,0^D} \psi_{0,0^D,0^D}~+ \sum_{j\geq 0} \sum_{\lambda^D_j \in \Lambda^D_j \times L^D} c_{\lambda^D_j} \psi_{\lambda^D_j}
\end{equation*}
 
\begin{prop}\label{prop:equi_hold_space_wlt_coef} \cite{Jaffard:1989:Holder_exponent_wavelet_coef,Jaffard:2007:Wavelet_leaders}
    Suppose that $\gamma>0$ and $N=\partent{\gamma+1}$. Consider a wavelet $\psi \in C^{N}(\R)$ having at least $N+1$-regular. Let $f:[0,1]^D \to \R$ be a locally bounded function with wavelet coefficients $\{c_{\lambda}\}_{\lambda\in\Lambda^D\times L^D}$, and let $x\in[0,1]^D$.

    If $f\in C^{\gamma}(x)$, then there exists a constant $>0$ such that for all $\lambda=(j,k,l)\in\Lambda^D\times L^D$,
    \begin{equation}\label{eq:equi_hold_space_wlt_coef}
        |c_{\lambda}| \leq M\ 2^{-j\gamma} \big( 1+|2^j x-k| \big)^{\gamma}.
    \end{equation}
    Conversely, if \eqref{eq:equi_hold_space_wlt_coef} holds true and if $f\in \bigcup_{\varepsilon>0} C^{\varepsilon}([0,1]^D)$, then $f\in C^{\gamma-\eta}(x)$, $\forall$ $\eta>0$.
\end{prop}

\begin{rem}\label{rem:incl_inh_besov}\cite[Remark 2.17]{Barral-Seuret:2023:Besov_Space_Part_2} 
    Let $\xi\in\mathcal{C}([0,1]^D)$ be $s$-H{\"o}lder with $s>0$, the embedding $B^{\xi}_{p,q}([0,1]^D) \hookrightarrow B^{s+\frac{D}{p}}_{p,q}([0,1]^D)$ holds.
    When $\xi$ satisfies \eqref{eq:P1_ball} for $C>0$ and $0<s_1<s_2$, one has $B^{s_2+\frac{D}{p}}_{p,q}([0,1]^D) \hookrightarrow B^{\xi}_{p,q}(\R^D) \hookrightarrow B^{s_1+\frac{D}{p}}_{p,q}([0,1]^D)$.
    Also, when $\alpha \geq \beta$, $B^{\xi^{(+\alpha)}}_{p,q}([0,1]^D) \hookrightarrow B^{\xi^{(+\beta)}}_{p,q}([0,1]^D)$ and $\widetilde{B}^{\xi^{(+\alpha)}}_{p,q}([0,1]^D) \hookrightarrow \widetilde{B}^{\xi^{(+\beta)}}_{p,q}([0,1]^D)$.
\end{rem}

\section{Traces and inhomogeneous Besov spaces}\label{sec:traces_in_besov}

\subsection{Wavelet characterisation of traces}
Observe that, by Lemma \ref{lemma:B_xi_s-holder}, every function in this paper is continuous, which implies that the trace $f_a$ in \eqref{eq:def_trace} is well-defined for every $a\in \R^{d'}$. Observe that for every $a\in\R^{d'}$, the trace $f_{a}:x\in\R^d \mapsto f(x,a) \in \R$ can be explicitly written in terms of wavelet expansion \eqref{eq:wlt_dec_on_01} as
\begin{equation}\label{eq:dec_trace}
    f_{a}(x) = c_{0,0^D,0^D} \psi_{0,0^d,0^d}(x)\psi_{0,0^{d'},0^{d'}}(a) ~+ \sum_{\lambda^{D}:=\lambda^{d} \times \lambda^{d'}\in \Lambda^D\times L^D} c_{\lambda^D} \psi_{\lambda^d}(x) \psi_{\lambda^{d'}}(a).
\end{equation}
We then decompose $f_a$ as $f_{a} = G_{a} + F_a$ where $G_{a}(x) = \sum_{\lambda^d\in \Lambda^d\times 0^d} d^G_{\lambda^d}(a) \psi_{\lambda^d}(x)$ with

\begin{equation}\label{eq:dec_trace_G_a}
    d^G_{\lambda_j^d}(a) =
    \begin{cases}
        \displaystyle \sum_{\lambda^{d'}\in\Lambda^{d'}_0 \times \{0,1\}^{d'}} c_{\lambda^d,\lambda^{d'}} \psi_{\lambda^{d'}}(a), & \text{for } j=0,\\
        \displaystyle \sum_{\lambda^{d'}\in\Lambda^{d'}_j\times \mathbf{L^{d'}}} c_{\lambda^d,\lambda^{d'}} \psi_{\lambda^{d'}}(a), &\text{for } j\geq 1,
    \end{cases}
\end{equation}
and $F_a(x) = \sum_{j\geq 0} \sum_{\lambda_j^d\in \Lambda_j^d\times L^d} d^F_{\lambda_j ^d}(a) \psi_{\lambda_j^d}(x)$ with
\begin{equation}\label{eq:dec_trace_F_a}
    d^F_{\lambda_j^d}(a) = \sum_{ {\lambda^{d'}\in\Lambda^{d'}_j\times \mathbf{\{0,1\}^{d'}} }} c_{\lambda_j^d,\lambda_j ^{d'}} \psi_{\lambda^{d'}_j}(a).
\end{equation}

Formula \eqref{eq:dec_trace_F_a} yields a wavelet decomposition of $F_a$ as the wavelets $\big(\psi_{\lambda^d}\big)_{\lambda^d\in \Lambda^d\times L^d}$ form a wavelet basis of $L^2([0,1]^d)$ when completed by the scaling function $\psi_{0,0^d,0^d}$. 
On the contrary, the decomposition $G_a$ in \eqref{eq:dec_trace_G_a} only involves the scaling functions $\psi_{j,k,0^d}$ and not wavelets.

So, to find out whether $f_a$ belongs to some inhomogeneous space $B^\mu_{p,q}([0,1]^d)$, one cannot directly use the definition \eqref{eq:snorm_inh_besov}, which involves only wavelet coefficients computed with $\psi_{\lambda^{d}}(a)$ for $\lambda^d\in \Lambda^d\times L^d$.
Fortunately, Jaffard proved the following result for standard Besov space:
\begin{prop} \label{prop:jaff_trace}\textnormal{\cite[Proposition 1]{Jaffard:1995:traces_negative_dim}}
    Let $s>0$ and a function $ g = \sum_{\lambda^d\in \Lambda^d\times 0^d} d_{\lambda^d} \psi_{\lambda^d}$.
    If $\sup_{j\geq 1} 2^{j(ps-d)} \sum_{\lambda^d\in \Lambda_j^d\times 0^d} \big|d_{\lambda^d}\big|^p < + \infty$, then $g\in B^{s}_{p,\infty}([0,1]^d)$.
\end{prop} 
Proposition \ref{prop:jaff_trace} is extended to inhomogeneous Besov spaces.  
\begin{prop}\label{prop:G_a_belonging_besov}
    Let $1\leq p,q \leq \infty$, $\kappa\in \mathcal{C}([0,1]^d)$ a doubling capacity and a function $g = \sum_{\lambda^d\in \Lambda^d\times 0^d} d_{\lambda^d} \psi_{\lambda^d}$.If
    \begin{equation}\label{eq:bound_coef_scal_fct}
        \big\|(\widetilde{\varepsilon}^{\kappa,p}_{j})_{j\in\N}\big\|_{\ell^q(\N)} < \infty, \quad \text{where} \quad 
        \widetilde{\varepsilon}^{\kappa,p}_{j} = \bigg\| \bigg(\frac{d_{\lambda_j^d}}{\kappa(\lambda_j^d)} \bigg)_{\lambda_j^d \in \Lambda_j^d\times 0^d}\bigg\|_{\ell^p(\Lambda_j^d\times 0^d)},
    \end{equation}
    then $g\in B^{\kappa}_{p,q}([0,1]^d)$.
\end{prop}

To prove that $f_a\in B^{\kappa}_{p,\infty}([0,1]^d)$, using \eqref{eq:dec_trace}, it is enough to check that  
\begin{equation}\label{eq:snorm_inh_besov_mod}
    |f_a|_{\kappa,p,q}=\big\|(\varepsilon^{\kappa,p}_{j})_{j\in\N}\big\|_{\ell^q(\N)} < \infty, ~\text{where}~ \varepsilon^{\kappa,p}_{j}=\Big\|\Big( \frac{c_{\lambda}}{\kappa(\lambda)} \Big)_{\lambda\in \Lambda^d_j \times {\{0,1\}^d}}\Big\|_{\ell^p(\Lambda^d_j \times \mathbf{\{0,1\}^d})} \hspace{-4mm}.
\end{equation}
The difference between \eqref{eq:snorm_inh_besov_mod} and the original characterisation \eqref{eq:snorm_inh_besov} is that the sum is now taken on all cubes $\Lambda^d\times \{0,1\}^d$. We keep this characterisation throughout the Section \ref{sec:traces_in_besov}.
 
\begin{proof}
    The idea is to decompose $\psi_{j,k,0^d}$ on the wavelets $\psi_{J,K,l}$ for $l\in L^d$ such that $j>J$.

    First, let us show that $g\in L^p([0,1]^d)$. Remark that, as $\kappa\in \mathcal{C}([0,1]^d)$, $\kappa$ is $s$-H{\"o}lder for some $s>0$. Thus, for some constant $K>0$, for every $\lambda_j^d\in \Lambda_j^d\times \{0,1\}^d$, $\kappa(\lambda_j^d)\leq K 2^{-js}$. Hence,
    \begin{equation}\label{eq:lp_norm_g_upper_bound}
        \|g\|_{L^p}
        \leq \sum_{j\geq 0} \bigg\|\sum_{\lambda_j^d\in\Lambda_j^d \times 0^d} d_{\lambda_j^d} \psi_{\lambda_j^d}\bigg\|_{L^p}
        \leq K \sum_{j\geq 0} 2^{-js} \bigg\|\sum_{\lambda_j^d\in\Lambda_j^d \times 0^d} \frac{d_{\lambda_j^d}}{\kappa(\lambda_j^d)} \psi_{\lambda_j^d}\bigg\|_{L^p}.
    \end{equation}
    Each $\psi_{\lambda_j^d}$ is compactly supported, with support included in $2^{-j}[-M_{\psi},M_{\psi}]^d+2^{-j}k$. In addition, 
    \begin{equation*}
        C(\psi):=\sup_{x\in\R^d}\sum_{\lambda_j^d\in \Lambda_j^d \times 0^d} |\psi_{\lambda_j^d}(x)| \leq (4M_{\psi}+1)^D \|\psi\|_{L^{\infty}}.
    \end{equation*}
    When $p=\infty$, one sees that
    \begin{equation*}
        \bigg|\sum_{\lambda_j^d\in\Lambda_j^d \times 0^d} \frac{d_{\lambda_j^d}}{\kappa(\lambda_j^d)} \psi_{\lambda_j^d}(x)\bigg|
        \leq \sum_{\lambda_j^d\in\Lambda_j^d \times 0^d} \bigg| {\frac{d_{\lambda_j^d}}{\kappa(\lambda_j^d)}}\bigg| |\psi_{\lambda_j^d}(x)|
        \leq \bigg\| {\frac{d_{\lambda_j^d}}{\kappa(\lambda_j^d)}}\bigg\|_{\ell^{\infty}} \sum_{\lambda_j^d\in\Lambda_j^d \times 0^d} |\psi_{\lambda_j^d}(x)|,
    \end{equation*}
    which is smaller than $ C(\psi)~\widetilde{\varepsilon}^{\kappa,\infty}_{j}$. This induces that 
    \begin{equation}\label{eq:lp_norm_g_p1}
        \bigg\|\sum_{\lambda_j^d\in\Lambda_j^d \times 0^d} \frac{d_{\lambda_j^d}}{\kappa(\lambda_j^d)} \psi_{\lambda_j^d}\bigg\|_{L^{\infty}}
        \leq C(\psi)~\widetilde{\varepsilon}^{\kappa,\infty}_{j}.
    \end{equation}
    When $p\in[1,\infty)$, let $q\in(1,\infty]$ be such that $1/p+1/q=1$.
    The H{\"o}lder inequality gives  
    \begin{align*}
        \bigg|\sum_{\lambda_j^d\in\Lambda_j^d \times 0^d} \frac{d_{\lambda_j^d}}{\kappa(\lambda_j^d)} \psi_{\lambda_j^d}(x)\bigg|
        &\leq \bigg(\sum_{\lambda_j^d\in\Lambda_j^d \times 0^d} \bigg| {\frac{d_{\lambda_j^d}}{\kappa(\lambda_j^d)}}\bigg| ^p |\psi_{\lambda_j^d}(x)| \bigg)^{1/p} \bigg(\sum_{\lambda_j^d\in\Lambda_j^d \times 0^d} |\psi_{\lambda_j^d}(x)| \bigg)^{1/q} \\
        &\leq C(\psi)^{1/q} \bigg(\sum_{\lambda_j^d\in\Lambda_j^d \times 0^d} \bigg| {\frac{d_{\lambda_j^d}}{\kappa(\lambda_j^d)}}\bigg| ^p |\psi_{\lambda_j^d}(x)| \bigg)^{1/p}.
    \end{align*}
    One deduces that
    \begin{align}
        \bigg\|\sum_{\lambda_j^d\in\Lambda_j^d \times 0^d} \frac{d_{\lambda_j^d}}{\kappa(\lambda_j^d)} \psi_{\lambda_j^d}\bigg\|_{L^p}
        &\leq C(\psi)^{1/q} \bigg(\sum_{\lambda_j^d\in\Lambda_j^d \times 0^d} \bigg| {\frac{d_{\lambda_j^d}}{\kappa(\lambda_j^d)}}\bigg| ^{p} \int_{\R^d} |\psi_{\lambda_j^d}(x)| dx \bigg)^{1/p} \notag\\
        &\leq C(\psi)^{1/q} \|\psi\|_{L^1}^{1/p} \,\widetilde{\varepsilon}^{\kappa,p}_{j}. \label{eq:lp_norm_g_p2}
    \end{align}
    Since $|\widetilde{\varepsilon}_{j}^{\kappa,p}| \leq \big\|\widetilde{\varepsilon}_{j}^{\kappa,p}\big\|_{\ell^q(\N)}$, by combining \eqref{eq:lp_norm_g_p1}, \eqref{eq:lp_norm_g_p2} and \eqref{eq:lp_norm_g_upper_bound}, it gives that 
    \begin{equation*}
        \|g\|_{L^p}
        \leq K \sum_{j\geq 0} 2^{-js} C \widetilde{\varepsilon}^{\kappa,p}_{j}
        \leq K \sum_{j\geq 0} 2^{-js} C \big\| \widetilde{\varepsilon}_{j}^{\kappa,p} \big\|_{\ell^q(\N)}
        \leq C' \big\| \widetilde{\varepsilon}_{j}^{\kappa,p} \big\|_{\ell^q(\N)} < \infty
    \end{equation*} 
    and $g\in L^p([0,1]^d)$.
    \smallskip

\noindent We show that the seminorm \eqref{eq:snorm_inh_besov} of $g$ is finite. Applying \eqref{eq:wlt_dec_on_01} to $g\in L^p([0,1]^d)$,  
    \begin{equation}
        g = c _{0,0^d,0^d} \psi_{0,0^d,0^d} + \sum_{\lambda^d \in \Lambda^d\times L^d} c_{\lambda^d} \psi_{\lambda^d},
    \end{equation}
    where for $(J,K,l)\in \Lambda^d\times L^d$
    \begin{align*}
        c _{0,0^d,0^d} 
        &= \int_{\R^D} g(x) \psi_{(0,0^d,0^d)}(x)dx \\
        &= \int_{\R^D} \sum_{j\geq 0} \sum_{k\in \{0,1,\ldots,2^j-1\}^d} d_{j,k,0^d} \psi_{(j,k,0^d)}(x) \psi_{(0,0^d,0^d)}(x)dx, \\
        \textnormal{and}\quad
        c_{J,K,l} 
        &= 2^{dJ} \int_{\R^D} g(x) \psi_{(J,K,l)}(x)dx \\
        &= 2^{dJ} \int_{\R^D} \sum_{j>J} \sum_{k\in \{0,1,\ldots,2^j-1\}^d} d_{j,k,0^d} \psi_{(j,k,0^d)}(x) \psi_{(J,K,l)}(x)dx.
    \end{align*} 
    The sum goes over $j>J$ since $\psi_{(J,K,l)}$ is orthogonal to all $\psi_{(j,k,0^d)}$ with $j\leq J$. Also, only wavelets coefficients indexed with dyadic cubes in $[0,1]^d$ are considered, hence $k\in \{0,1,\ldots,2^j-1\}^d$.
    
    One sees that for any fixed $j$ and $J$, there is a finite number of $k\in \{0,1,\ldots,2^j-1\}^d$, such that the support of $\psi_{(j,k,0^d)}$ is included in the support of $\psi_{(J,K,l)}$. 
    Since $\psi^0$ and $\psi^1$ are compactly supported, their support is included in an interval $[-M_{\psi},M_{\psi}]$. 
    Denote, for every $j,J,K$,
    \begin{equation*}
        S_{j,J,K} = \big\{k\in \{0,1,\ldots,2^j-1\}^d~:~\operatorname{supp}(\psi_{(j,k,0^d)})~\cap~\operatorname{supp}(\psi_{(J,K,l)}) \neq \emptyset \big\}.
    \end{equation*}
    This set contains the dyadic cubes $\lambda_j^d\in\Lambda_j^d\times 0^d$ such that the associated $d_{\lambda_j^d}$ appears in the computation of $c_{\lambda_J^d}$ for $\lambda_J^d\in\Lambda_J^d \times L^d$ above. Observe that for every $j,J\in\N$, 
    \begin{equation}\label{eq:card_S_jJK}
        \#S_{j,J,K} \leq (2 M_{\psi}+1)^d~2^{d(j-J)}.
    \end{equation}

    To establish that $g\in B^{\kappa}_{p,q}([0,1]^d)$, we use \eqref{eq:snorm_inh_besov}. First, one has
    \begin{equation*}
        \bigg| \frac{c_{J,K,l}}{\kappa(\lambda^d_{J,K})} \bigg| 
        \leq 2^{dJ} \sum_{j>J} \sum_{k\in \{0,1,\ldots,2^j-1\}^d} \bigg|\frac{d_{j,k,0^d}}{\kappa(\lambda^d_{j,k})}\bigg| \bigg|\frac{\kappa(\lambda^d_{j,k})}{\kappa(\lambda^d_{J,K})}\bigg| \int_{\R^D} \big|\psi_{(j,k,0^d)}(x) \psi_{(J,K,l)}(x)\big| dx.
    \end{equation*}
    Hence, using the definition of $S_{j,J,K}$, for some constant $C_\psi$ depending on $\psi$ only, 
    \begin{align*}
        \bigg| \frac{c_{J,K,l}}{\kappa(\lambda^d_{J,K})} \bigg| 
        &\leq 2^{dJ} \sum_{j>J} \sum_{k\in S_{j,J,K}} \bigg|\frac{d_{j,k,0^d}}{\kappa(\lambda^d_{j,k})}\bigg| \bigg|\frac{\kappa(\lambda^d_{j,k})}{\kappa(\lambda^d_{J,K})}\bigg| \|\psi\|_{L^\infty}^d \int_{\R^D} \big|\psi_{(j,k,0^d)}(x)\big| dx \\
        &\leq C_{\psi} 2^{dJ} \sum_{j>J} 2^{-jd} \sum_{k\in S_{j,J,K}} \bigg|\frac{d_{j,k,0^d}}{\kappa(\lambda^d_{j,k})}\bigg| \bigg|\frac{\kappa(\lambda^d_{j,k})}{\kappa(\lambda^d_{J,K})}\bigg|.
    \end{align*}
    The H{\"o}lder inequality gives, for $\frac{1}{p}+\frac{1}{r}=1$,
    \begin{equation*}
        \bigg| \frac{c_{J,K,l}}{\kappa(\lambda^d_{J,K})} \bigg| 
        \leq C_{\psi} 2^{dJ} \sum_{j>J} 2^{-jd} \bigg (\sum_{k\in S_{j,J,K}} \bigg|\frac{d_{j,k,0^d}}{\kappa(\lambda^d_{j,k})}\bigg| ^p\bigg)^{1/p} \bigg(\sum_{k\in S_{j,J,K}} \bigg| {\frac{\kappa(\lambda^d_{j,k})}{\kappa(\lambda^d_{J,K})}}\bigg| ^r\bigg)^{1/r}.
    \end{equation*}
    Observe that for all $k\in S_{j,J,K}$, $\lambda^d_{j,k} \subset (4 M_{\psi}+1)\,\lambda^d_{J,K}$. Thus by the doubling inequality \eqref{eq:doub_meas_ball} for $\kappa$ with the constant $C_{\kappa}>0$, one has, for all $k\in S_{j,J,K}$, 
    \begin{equation}\label{eq:ratio_kappa}
        \frac{\kappa(\lambda^d_{j,k})}{\kappa (\lambda^d_{J,K} )} 
        = \frac{\kappa(\lambda^d_{j,k})}{\kappa ((4 M_{\psi}+1)\,\lambda^d_{J,K} )} \frac{\kappa ((4 M_{\psi}+1)\,\lambda^d_{J,K} )}{\kappa (\lambda^d_{J,K} )} 
        \leq C_{\kappa}^{\partent{\log_2(4 M_{\psi}+1)}+1}.
    \end{equation}
    Applying \eqref{eq:card_S_jJK} and \eqref{eq:ratio_kappa} to the second sum above, one gets
    \begin{align*}
        \bigg| \frac{c_{J,K,l}}{\kappa(\lambda^d_{J,K})} \bigg| 
        &\leq C_{\psi,\kappa} 2^{dJ} \sum_{j>J} 2^{-jd} \bigg(\sum_{k\in S_{j,J,K}} \bigg|{\frac{d_{j,k,0^d}}{\kappa(\lambda^d_{j,k})}}\bigg|^p\bigg)^{1/p} 2^{\frac{d}{r}(j-J)} \\
        &= C_{\psi,\kappa} \sum_{j>J} 2^{d(j-J)(-1/p)} \bigg(\sum_{k\in S_{j,J,K}} \bigg|{\frac{d_{j,k,0^d}}{\kappa(\lambda^d_{j,k})}}\bigg|^p\bigg)^{1/p}. 
    \end{align*}
    It is immediate that $\varepsilon^{\kappa,p}_{J} =\Big\| \Big( \Big| \frac{c_{J,K,l}}{\kappa(\lambda^d_{J,K})}\Big|\Big)_{(K,l)\in \{0,\ldots,2^J-1\}^d \times L^d}\Big\|_{\ell^p}$ satisfies
    \begin{align*}
        \varepsilon^{\kappa,p}_{J}  
        &\leq \bigg\|{\bigg( C \sum_{j>J} 2^{-\frac{d}{p}(j-J)} \bigg(\sum_{k\in S_{j,J,K}} \bigg|{\frac{d_{j,k,0^d}}{\kappa(\lambda^d_{j,k})}}\bigg|^p\bigg)^{1/p} \bigg)_{(K,l)\in \{0,\ldots,2^J-1\}^d \times L^d}}\bigg\|_{\ell^p}
    \end{align*}
    In the last norm, the elements do not depend on $l\in L^d$, hence 
    \begin{equation}\label{eq:up_bound_coef_g}
        \varepsilon^{\kappa,p}_{J}
        \leq C_{\psi,\kappa,d} \sum_{j>J} 2^{-\frac{d}{p}(j-J)} \bigg\| \bigg(\sum_{k\in S_{j,J,K}} \bigg|{\frac{d_{j,k,0^d}}{\kappa(\lambda^d_{j,k})}}\bigg|^p\bigg)^{1/p}_{K\in \{0,\ldots,2^J-1\}^d}\bigg\|_{\ell^p}.
    \end{equation}
    Recalling that $\psi^0$ and $\psi^1$ have their support in $[-M_{\psi},M_{\psi}]$, for a fixed $J\in\N$ and $l\in L^d$, a point $x\in[0,1]^d$ is included in the support of (at most) $(2 M_{\psi}+1)^d$ wavelets $\psi_{(J,K,l)}$. Thus, for $j>J$, a scaling function $\psi_{(j,k,0^d)}$ has a support intersecting that of $\psi_{(J,K,l)}$ for at most $(4 M_{\psi}+1)^d$ multi-integer $K$. One deduces that, for $j,J$ fixed with $j>J$, a $k\in\{0,\ldots,2^j-1\}^d$ is in at most $(4 M_{\psi}+1)^d$ sets $S_{j,J,K}$ for $K\in\{0,\ldots,2^J-1\}^d$. One sees that
    \begin{equation*}
        \sum_{K\in \{0,\ldots,2^J-1\}^d} \sum_{k\in S_{j,J,K}} \bigg|{\frac{d_{j,k,0^d}}{\kappa(\lambda^d_{j,k})}}\bigg|^p \leq C_{\psi,d} \sum_{k\in \{0,\ldots,2^j-1\}^d} \bigg|{\frac{d_{j,k,0^d}}{\kappa(\lambda^d_{j,k})}}\bigg|^p = C_{\psi,d} (\widetilde{\varepsilon}_{j}^{\kappa,p})^p.
    \end{equation*}
    So $\varepsilon^{\kappa,p}_{J}
        \leq C_{\psi,\kappa,d} \sum_{j>J} 2^{-\frac{d}{p}(j-J)} \widetilde{\varepsilon}_{j}^{\kappa,p} $, and by the H\"older inequality, for a constant $C$ that changes along the computations,
    \begin{equation*}
        |\varepsilon^{\kappa,p}_{J}|
        \leq C \Big(\sum_{j>J} 2^{-\frac{d}{p}(j-J)} |\widetilde{\varepsilon}_{j}^{\kappa,p}|^q\Big)^{1/q} \Big(\sum_{j>J} 2^{-\frac{d}{p}(j-J)}\Big)^{1/(1-1/q)} 
        \leq C \Big(\sum_{j>J} 2^{-\frac{d}{p}(j-J)} |\widetilde{\varepsilon}_{j}^{\kappa,p}|^q\Big)^{1/q}.
    \end{equation*}
    Hence, for $C$ depending on $\psi,\kappa,d$ and $q$,
    \begin{align*}
        \sum_{J\in\N} |\varepsilon^{\kappa,p}_{J}|^q
        &\leq C \sum_{J\in\N} \sum_{j>J} 2^{-\frac{d}{p}(j-J)} |\widetilde{\varepsilon}_{j}^{\kappa,p} |^q 
        \leq C \sum_{j\in\N} |\widetilde{\varepsilon}_{j}^{\kappa,p} |^q \sum_{J<j} 2^{-\frac{d}{p}(j-J)}
        \leq C \sum_{j\in\N} |\widetilde{\varepsilon}_{j}^{\kappa,p} |^q.
    \end{align*}
    One concludes by \eqref{eq:bound_coef_scal_fct} that $g\in B^{\kappa}_{p,q}([0,1]^d)$.
\end{proof} 

\subsection{Proof of Theorem \ref{theo:incl_trace}}

The product capacity $\xi = \mu \times \nu$ is fixed. Recall that $\Gamma^p_{\nu,r} = h^{r}_{\nu} +\frac{\dim(\nu_{r})}{p}$, with $\Gamma^{\infty}_{\nu,r} = h^{r}_{\nu}$.

\begin{prop}\label{prop:belong_trace_inter}
    Let $p\in [1,+\infty]$, $\Gamma<\Gamma^p_{\nu,r}$, $f\in \widetilde{B}^{\xi}_{p,q}([0,1]^{D})$ and $r \in \R$. Then for $\nu_{r}$-almost all $a\in[0,1]^{d'}$, $f_a \in B^{\mu^{(+\Gamma)}}_{p,1}([0,1]^d) $.
\end{prop}
\begin{proof} 
    Since $f$ is continuous, each trace $f_a$ \eqref{eq:def_trace} is also continuous and $f_a\in L^{p}([0,1]^d)$.
    
    We start with the case $ p<+\infty$. Consider $\Gamma<\Gamma^p_{\nu,r}$, $f\in \widetilde{B}^{\xi}_{p,q}([0,1]^{D})$ and $r \in \R$. To prove Proposition \ref{prop:belong_trace_inter}, one needs to show that $|f_a|_{\mu^{(+\Gamma)},p,1}<\infty$. Recalling \eqref{eq:snorm_inh_besov_mod}, and since $\mu$, $\nu$ and $f$ are fixed, we introduce a more convenient notation for the quantity $\varepsilon^{\mu^{(+\Gamma)},p}_{j}$ associated with $f_a$: we rewrite 
    \begin{equation}\label{eq:S_j_besov_func}
        S^{p}_{j}(a,\Gamma) := \sum_{\lambda^d\in \Lambda^{d}_j \times \{0,1\}^{d}} \bigg|\frac{d_{\lambda^d}(a)}{\mu(\lambda^d) 2^{-j\Gamma}}\bigg|^p \hspace{5mm} (=\varepsilon^{\mu^{(+\Gamma)},p}_{j}).
    \end{equation}
    In $S^{p}_{j}(a,\Gamma)$, the coefficients $d_{\lambda^d}(a)$ are either $d^G_{\lambda^d}(a)$ if $\lambda^d\in \Lambda^d\times 0^d$ or $d^F_{\lambda^d}(a)$ if $\lambda^d\in \Lambda^d\times L^d$. Let us show an upper bound for the $S^{p}_{j}(a,\Gamma)$.

    First, for $j\geq 1$, for the wavelet coefficient \eqref{eq:dec_trace_G_a} of $G_{a}$
    \begin{equation*}
        \sum_{\lambda^d\in \Lambda_j^d\times 0^d} \bigg|\frac{d^G_{\lambda^d}(a)}{\mu(\lambda^d) 2^{-j\Gamma}}\bigg|^p = \sum_{\lambda^d\in \Lambda_j^d\times 0^d} \bigg| {\frac{1}{\mu(\lambda^d) 2^{-j\Gamma}}}\bigg| ^p \cdot \Big| {\sum_{\lambda^{d'}\in\Lambda^{d'}_j \times L^{d'}} c_{\lambda^d,\lambda^{d'}} \psi_{\lambda^{d'}}(a)}\Big| ^p.
    \end{equation*}
    By the compactness of the supports of $\psi^{0}$ and $\psi^{1}$, for each $a$, there is a finite number (bounded independently of $j$ or $a$) of non-zero terms $\psi_{\lambda^{d'}}(a)$ in the sum on $\Lambda^{d'}_j \times L^{d'}$. Hence, for some constant $C_{p,\psi}>0$ depending on $p$, $\psi^{0}$ and $\psi^{1}$,
    \begin{align*}
        \Big| {\sum_{\lambda^{d'}\in\Lambda^{d'}_j \times L^{d'}} c_{\lambda^d,\lambda^{d'}} \psi_{\lambda^{d'}}(a)}\Big| ^p 
        &\leq C_{p,\psi} \sum_{\lambda^{d'}\in\Lambda^{d'}_j \times L^{d'}} \big|c_{\lambda^d,\lambda^{d'}} \psi_{\lambda^{d'}}(a)\big|^p \\
        &\leq C_{p,\psi} \sum_{\lambda^{d'}\in\Lambda^{d'}_j \times \{0,1\}^{d'}} \big|c_{\lambda^d,\lambda^{d'}} \psi_{\lambda^{d'}}(a)\big|^p.
    \end{align*}
    Using that $\xi(\lambda^d \times \lambda^{d'})=\mu(\lambda^d) \cdot \nu(\lambda^{d'})$, one deduces that
    \begin{align*}
        \sum_{\lambda^d\in \Lambda_j^d\times 0^d} \bigg|\frac{d^G_{\lambda^d}(a)}{\mu(\lambda^d) 2^{-j\Gamma}}\bigg|^p 
        &\leq C_{p,\psi} \sum_{\lambda^d\in \Lambda_j^d\times 0^d} \bigg|\frac{1}{\mu(\lambda^d) 2^{-j\Gamma}}\bigg|^p \sum_{\lambda^{d'}\in\Lambda^{d'}_j \times \{0,1\}^{d'}} \big|c_{\lambda^d,\lambda^{d'}} \psi_{\lambda^{d'}}(a)\big|^p\\
        &= C_{p,\psi} \sum_{\lambda^D \in \Lambda_j^D\times (0^d \times \{0,1\}^{d'})} \bigg|\frac{1}{\mu(\lambda^d) 2^{-j\Gamma}}\bigg|^p \big|c_{\lambda^D} \psi_{\lambda^{d'}}(a)\big|^p \\
        &= C_{p,\psi} \sum_{\lambda^D \in \Lambda_j^D\times (0^d \times \{0,1\}^{d'})} \bigg|\frac{c_{\lambda^D}}{\xi(\lambda^D)}\bigg|^p \bigg|\frac{\nu(\lambda^{d'})}{2^{-j\Gamma}}\bigg|^p \big|\psi_{\lambda^{d'}}(a)\big|^p. \numberthis \label{eq:maj_sum_coef_trace_1}
    \end{align*}
    When $j=0$, using that 
    \begin{equation*}
        \Big| {\sum_{\lambda^{d'}\in\Lambda^{d'}_0 \times \{0,1\}^{d'}} c_{\lambda^d,\lambda^{d'}} \psi_{\lambda^{d'}}(a)}\Big| ^p 
        \leq C_{p,\psi} \sum_{\lambda^{d'}\in\Lambda^{d'}_0 \times \{0,1\}^{d'}} \big|c_{\lambda^d,\lambda^{d'}} \psi_{\lambda^{d'}}(a)\big|^p,
    \end{equation*}
    the same computations give that 
    \begin{equation}\label{eq:maj_sum_coef_trace_2}
        \sum_{\lambda^d\in \Lambda_0^d\times 0^d} \bigg|\frac{d^G_{\lambda^d}(a)}{\mu(\lambda^d) 2^{-j\Gamma}}\bigg|^p 
        \leq C_{p,\psi} \sum_{\lambda^D \in \Lambda_0^D\times (0^d \times \{0,1\}^{d'})} \bigg|\frac{c_{\lambda^D}}{\xi(\lambda^D)}\bigg|^p \bigg|\frac{\nu(\lambda^{d'})}{2^{-j\Gamma}}\bigg|^p \big|\psi_{\lambda^{d'}}(a)\big|^p.
    \end{equation}
    The same computations hold for the wavelet coefficients \eqref{eq:dec_trace_F_a} of $F_a$ and, for every $j\in\N$,
    \begin{equation}\label{eq:maj_sum_coef_trace_3}
        \sum_{\lambda^d\in \Lambda_j^d\times L^d} \bigg|\frac{d^F_{\lambda^d}(a)}{\mu(\lambda^d) 2^{-j\Gamma}}\bigg|^p 
        \leq C_{p,\psi} \sum_{\lambda^D \in \Lambda_j^D\times (L^d \times \{0,1\}^{d'})} \bigg|\frac{c_{\lambda^D}}{\xi(\lambda^D)}\bigg|^p \bigg|\frac{\nu(\lambda^{d'})}{2^{-j\Gamma}}\bigg|^p \big|\psi_{\lambda^{d'}}(a)\big|^p.
    \end{equation}
    Noticing that $L^d \cup 0^d = \{0,1\}^d$, we obtain from \eqref{eq:maj_sum_coef_trace_1}, \eqref{eq:maj_sum_coef_trace_2} and \eqref{eq:maj_sum_coef_trace_3} that
    \begin{equation}\label{eq:maj_Sj}
        S^{p}_{j}(a,\Gamma) 
        \leq C_{p,\psi}\sum_{\lambda^D\in \Lambda^{D}_j \times \{0,1\}^{D}} \bigg|\frac{c_{\lambda^D}}{\xi(\lambda^D)}\bigg|^p \bigg|\frac{\nu(\lambda^{d'})}{2^{-j\Gamma}}\bigg|^p \big|\psi_{\lambda^{d'}}(a)\big|^p.
    \end{equation}
    
    Recall that as a Gibbs capacity, $\xi$ satisfies \eqref{eq:P1_ball} for some $s_1>0$.
    Let us choose $M_{\nu_{r}}$ large enough that for every $m\geq M_{\nu_{r}}$,
    \begin{equation}\label{eq:condition_m}
        m \geq \partent{ {1}/{s_1}}+1 ~~\mbox{and}~~ \Gamma < \Gamma^p_{\nu,r}- ({2p+2})/{pm}.
    \end{equation} 
    
    Let $n\in\N^*$, $m \geq M_{\nu_{r}}$ and $A^{\nu_r}_{n,m}$ be the set and $J_{n,m}$ the integer from Proposition \ref{prop:bound_meas_dyadics}.
    Using \eqref{eq:maj_Sj} and integrating $S^{p}_{j}(a,\Gamma)$ over $A^{\nu_r}_{n,m}$ yields
    \begin{align*}
        &\int_{A^{\nu_r}_{n,m}} S^{p}_{j}(a,\Gamma) d\nu_r(a) = \int_{[0,1]^{d'}} S^{p}_{j}(a,\Gamma) d\nu_{r\vert_{A^{\nu_r}_{n,m}}}(a) \\
        &\leq C_{p,\psi} \sum_{\lambda^D\in \Lambda^{D}_j \times \{0,1\}^{D}} \bigg|\frac{c_{\lambda^D}}{\xi(\lambda^D)}\bigg|^p \bigg|\frac{\nu(\lambda^{d'})}{2^{-j\Gamma}}\bigg|^p \int_{[0,1]^{d'}} \big|\psi_{\lambda^{d'}}(a)\big|^p d\nu_{r\vert_{A^{\nu_r}_{n,m}}}(a) \\
        &\leq C_{p,\psi} \sum_{\lambda^D\in \Lambda^{D}_j \times \{0,1\}^{D}} \bigg|\frac{c_{\lambda^D}}{\xi(\lambda^D)}\bigg|^p \bigg|\frac{\nu(\lambda^{d'})}{2^{-j\Gamma}}\bigg|^p \int_{[0,1]^{d'}} \|\psi\|_{\infty}^{d'p} \mathds{1}_{\widetilde{K}\lambda^{d'}}(a) d\nu_{r\vert_{A^{\nu_r}_{n,m}}}(a)
    \end{align*}
    with $\|\psi\|_{\infty} = \max\{\|\psi^{0}\|_{\infty}, \|\psi^{1}\|_{\infty} \}$.
    Hence there is $C>0$, depending also on $d$, such that
    \begin{align*}
        &\int_{A^{\nu_r}_{n,m}} S^{p}_{j}(a,\Gamma) d\nu_r(a) \\
        &\leq C \sum_{\lambda^D\in \Lambda^{D}_j \times \{0,1\}^{D}} \bigg|\frac{c_{\lambda^D}}{\xi(\lambda^D)}\bigg|^p \bigg|\frac{\nu(\lambda^{d'})}{2^{-j\Gamma}}\bigg|^p \nu_r \big( A^{\nu_r}_{n,m} \cap \widetilde{K}\lambda^{d'} \big) \\
        &\leq C \sum_{\lambda^D\in \Lambda^{D}_j \times \{0,1\}^{D}} \bigg|\frac{c_{\lambda^D}}{\xi(\lambda^D)}\bigg|^p \bigg|\frac{\nu(\lambda^{d'})}{2^{-j\Gamma}}\bigg|^p \nu_r \big( \widetilde{K}\lambda^{d'} \big) \mathds{1}_{A^{\nu_r}_{n,m} \cap \widetilde{K}\lambda^{d'} \neq \emptyset} \\
        &\leq C \sum_{\substack{\lambda^{d'} \in \Lambda^{d'}_j \times \{0,1\}^{d'} \\ A^{\nu_r}_{n,m} \cap \widetilde{K}\lambda^{d'} \neq \emptyset}} \bigg|\frac{\nu(\lambda^{d'})}{2^{-j\Gamma}}\bigg|^p \nu_r \big( \widetilde{K}\lambda^{d'} \big) \sum_{\lambda^{d} \in \Lambda^{d}_j \times \{0,1\}^{d}} \bigg|\frac{c_{\lambda^d,\lambda^{d'}}}{\xi(\lambda^d \times \lambda^{d'})}\bigg|^p \numberthis\label{eq:maj_int_S_j_over_Anm_sum}
    \end{align*}
    as $A^{\nu_r}_{n,m} \cap \widetilde{K}\lambda^{d'} \subset \widetilde{K}\lambda^{d'}$.
    Applying the doubling property \eqref{eq:doub_meas_ball} of $\nu_r$ to pass from $\nu_r(\widetilde{K}\lambda^{d'})$ to $\nu_r (\lambda^{d'} )$, the last inequality gives that, for $C>0$ depending on $\nu$ and $r$, that
    \begin{equation*}
        \int_{A^{\nu_r}_{n,m}} S^{p}_{j}(a,\Gamma)\ d\nu_{r}(a)
        \leq C \sum_{\substack{\lambda^{d'}\in \Lambda^{d'}_j \times \{0,1\}^{d'} \\ \widetilde{K}\lambda^{d'} \cap \ A^{\nu_r}_{n,m} \neq \emptyset}} \frac{\nu(\lambda^{d'})^{p} \ \nu_{r}(\lambda^{d'})}{2^{-jp\Gamma}} \sum_{\lambda^{d}\in \Lambda^{d}_j \times \{0,1\}^{d}} \bigg|\frac{c_{\lambda^d,\lambda^{d'}}}{\xi(\lambda^d \times \lambda^{d'})}\bigg|^p.
    \end{equation*}

    Recalling that $\Gamma^p_{\nu,r} = h^{r}_{\nu} +\frac{\dim(\nu_{r})}{p}$, \eqref{eq:bound_meas_alm_all_dyad} and \eqref{eq:bound_auxmeas_alm_all_dyad} give
    \begin{align*}
        &\int_{A^{\nu_r}_{n,m}} S^{p}_{j}(a,\Gamma)\ d\nu_{r}(a) \\
        &\leq C \sum_{\substack{\lambda^{d'}\in \Lambda^{d'}_j \times \{0,1\}^{d'} \\ \widetilde{K}\lambda^{d'} \cap \ A^{\nu_r}_{n,m} \neq \emptyset}} 2^{-j(p\Gamma_{\nu,r}-(p+1)/m-p\Gamma)} \sum_{\lambda^{d}\in \Lambda^{d}_j \times \{0,1\}^{d}} \bigg|\frac{c_{\lambda^d,\lambda^{d'}}}{\xi(\lambda^d \times \lambda^{d'})}\bigg|^p\\
        &\leq C\ 2^{-j(p\Gamma_{\nu,r}-(2p+1)/m-p\Gamma)} \sum_{\substack{\lambda^{d'}\in \Lambda^{d'}_j \times \{0,1\}^{d'} \\ \widetilde{K}\lambda^{d'} \cap \ A^{\nu_r}_{n,m} \neq \emptyset}} \sum_{\lambda^{d}\in \Lambda^{d}_j \times \{0,1\}^{d}} \bigg|\frac{c_{\lambda^d,\lambda^{d'}}}{\xi^{(-1/m)}(\lambda^d\times\lambda^{d'})}\bigg|^p.
    \end{align*} 
    By \eqref{eq:wlt_dec_on_01}, observe that:
    \begin{itemize}
        \item for $j\geq 1$, $f$ has no wavelet coefficient associated with the index $l^D=0^D$, hence the sum of $\lambda^D$ over $\Lambda^D_j \times \{0,1\}^D$ is the same as the sum over $\Lambda^D_j \times L^D$ and, by \eqref{eq:snorm_inh_besov},
        \begin{equation*}
            \sum_{\lambda^{d'}\in \Lambda^{d'}_j \times \{0,1\}^{d'}} \sum_{\lambda^{d}\in \Lambda^{d}_j \times \{0,1\}^{d}} \bigg|\frac{c_{\lambda^d,\lambda^{d'}}}{\xi^{(-1/m)}(\lambda^d \times \lambda^{d'})}\bigg|^p 
            = \sum_{\lambda^{D} \in \Lambda^{D}_j \times L^D} \bigg|\frac{c_{\lambda^{D}}}{\xi^{(-1/m)}(\lambda^{D})}\bigg|^p,
        \end{equation*}
        \item for $j=0$, since $ |c_{(0,0^D,0^D)} |<\infty$, 
        \begin{align*}
            \sum_{\lambda^{D} \in \Lambda^{D}_0 \times \{0,1\}^D} \bigg|\frac{c_{\lambda^{D}}}{\xi^{(-1/m)}(\lambda^{D})}\bigg|^p 
            &= \frac{1}{\xi^{(-1/m)}([0,1]^{D})} \bigg(\big|c_{(0,0^D,0^D)}\big|^p + \hspace{-10pt}\sum_{\lambda^{D} \in \Lambda^{D}_0 \times L^D} \big|c_{\lambda^{D}}\big|^p \bigg)\\
            &\leq C \big(\big|c_{(0,0^D,0^D)}\big|^p + \big(\varepsilon^{\xi^{(-1/m)}}_{0}\big)^p\big).
        \end{align*}
    \end{itemize}
    
    The choice \eqref{eq:condition_m} for $m$ ensures that $ {1}/{m}<s_1$ and $f\in \widetilde{B}^{\xi^{(-1/m)}}_{p,q}([0,1]^D)$. This implies that $\big(\varepsilon^{\xi^{(-1/m)}}_{j}\big)_{j\in\N}\in l^q(\N)$, and in particular that this sequence is bounded. Hence 
    \begin{equation}\label{eq:smaller_in_mean_of_Sj}
        \int_{A^{\nu_r}_{n,m}} S^{p}_{j}(a,\Gamma)~d\nu_{r}(a) \leq C~2^{-j(p\Gamma_{\nu,r}-(2p+1)/m-p\Gamma)}.
    \end{equation}
    Let 
    \begin{equation}\label{eq:A_p_j}
        A^{\nu_{r},p}_{n,m,j} := \big\{ a \in A^{\nu_{r}}_{n,m}\ : \ S^{p}_{j}(a,\Gamma) > C~j^2~2^{-j(p\Gamma_{\nu,r}-(2p+1)/m-p\Gamma)}\big\}.
    \end{equation}
    Markov's inequality using \eqref{eq:smaller_in_mean_of_Sj} gives that $\nu_{r}(A^{\nu_{r},p}_{n,m,j}) \leq \frac{2^{-j(p\Gamma_{\nu,r}-(2p+1)/m-p\Gamma)}}{j^2~2^{-j(p\Gamma_{\nu,r}-(2p+1)/m-p\Gamma)}} = j^{-2}$.
    So $\sum_{j\in\N} \nu_{r}(A^{\nu_{r},p}_{n,m,j})$ converges and the Borel-Cantelli lemma yields that 
    \begin{equation}\label{eq:limsup_A_j_zero}
        \nu_{r}(\limsup_{j\to\infty} A^{\nu_{r},p}_{n,m,j}) = 0.
    \end{equation}
    So, for every $a \in \widetilde{A}^{\nu_{r},p}_{n,m} := A^{\nu_r}_{n,m} \backslash \limsup_{j\to\infty} A^{\nu_{r},p}_{n,m,j}$, there is $J'_{n,m}\in \N$ such that for $j \geq J'_{n,m}$
    \begin{align*}
        S^{p}_{j}(a,\Gamma) 
        &\leq C j^{2} 2^{-j(p\Gamma_{\nu,r}-(2p+1)/m-p\Gamma)} 
        = C 2^{-j \big( p\Gamma_{\nu,r}-\frac{(2p+1)}{m}-2\frac{\log_2(j)}{j}-p\Gamma \big)} \\
        &\leq C 2^{-j \big( p\Gamma_{\nu,r}-\frac{(2p+2)}{m}-p\Gamma \big)},
    \end{align*}
    that last inequality being true for large integers $j$. Remark that for all $j \leq J'_{n,m}$, $S^{p}_{j}(a,\Gamma)$ is bounded. By our choice \eqref{eq:condition_m} for $m$, one has $\Gamma < \Gamma_{\nu,r} - \frac{2p+2}{pm}$, implying that the sum $\sum_{j\geq 1} S^{p}_{j}(a,\Gamma)$ converges and $f_a \in B^{\mu^{(+\Gamma)}}_{p,1}([0,1]^d)$ by the characterisation \eqref{eq:snorm_inh_besov_mod}.
    
    For every $n\in\N^*$, one has $\nu_r ( \widetilde{A}^{\nu_{r},p}_{n,m} )=1-\frac{1}{n}$ by \eqref{eq:limsup_A_j_zero}. So $\nu_r ( \bigcup_{n\in\N^*}\widetilde{A}^{\nu_{r},p}_{n,m} )=1$ and $\bigcup_{n\in\N^*}\widetilde{A}^{\nu_{r},p}_{n,m}$ has full $\nu_r$-measure. So, for every $a\in\bigcup_{n\in\N^*}\widetilde{A}^{\nu_{r},p}_{n,m}$, $f_a \in B^{\mu^{(+\Gamma)}}_{p,1}([0,1]^d)$, hence the conclusion.
    
    Let us now explain how to adapt the proof to the case $p=\infty$. Consider $\Gamma<\Gamma^{\infty}_{\nu,r} = h^{r}_{\nu}$, $f\in \widetilde{B}^{\xi}_{p,q}([0,1]^{D})$ and $r \in \R$.
    The inequality \eqref{eq:maj_Sj} with the infinity norm gives
    $S^{\infty}_{j}(a,\Gamma) := \sup_{\lambda^d\in \Lambda^{d}_j \times \{0,1\}^{d}} \big|\frac{d_{\lambda^d}(a)}{\mu(\lambda^d) 2^{-j\Gamma}}\big|$ is bounded above by $C_{\psi}\sup_{\substack{\lambda^{D}\in \Lambda^{D}_j \times \{0,1\}^{D} \\ \widetilde{K}\lambda^{d'} \cap \ A^{\nu_r}_{n,m} \neq \emptyset}} \big|\frac{c_{\lambda^D}}{\xi(\lambda^D)}\big| \big|\frac{\nu(\lambda^{d'})}{2^{-j\Gamma}}\big|$.
    By \eqref{eq:bound_meas_alm_all_dyad}, $S^{\infty}_{j}(a,\Gamma) \leq C_{\psi}~2^{-j(h_{\nu}^{r}-2/m-\Gamma)}~\sup_{\lambda^D\in \Lambda^{D}_j \times \{0,1\}^{D}} \big|\frac{c_{\lambda^D}}{\xi^{(-1/m)}(\lambda^D)}\big|$ for every $a\in A^{\nu_r}_{n,m}$.
    Then \eqref{eq:smaller_in_mean_of_Sj} is replaced by $S^{\infty}_{j}(a,\Gamma) \leq C~2^{-j(h_{\nu}^{r}-2/m-\Gamma)}$, which in turn changes \eqref{eq:A_p_j} into $A^{\nu_{r},\infty}_{n,m,j} := \big\{ a \in A^{\nu_{r}}_{n,m}~:~S^{p}_{j}(a,\Gamma) > C~j^{2}~2^{-j(h_{\nu}^{r}-2/m-\Gamma)}\big\}$.   
    Using $A^{\nu_{r},\infty}_{n,m,j}$ with $m$ large enough that $\Gamma < h_{\nu}^{r}-2/m$, the same Markov and Borel-Cantelli argument shows that for $\nu_{r}$-almost all $a\in[0,1]^{d'}$, $f_a \in B^{\mu^{(+\Gamma)}}_{\infty,1}([0,1]^d)$.
\end{proof}

We explain how Proposition \ref{prop:belong_trace_inter} implies Theorem \ref{theo:incl_trace}. Theorem \ref{theo:incl_trace} can be reworded as follows: for $f\in \widetilde{B}^{\xi}_{p,q}([0,1]^D)$ and $r\in\R$ fixed, there is a set $A_r(f)$ of full $\nu_r$-measure such that for every $a\in A_r(f)$,
\begin{equation*}
    \textnormal{for all } 0 \leq \Gamma < \Gamma^p_{\nu,r}, \quad f_a \in B^{\mu^{(+\Gamma)}}_{p,q}([0,1]^d).
\end{equation*}
    
\begin{proof}[Proof of Theorem \ref{theo:incl_trace}]
    Fix $p\in [1,+\infty]$ and $f\in\widetilde{B}^{\xi}_{p,q}([0,1]^{D})$. 
    For every $r\in\R$ and $n\in\N^*$, by Proposition \ref{prop:belong_trace_inter} for $\Gamma^{p}_{\nu,r} - 1/n$, call $A_{r,n} \subset [0,1]^{d'}$ a set of full $\nu_r$-measure such that for every $a\in A_{r,n}$, $f_a \in B^{\mu^{(+\Gamma^{p}_{\nu,r} - 1/n)}}_{p,1}([0,1]^d)$.
    Since the sets $B^{\mu^{(+\Gamma)}}_{p,1}([0,1]^d)$ are decreasing with $\Gamma$, taking the countable intersection $\bigcap_{n\in\N^*} A_{r,n}$ of full $\nu_r$-measure gives that for $\nu_{r}$-almost all $a\in[0,1]^{d'}$,
    \begin{equation*}
        f_a \in \bigcap_{0 \leq \Gamma < \Gamma^p_{\nu,r}} B^{\mu^{(+\Gamma)}}_{p,1}([0,1]^d) = \widetilde{B}^{\mu^{(+\Gamma^p_{\nu,r})}}_{p,1}([0,1]^d) \subset \widetilde{B}^{\mu^{(+\Gamma^p_{\nu,r})}}_{p,q}([0,1]^d) \subset \widetilde{B}^{\mu^{(+\Gamma^p_{\nu,r})}}_{p,\infty}([0,1]^d).
    \end{equation*}
\end{proof}

As a consequence of the previous proof, one gets the following proposition.
\begin{prop}\label{prop:inclusion_trace}
    For every $f\in \widetilde{B}^{\xi}_{p,q}([0,1]^{D})$, every $a\in[0,1]^{d'}$, $f_a \in \widetilde{B}^{\mu^{(+h^{\min}_{\nu})}}_{p,1}([0,1]^d)$.
\end{prop}
\begin{proof}
    Let $m > \partent{ {2}/{h^{\min}_{\nu}}}$. Using \eqref{eq:bound_hmin_hmax} and \eqref{eq:maj_Sj}, there exists a positive sequence $(\eta_j)_{j\N}$ tending to 0 such that
    \begin{align*}
        S^{p}_{j}(a,h^{\min}_{\nu} - 2/m)
        &\leq C_{p,\psi}\sum_{\lambda^D\in \Lambda^{D}_j \times \{0,1\}^{D}} \bigg|\frac{c_{\lambda^D}}{\xi^{(-1/m)}(\lambda^D)}\bigg|^p \big|2^{-j(1/m - \eta_j)}\big|^p \big|\psi_{\lambda^{d'}}(a)\big|^p \\
        &\leq C\ 2^{-jp(1/m - \eta_j)} (\varepsilon^{\xi^{(-1/m)}}_{j} )^p.
    \end{align*}
    Let $J_m$ be such that for all $j \geq J_{m}$, $\eta_j < 1/m$. As $f\in \widetilde{B}^{\xi}_{p,q}([0,1]^{D})$, the sequence $\big(\varepsilon^{\xi^{(-1/m)}}_{j}\big)_j$ belongs to $l^q(\N)$ and is bounded. Hence $(S_j(a,h^{\min}_{\nu} - 2/m)^{1/p})_{j}\in l^1(\N)$. One deduces that $f_a \in B^{\mu^{(+h^{\min}_{\nu} - 1/m)}}_{p,1}([0,1]^d)$. This holds for all $m$, so $f_a \in \widetilde{B}^{\mu^{(+h^{\min}_{\nu})}}_{p,1}([0,1]^d)$. 
\end{proof}

As a consequence, for every $f\in \widetilde{B}^{\xi}_{p,q}([0,1]^{D})$ every $a\in[0,1]^{d'}$, $f_a$ can be written 
\begin{equation}\label{eq:dec_trace_2}
    f_a = c_{0,0^d,0^d}(a) \psi_{0,0^d,0^d}~+ \sum_{j\geq 0} \sum_{\lambda^d_j \in \Lambda^d_j \times L^d} c_{\lambda^d_j}(a) \psi_{\lambda^d_j}.
\end{equation}

\begin{rem}[Link to the standard Besov spaces]
    Let $\mathcal{L}^{d}$ be the $d$-dimensional Lebesgue measure. Let $D=d+d'$, $0<p,s<\infty$, with $s> {D}/{p}$, and $q\in(0,p)$. Consider $\mu= \big( \mathcal{L}^{d} \big)^{ {s}/{D}- {1}/{p}}$ and $\nu= \big( \mathcal{L}^{d'} \big)^{ {s}/{D}- {1}/{p}}$. and $\xi = \mu \times \nu = \big( \mathcal{L}^{D} \big)^{ {s}/{D}- {1}/{p}}$. One immediately checks that $B^{\xi}_{p,q} = B^{s}_{p,q}$. We recover the result of \cite{Aubry-Maman-Seuret:2013:Traces_besov_results} by applying Theorem \ref{theo:incl_trace} to these capacities, i.e.:
    
    \textit{Let $D\in\N^*$, $0<p,s<\infty$, with $s- {D}/{p}>0$, and let $q\in(0,\infty]$. If $f\in B^{s}_{p,q}([0,1]^{D})$, then for Lebesgue-almost all $a\in [0,1]^{d'}$,} $f_a\in \widetilde{B}^{s}_{p,1}([0,1]^{d})=\bigcap_{0<\varepsilon< {d}/{p}} B^{s-\varepsilon}_{p,1}([0,1]^{d})$.

    Indeed, for all $r\in\N$, $\nu_r=\mathcal{L}^{d'}$ and $\Gamma_{\nu,r}=d' \big( {s}/{D}- {1}/{p} \big) + {d'}/{p}$. One can check that $\mu^{(+\Gamma_{\nu,r})} = \big( \mathcal{L}^{d} \big)^{ {s}/{d}- {1}/{p}}$ giving $\bigcap_{0<\Gamma<\Gamma_{\nu,r}} B^{\mu^{(+\Gamma)}}_{p,1}([0,1]^{d}) = \bigcap_{0<\varepsilon< {d}/{p}} B^{s-\varepsilon}_{p,1}([0,1]^{d})$.
\end{rem}

\section{Upper bound for the spectrum of traces of \texorpdfstring{$\widetilde{B}^{\xi}_{\infty,q}$}{Btilde xi inf,q} functions}\label{sec:upper_bound_spectrum}

Let $\mu\in\mathscr{E}_d$ on $[0,1]^d$ and $\nu\in\mathscr{E}_{d'}$ on $[0,1]^{d'}$ with $D=d+d'$. 
Let $\xi$ be the product capacity on $[0,1]^D$, $\xi := \mu \times \nu$. 
The local behavior of elements of $B^{\xi}_{p,q}([0,1]^D)$ and $\widetilde{B}^{\xi}_{p,q}([0,1]^D)$ is described via their pointwise H{\"o}lder exponent, characterised by wavelet leaders \cite{Jaffard:2007:Wavelet_leaders}.

\begin{defi}[Wavelet Leaders]\label{defi:wlt_lead}
    Let $(\psi^l)_{l\in L^{D}}$ be a family of wavelets that form an orthogonal basis of $L^2(\R^D)$. For $f\in L^p_{\operatorname{loc}}(\R^D)$ with $p\in[1,+\infty]$, call $(c_{\lambda^D})_{\lambda^D\in\Lambda^D \times L^D}$ the wavelet coefficients of $f$. The wavelet leader of $f$ associated with $\lambda^D\in\Lambda^D$ is defined as 
    \begin{equation}
        L^{f}_{\lambda^D} = \sup \{ |c_{\lambda'}| : \lambda'=(j,k,l)\in \Lambda^D \times L^D, \lambda_{j,k}^D \subset 3 \lambda^{D}\}.
    \end{equation}
\end{defi}
 
\begin{nota}
    Consider $f\in \widetilde{B}^{\xi}_{\infty,q}([0,1]^D)$ and, for $a\in[0,1]^{d'}$, the associated trace $f_a$ with wavelet coefficients denoted by $(c_{\lambda^{d}}(a))_{\lambda^{d}\in\Lambda^{d} \times L^{d}}$. Following Definition \ref{defi:wlt_lead}, we denote the wavelet leader of $f_a$, the trace of $f$ in $a$, associated with $\lambda^d\in\Lambda^d$ by (notice the dependance in $a$)
    \begin{equation}\label{eq:wlt_lead_trace}
        L^{f}_{\lambda^d}(a) = \sup \{|d_{\lambda}(a)| : \lambda=(j,k,l)\in \Lambda^d \times L^d, \lambda^{d}_{j,k} \subset 3 \lambda^d \}.
    \end{equation}
    The mention of $f$ on leaders is omitted when the function is obvious from the context.
\end{nota}

\noindent The pointwise H{\"o}lder exponent \eqref{eq:hold_exp} is characterised by wavelet leaders \cite[Prop. 4]{Jaffard:2007:Wavelet_leaders}.

\begin{prop}\label{prop:hold_exp_wlt}
    Let $N\in\N^*$ and assume that the wavelets belong to $C^{N}(\R^D)$ with at least $N+1$ vanishing moments. If $f \in\mathcal{C}^{\varepsilon}(R^D)$ for some $\varepsilon>0$, then for every $x_0\in\R^D$, $h_f(x_0)<N$ if and only if $\liminf_{j\to +\infty} \frac{\log L^{f}_{\lambda_j^D(x_0)}}{\log (2^{-j})} <N$, and in this case 
    \begin{equation*}
        h_{f}(x_0)= \liminf_{j\to +\infty} \frac{\log L^{f}_{\lambda_j^D(x_0)}}{\log (2^{-j})}.
    \end{equation*}
\end{prop}
\noindent Let us start with an intermediate result key for proving Theorem \ref{theo:up_bound_spect_all_fct} in Sections \ref{sec:sat_fct} and \ref{sec:proof_prev_spect}.

\begin{lemme}\label{lemma:maj_exp_besov_inf_inf_all_a}
    For all $f\in \widetilde{B}^{\xi}_{\infty,q}$, for every $x\in[0,1]^{d}$ and $a\in[0,1]^{d'}$, $h_{f_a}(x) \geq \underline{h}_{\mu}(x) + h_{\nu}^{\min}$.
\end{lemme}
\begin{proof}
    Let $f\in \widetilde{B}^{\xi}_{\infty,q}$ and $a\in[0,1]^{d'}$. Consider the wavelet decomposition \eqref{eq:dec_trace_2} of $f_a$
    \begin{equation}
        f_a = c_{0,0^d,0^d}(a) \psi_{0,0^d,0^d} + \sum_{\lambda^d \in \Lambda^d\times L^d} c_{\lambda^d}(a) \psi_{\lambda^d}.
    \end{equation}
    Pay attention that $c_{\lambda^d}(a)$ does not have an explicit formula like for $d^G_{\lambda^d}(a)$ in \eqref{eq:dec_trace_G_a} and $d^F_{\lambda^d}(a)$ in \eqref{eq:dec_trace_F_a}.
    Using Proposition \ref{prop:inclusion_trace}, one has $f_a \in \widetilde{B}^{\mu^{(+h^{\min}_{\nu})}}_{\infty,1}([0,1]^d)$. 
    Then, for every $m \geq \partent{ {1}/{h^{\min}_{\nu}}}$, there is $C>0$ such that $\sup_{j\in\N}\sup_{\lambda^d \in \Lambda_j^d\times L^d} \Big| {\frac{c_{\lambda^d}(a)}{\mu^{(+h^{\min}_{\nu}-1/m)}(\lambda^d)}} \Big| \leq C. $
    As a consequence, for every $x \in [0,1]^d$,
    \begin{align*}
        L^f_{\lambda^{d}_{j}(x)}(a)
        &= \sup_{\lambda \subset 3 \lambda^{d}_{j}(x)} \big|c_{\lambda}(a)\big|
        \leq C\sup_{\lambda_{j'} \subset 3 \lambda^{d}_{j}(x)} \big|\mu(\lambda_{j'})\ 2^{-j' (h^{\min}_{\nu}-\frac{1}{m})}\big| \\
        &\leq C \big|\mu(3\lambda^{d}_{j}(x))\ 2^{-j (h^{\min}_{\nu}-\frac{1}{m})}\big|.
    \end{align*}
     The doubling property \eqref{eq:doub_meas_ball} of $\mu$ yields, for a constant $C>0$,
    \begin{equation}\label{eq:maj_L_ja}
        L^f_{\lambda^{d}_{j}(x)}(a) \leq C \big|\mu(\lambda^{d}_{j}(x))\ 2^{-j (h^{\min}_{\nu}-\frac{1}{m})}\big|.
    \end{equation}
    Hence, for all $j\in\N$, $
        \log L^f_{\lambda^{d}_{j}(x)}(a) \leq \log C + \log \mu(\lambda^{d}_{j}(x)) + \log\big(2^{-j (h^{\min}_{\nu}- {1}/{m})}\big)$. Dividing by $\log(2^{-j})$ and taking the $\liminf$ on both sides gives $h_{f_a}(x) \geq \underline{h}_{\mu}(x) + h^{\min}_{\nu} - 1/m$.
    This holds for every large $m$, hence the conclusion.
\end{proof}

\begin{lemme}\label{lemma:maj_exp_besov_inf_inf_alm_all_a}
    Let $r\in\R$. For all $f\in \widetilde{B}^{\xi}_{\infty,q}$, for every $x\in[0,1]^{d}$ and $\nu_r$-almost all $a\in[0,1]^{d'}$, $h_{f_a}(x) \geq \underline{h}_{\mu}(x) + h^{r}_{\nu}$.
\end{lemme}
\begin{proof}
    The arguments are identical to those of Lemma \ref{lemma:maj_exp_besov_inf_inf_all_a}. Theorem \ref{theo:incl_trace} gives for $\nu_{r}$-almost all $a\in[0,1]^{d'}$, $f_a \in \widetilde{B}^{\mu^{(+h^{r}_{\nu})}}_{\infty,1}([0,1]^d)$.
    Then, for every $m \geq \partent{ {1}/{h^{r}_{\nu}}}$, there exists a constant $C>0$ such that
    \begin{equation*}
        \sup_{j\in\N}\sup_{\lambda^d \in \Lambda_j^d\times L^d} \bigg|\frac{c_{\lambda^d}(a)}{\mu^{( +(h^{r}_{\nu}-1/m) )}(\lambda^d)}\bigg| \leq C.
    \end{equation*}
    Thus, for all $x \in [0,1]^d$, as in \eqref{eq:maj_L_ja},
    \begin{equation*}
        L^f_{\lambda^{d}_{j}(x)}(a)
        = \sup_{\lambda \subset 3 \lambda^{d}_{j}(x)} \big|c_{\lambda}(a)\big|
        \leq C \hspace{-5pt}\sup_{\lambda_{j'} \subset 3 \lambda^{d}_{j}(x)} \big|\mu(\lambda_{j'}) 2^{-j' (h^{r}_{\nu}-\frac{1}{m})}\big|
        \leq C \big|\mu(\lambda^{d}_{j}(x)) 2^{-j (h^{r}_{\nu}-\frac{1}{m}))}\big|.
    \end{equation*}
    As previously, one concludes that for $\nu_r$-almost all $a\in[0,1]^{d'}$, $h_{f_a}(x) \geq \underline{h}_{\mu}(x) + h^{r}_{\nu}$.
\end{proof}

We are now ready to prove Theorem \ref{theo:up_bound_spect_all_fct}.

\begin{proof}[Proof of Theorem \ref{theo:up_bound_spect_all_fct}]
    Considering a function in $\widetilde{B}^{\xi}_{\infty,q}$ without any further hypothesis, one can only get an upper bound of the spectrum of the trace of such a function. Having both Lemma \ref{lemma:maj_exp_besov_inf_inf_all_a} and \ref{lemma:maj_exp_besov_inf_inf_alm_all_a}, we now prove the upper bound for $\sigma_{f_a}$.
    
    Let us start by \eqref{eq:theo_up_b_spect:all_a}. Using Lemma \ref{lemma:maj_exp_besov_inf_inf_all_a}, one has that, for $a\in[0,1]^{d'}$ and $h\in\R^+$, $\{x \in \R^{d}: h_{f_a}(x) \leq h\} \subset \{x\in \R^{d}: \underline{h}_{\mu}(x) \leq h - h_{\nu}^{\min}\}$. 
    Hence
    \begin{equation*}
        E_{f_a}(h) \subset \{x \in \R^{d}: h_{f_a}(x) \leq h\} \subset \{x\in \R^{d}: \underline{h}_{\mu}(x) \leq h - h_{\nu}^{\min}\} = \underline{E}_{\mu}^{\leq}(h- h^{\min}_{\nu}).
    \end{equation*}
    When $h\leq h^{0}_{\nu}+h^{\min}_{\nu}$, item \ref{item:dim_E_leq} of Proposition \ref{prop:prop_capa_SMF} gives $\dimH\big(\underline{E}_{\mu}^{\leq}(h- h^{\min}_{\nu})\big) = \tau_{\mu}^{*}(h-h^{\min}_{\nu})$.
    As $\mu$ verifies the SMF, for every $h\leq h^{0}_{\nu}+h^{\min}_{\nu}$, one has $\tau_{\mu}^{*}(h-h^{\min}_{\nu}) = \sigma_{\mu}(h-h^{\min}_{\nu})$. 
    Moreover, for every $h > h^{0}_{\nu}+h^{\min}_{\nu}$, $\{x\in \R^{d}: \underline{h}_{\mu}(x) \leq h - h^{\min}_{\nu}\}$ has dimension $d$.
    This induces the result as $\sigma_{f_a}(h) = \dimH(E_{f_a}(h))$.

    Let us finish with the proof of \eqref{eq:theo_up_b_spect:almost_a}. 
    Let $f\in \widetilde{B}^{\xi}_{\infty,q}$ and $r\in\R$. By Lemma \ref{lemma:maj_exp_besov_inf_inf_alm_all_a}, for $\nu_r$-almost all $a\in[0,1]^{d'}$, one has, for all $h\in\R^+$,
    \begin{equation*}
        E_{f_a}(h) \subset \{x \in \R^{d}: h_{f_a}(x) \leq h\} \subset \{x\in \R^{d}: \underline{h}_{\mu}(x) \leq h - h^{r}_{\nu}\} = \underline{E}_{\mu}^{\leq}(h- h^{r}_{\nu}).
    \end{equation*}
    Then one concludes in the same way as for \eqref{eq:theo_up_b_spect:all_a}.
\end{proof}

\section{A saturating function in \texorpdfstring{$\widetilde{B}^{\mu}_{\infty,q}$}{B mu inf,q}}\label{sec:sat_fct}

From now on, we focus on $p=+\infty$.
For the rest of the paper, the wavelet family considered will be constructed as in Section \ref{sec:mr_wlt_ana} with $\psi^{1}$ verifying the property $(R)$ of Definition \ref{defi:prop_R}.

In order to compute the singularity spectrum of $\mathscr{G}_a$, this property (R) of the wavelet $\psi$ is necessary. Indeed, in \cite{Rible-Seuret-1}, the following theorem is proved.

\begin{theo}\label{theo:low_bound_1per_wlt}
    Let $\psi:\R\to\R $ satisfy property $(R)$, with $\operatorname{supp}(\psi)=[0,K_{\psi}]$. Let
    \begin{equation}\label{eq:defi_G}
        G:x\in\R \mapsto \sum_{n\in\Z}\psi(x-nK_{\psi}).
    \end{equation}
    There exists $\alpha>0$ such that , for every $d'\in \N^*$, there exists an integer $N(d')$ and a sequence of multi-integer $(\overline{p_j}=(p_{j,1},\ldots,p_{j,d'}))_{j\in\N}\in(\{0,\ldots,K_{\psi}-1\}^{d'})^{\N}$ enjoying the following property: for every $J\equiv 0 \pmod{N(d')}$ and $x\in\R^{d'}$, there is an integer $\ell\in\{0,\ldots,N(d')-1\}$ that verifies $\prod_{i=1}^{d'} |G(2^{J+\ell} x_i - p_{J+\ell,i})| \geq \alpha^{d'}$.
\end{theo}
In other words, (R) ensures that dilated and translated transforms of tensor products of periodised versions of $\psi$ cannot all vanish at the same point, and even more, they are non-zero very regularly. This is key to prove that the traces of well-chosen \textit{saturating} functions cannot be zero too often, and is key is our proof of the prevalence Theorem \ref{theo:prev_spect_inh_besov_inf_q}.

Consider the integer $N(d')$ and the sequence $(\overline{p_j}=(p_{j,1},\ldots,p_{j,d'}))_{j\in\N}$ given by Theorem \ref{theo:low_bound_1per_wlt}.
\begin{defi}\label{defi:Gd}
    The map $G_{j}^{d'}:\R^{d'} \to \R$ is given by
    \begin{equation}\label{eq:Gd}
        \textnormal{for every } x=(x_1,...,x_{d'})\in \R^{d'},~~ G_{j}^{d'}(x) = \prod_{i=1}^{d'} G(2^j x_i - p_{j,i}).
    \end{equation}
\end{defi}

To establish generic or prevalent properties of traces of function in $\widetilde{B}^{\xi}_{\infty,q}([0,1]^D)$, a classical approach consists in perturbing a set of functions by a \textit{saturating function} in $\widetilde{B}^{\xi}_{\infty,q}([0,1]^D)$ whose multifractal structure exhibits the property claimed to be generic. 
\subsection{Construction of the saturating function \texorpdfstring{$\mathscr{G}_q$}{Gq}}\label{subsec:generic_spectrum:definition_saturationg_function}

\begin{defi}\label{defi:saturating_func_besov_inf_inf}
    Let us define, for every $\lambda_0^D\in\Lambda^D_0 \times \{0,1\}^D$, $g_{\lambda_0^D}=0$, and for every $\lambda_j^D=(j,(k,k'),(l,l'))\in\Lambda^D_j \times L^D$ with $j\geq1$,
    \begin{equation}\label{eq:coef_saturating_func}
        g_{\lambda_j^D} = \begin{cases}
            \displaystyle {j^{-\frac{2}{q}}}~\xi(\lambda^{d}_{j,k}\times\lambda^{d'}_{j,k'}) \sum_{n\in\Z^{d'}} \delta^{k'}_{p_{j}+nK_{\psi}} &\text{if } l\neq 0^d \text{ and } l'=1^{d'},\\
            0 &\text{if } l= 0^d \text{ or } l'\neq1^{d'},
        \end{cases}
    \end{equation}
    with the convention $ {2}/{\infty}=0$.
    The saturating function is $\mathscr{G}_q := \sum_{\lambda^D \in\Lambda^D\times L^D} g_{\lambda^D} \psi_{\lambda^D}$.
\end{defi}
 
The idea is that $\mathscr{G}_q$ has the largest possible wavelet coefficients to maximise $\varepsilon^{\xi,\infty}_{j}$ in \eqref{eq:snorm_inh_besov}, while ensuring that the non-zero wavelet coefficients are sparse and well-separated, to easily compute the wavelet coefficients of its traces.

\begin{prop}
    The function $\mathscr{G}_q$ belongs to $\widetilde{B}^{\xi}_{\infty,q}([0,1]^D)$.
\end{prop}
\begin{proof}
    First, since $\xi$ is $s$-H{\"o}lder, one has
    \begin{equation*}
        \big\|\mathscr{G}_{q}\big\|_{L^{\infty}}
        \leq \sum_{j\geq 1} \bigg\|\sum_{\lambda_j^d\in \Lambda_j^D\times L^D} g_{\lambda_j^D} \psi_{\lambda_j^D}\bigg\|_{L^{\infty}}
        \leq C \sum_{j\geq 1} 2^{-js}\bigg\|\bigg(\frac{g_{\lambda_j^D}}{\xi(\lambda_j^D)}\bigg)\bigg\|_{\ell^{\infty}}
        \leq C \sum_{j\geq 1} 2^{-js} \varepsilon_{j}^{\xi,\infty}.
    \end{equation*}
    Hence $\mathscr{G}_q$ is a bounded function. Next, recall \eqref{eq:snorm_inh_besov}. As for every $j\geq1$ and $\lambda_j^D\in\Lambda^D_j \times L^D$, $g_{\lambda_j^D}\leq j^{- {2}/{q}} \xi(\lambda^D)$, one has $\varepsilon^{\xi,\infty}_{j} = \Big\|\Big(\frac{g_{\lambda^D}}{\xi(\lambda^D)}\Big)_{\lambda^D}\Big\|_{\ell^{\infty}(\Lambda^D_j \times L^D)}\leq j^{-2/q}$.
    This implies that $\big\|\mathscr{G}_q\big\|_{B^{\xi}_{\infty,q}} < \infty$ and $\mathscr{G}_q \in B^{\xi}_{\infty,q}([0,1]^D) \subset \widetilde{B}^{\xi}_{\infty,q}([0,1]^D)$.
\end{proof}

For $a=(a_1,\ldots,a_{d'})\in[0,1]^{d'}$, the trace $\mathscr{G}_{q,a}:x\in[0,1]^d \mapsto \mathscr{G}_q(x,a) \in \R$ has a convenient wavelet decomposition. Indeed, because of our choice \eqref{eq:coef_saturating_func}, for the trace $\mathscr{G}_{q,a}$, all the coefficients \eqref{eq:dec_trace_G_a} are zero since $g_{\lambda_j^D}=0$ when $\lambda_j^D=\lambda_j^d\times\lambda_j^{d'}$ with $\lambda_j^d\in \Lambda_j^d\times \textbf{0}^d$. Hence, only the coefficients \eqref{eq:dec_trace_F_a} remain. Writing then
\begin{equation*}
    \mathscr{G}_{q,a}(x) = \sum_{j\geq 0} \sum_{\lambda_j^{d}\in \Lambda^d_j\times L^d} e_{\lambda_j^d}(a) \psi_{\lambda_j^{d}}(x) ,
\end{equation*}
the coefficients of \eqref{eq:dec_trace_F_a} can be written (recall \eqref{eq:dec_trace_F_a}) as
\begin{equation}\label{eq:def-coef-trace-1}
  e_{\lambda_j^d}(a)= \sum_{ {\lambda^{d'}\in\Lambda^{d'}_j\times \mathbf{1^{d'}} }} g_{\lambda_j^ d\times\lambda_j ^{d'}}\psi_{\lambda^{d'}_j}(a).
\end{equation}
In addition, these remaining coefficients have a fairly simple expression. To see this, for every $\lambda_j^d=(j,k_d,l_d)\in\Lambda_j^d\times L^d$, consider the integer $k_j(a)$ defined by  
\begin{equation}\label{eq:def_kja}
    k_j(a) := \Big(p_{j,1}+\Big\lfloor{\frac{2^j a_1-p_{j,1}}{K_{\psi}}}\Big\rfloor K_{\psi},\ldots,p_{j,d'}+\Big\lfloor{{\frac{2^j a_{d'}-p_{j,d'}}{K_{\psi}}}}\Big\rfloor K_{\psi} \Big) .
\end{equation}
Since $\psi_{\lambda_j^{d'}}$ has a support of length $N_\psi 2^{-j}$, and since the coefficients $g_{\lambda_{j,k}^ d\times\lambda_{j,k'} ^{d'}}$ are zero except when $k'=p_j \mod N_\psi$, the only non-zero wavelet coefficient $g_{\lambda_j^D}$ in the sum \eqref{eq:def-coef-trace-1} is the one at $\lambda_j^D=(j,(k_d,k_j(a)),(l_d,1^{d'}))$. Observe that as soon as $j$ is so large that $N_\psi2^{-j} \leq a/2$, then $ k_j(a) $ has strictly positive entries. One concludes that $a\in (0,1]^{d'}$ being fixed, when $j$ is large, one has  
\begin{equation}\label{eq:definition_trace_func_saturation}
     e_{\lambda_j^d}(a) := {j^{- {2}/{q}}}\mu(\lambda_j^d) \nu\big(\lambda_{j,k_j(a)}^{d'}\big) G^{d'}_{j}(a).   
\end{equation}
This formula is key for the rest of the paper, and justifies our choice for the saturation function $\mathcal{G}_q$. One can now understand the relevance of Theorem \ref{theo:low_bound_1per_wlt}, which allows us to to ensure that the coefficients $|e_{\lambda_j^d}(a)|$ are not too small - this will be used to bound by above the pointwise H\"older exponent of the traces. 

Regarding the multifractal properties, Theorem \ref{theo:incl_trace} applied to $\mathscr{G}_q$ gives:
\begin{coro}
    For every $r\in\R$, for $\nu_{r}$-almost all $a\in[0,1]^{d'}$, $\mathscr{G}_{q,a} \in \widetilde{B}^{\mu^{(+h_{\nu}^{r})}}_{\infty,q}([0,1]^d)$.
\end{coro}

\subsection{Singularity spectrum of the traces \texorpdfstring{$\mathscr{G}_{q,a}$}{G(q-a)} of the saturating function}\label{subsubsec:spectrum_sat_fct}

\begin{prop}\label{prop:equality_exponent_and_spectrum_sat_fct_Besov_inf_inf}
    For every $r\in\R$, for every $x\in[0,1]^d$ and for $\nu_{r}$-almost all $a\in(0,1]^{d'}$,
    \begin{equation}\label{eq:prop:equality_exponent_saturating_func}
        h_{\mathscr{G}_{a}}(x) = \underline{h}_{\mu}(x) + h^{r}_{\nu} = \underline{h}_{\mu^{(+h^{r}_{\nu})}}(x).
    \end{equation}
    Hence
    \begin{equation}\label{eq:prop:equality_spectrum_saturating_func}
        \sigma_{\mathscr{G}_{a}} = \tau^*_{\mu^{(+h^{r}_{\nu})},\infty} = \tau^*_{\mu}(\cdot-h^{r}_{\nu}).
    \end{equation}
\end{prop}
\begin{proof}
    Let $m,n\in\N^*$ be such that $m \geq M_{\nu_{r}} := \partent{ {1}/{h^{r}_{\nu}}}+1$. 
    Let $A^{\nu_r}_{n,m}$ and $J_{n,m}$ be the set and integer obtained in Proposition \ref{prop:bound_meas_dyadics}. Let $a\in A^{\nu_r}_{n,m}\cap (0,1]^{d'}$. By \eqref{eq:bound_meas_alm_all_dyad}, one has $2^{-j(h^{r}_{\nu}+1/m)} \leq \nu(\lambda_j^{d'}) \leq 2^{-j(h^{r}_{\nu}-1/m)}$ for all $\lambda_j^{d'}\in\Lambda_j^{d'}$ such that $a\in (2 K_{\psi}) \lambda_j^{d'}$ and $j\geq J_{n,m}$.
    
    We first bound above $L_{\lambda^{d}}(a)$ \eqref{eq:wlt_lead_trace} of $\mathscr{G}_{q,a}$.
    Consider $j\equiv 0 \pmod{N(d')}$. Using \eqref{eq:definition_trace_func_saturation}, for $\lambda_j^{d}=(j,k)\in\Lambda_j^{d}$, and that $|G^{d'}_{j'}(a)| \leq \max\{\|\psi\|_{L^{\infty}}^{d'},1\}$, one has
    \begin{align*}
        L_{\lambda^{d}_j}(a) 
        &= \sup_{j'\geq j} \sup_{ {\Lambda_{j'}^{d}\ni \lambda^{d}_{j'}\subset 3 \lambda^d_j }} \bigg|\frac{1}{(j')^{2/q}}\mu(\lambda_{j'}^{d})\nu\big(\lambda^{d'}_{j',k_{j'}(a)}\big) G^{d'}_{j'}(a)\bigg| \\
        &\leq \sup_{j'\geq j} \sup_{ {\Lambda_{j'}^{d}\ni \lambda^{d}_{j'}\subset 3 \lambda^d_j }} \big|{\mu(\lambda^{d}_{j'})} 2^{-j'(h^{r}_{\nu}-1/m)}\big| \max\{\|\psi\|_{L^{\infty}}^{d'},1\} \\
        &\leq C 2^{-j(h^{r}_{\nu}-1/m)} \big|{\mu(3\lambda^{d}_{j})}\big|.  
    \end{align*}
    For every $x\in\R^d$, since $3\lambda^{d}_{j}(x) \subset B(x,2\times2^{-j})$, one has that $L_{\lambda^{d}_{j}(x)}(a) \leq C 2^{-j(h^{r}_{\nu}-1/m)} \allowbreak \mu(B(x,2\times2^{-j}))$.
    The doubling property \eqref{eq:doub_meas_ball} of $\mu$ then yields for some $C>0$ depending on $\mu,\psi$ and $d'$
    \begin{equation}\label{eq:leader_saturation_function_upper_bound}
        L_{\lambda^{d}_{j}(x)}(a) \leq C 2^{-j(h^{r}_{\nu}-1/m)} \mu(B(x,2^{-j})).
    \end{equation}
    We now bound by below $L_{\lambda^{d}}(a)$. Consider $j$ so large that $k_j(a)$ contains only positive integers. 
    Let $x\in[0,1]^d$. By Theorem \ref{theo:low_bound_1per_wlt}, for every $j\equiv 0 \mod N(d')$, there is $J\in\{0,\ldots,N(d')-1\}$ such that $\big|G_{j+J}^{d'}(x)\big| \geq \alpha^{d'}$. By \eqref{eq:definition_trace_func_saturation} and \eqref{eq:bound_meas_alm_all_dyad} one deduces that
    \begin{align*}
        L_{\lambda^{d}_{j}(x)}(a) 
        \geq \frac{1}{(j+J)^{2/q}} \mu(\lambda^{d}_{j+J}(x)) \nu (\lambda^{d'}_{j+J,k_{j+J}(a)} ) \big|G^{d'}_{j+J}(a)\big|\\
        \geq 2^{-2\log_2(j+J)/q} \alpha^{d'} 2^{-(j+J)(h^{r}_{\nu}+1/m)} \mu(\lambda^{d}_{j+J}(x)).
    \end{align*}
    Denoting by $y_{\lambda_{j}}$ the center of the cube $\lambda_{j}$, observe that $\lambda^{d}_{j+J}(x) = B (y_{\lambda^{d}_{j+J}(x)},2^{-(j+J+1)} )$. Hence, using the constant $C_{\mu}$ appearing in \eqref{eq:doub_meas_ball}, one gets
    \begin{align*}
        L_{\lambda^{d}_{j}(x)}(a)
        &\geq 2^{-2\log_2(j+N(d'))/q} \alpha^{d'} 2^{-N(d')(h^{r}_{\nu}+1/m)} 2^{-j(h^{r}_{\nu}+1/m)} \mu\big(B(y_{\lambda_{j+J}(x)},2^{-(j+J+1)})\big) \\
        &\geq C 2^{-2\log_2(j+N(d'))/q} 2^{-j(h^{r}_{\nu}+1/m)} \frac{1}{C_{\mu}^{J+2}} \mu\big(B(y_{\lambda_{j+J}(x)},2^{-j+1})\big) \\
        &\geq C 2^{-j \big(h^{r}_{\nu}+\frac{2}{q}\frac{\log_2(j+N(d'))}{j}+1/m \big)} \mu\big(B(x,2^{-j})\big). \numberthis\label{eq:leader_saturation_function_lower_bound}
    \end{align*}
    Last equation \eqref{eq:leader_saturation_function_lower_bound} holds only when $j=0 \mod N(d')$. Since the wavelet leaders $L_{\lambda^{d}_{j}(x)}(a)$ are decreasing in $j$, the liminf of the above quantities when $j\to +\infty$ equals the liminf restricted to $j=0 \mod N(d')$. Hence, Proposition \ref{prop:hold_exp_wlt} combined with \eqref{eq:leader_saturation_function_upper_bound}, \eqref{eq:leader_saturation_function_lower_bound} and the lower local dimension $\underline{h}_{\mu}(x)$ together yield that
    \begin{equation*}
        (h^{r}_{\nu}-1/m) + \underline{h}_{\mu}(x) \leq h_{\mathscr{G}_{a}}(x) \leq (h^{r}_{\nu}+1/m) + \underline{h}_{\mu}(x).
    \end{equation*}
    These inequalities hold for $a\in A^{\nu_r}_{n,m}\cap(0,1]^{d'}$ and remain true for every $a\in \bigcup_{n\in\N^*} A^{\nu_r}_{n,m}\cap(0,1]^{d'}$ and $m \geq M_{\nu_{r}}$ as they do not depend on $n$. Using that the sets $\bigcup_{n\in\N^*} A^{\nu_r}_{n,m}$ are of full $\nu_r$-measure for all $m$, $\big(\bigcap_{m \geq M_{\nu_r}} \bigcup_{n\in\N^*} A^{\nu_r}_{n,m}\big)\cap(0,1]^{d'}$ is of full $\nu_r$-measure as a countable intersection of full $\nu_r$-measure sets. Hence for every $a\in \big(\bigcap_{m \geq M_{\nu_r}} \bigcup_{n\in\N^*} A^{\nu_r}_{n,m}\big)\cap(0,1]^{d'}$,
    \begin{equation}
        h_{\mathscr{G}_{a}}(x) = \underline{h}_{\mu}(x) + h^{r}_{\nu}.
    \end{equation}
    Then by definition of $\mu^{(+h^{r}_{\nu})}$, $\underline{h}_{\mu^{(+h^{r}_{\nu})}}(x) = \underline{h}_{\mu}(x) + h^{r}_{\nu}$ giving \eqref{eq:prop:equality_exponent_saturating_func}. 
    Since $\mu$ satisfies the SMF, one concludes that for $\nu_{r}$-almost all $a\in(0,1]^{d'}$, $\sigma_{\mathscr{G}_{a}} = \sigma_{\mu^{(+h^{r}_{\nu})}} = \tau^*_{\mu}(\cdot-h^{r}_{\nu})$.
\end{proof}

\section{Prevalent singularity spectrum of traces}\label{sec:proof_prev_spect}

\subsection{Prevalence property of an auxiliary set}

The key result to obtain the prevalent value of the singularity spectrum of functions in $\widetilde{B}^{\xi}_{\infty,q}([0,1]^D)$ in Theorem \ref{theo:prev_spect_inh_besov_inf_q} is the following:

\begin{prop}\label{prop:prevalent_set_in_B_xi_inf_q} 
    Let $r\in \R$. The set
    \begin{multline}
        \mathcal{F}^{r} = \big\{ f\in \widetilde{B}^{\xi}_{\infty,q}([0,1]^D) : \exists A(f) \text{ of full $\nu_r$-measure in } (0,1]^{d'} \big. \\
        \big. \text{such that } a\in A(f) \implies \forall x \in [0,1]^{d},~h_{f_{a}}(x) \leq \overline{h}_{\mu}(x) + h^{r}_{\nu} \big\}
    \end{multline}
    is prevalent in $\widetilde{B}^{\xi}_{\infty,q}([0,1]^D)$.
\end{prop}

The proof of Proposition \ref{prop:prevalent_set_in_B_xi_inf_q} is postponed to Sections \ref{subsec:analyticity_O_nm} and \ref{subsec:shyness_O_nm}. For the moment, we admit it and explain how to deduce Theorem \ref{theo:prev_spect_inh_besov_inf_q}. 

\begin{proof}[Proof of Theorem \ref{theo:prev_spect_inh_besov_inf_q}]
    Using Lemma \ref{lemma:maj_exp_besov_inf_inf_alm_all_a} and Proposition \ref{prop:prevalent_set_in_B_xi_inf_q}, for all $f\in\mathcal{F}^{r}$, there is a set $A(f)$ of full $\nu_r$-measure in $(0,1]^{d'}$ such that for every $x\in[0,1]^{d}$, 
    \begin{equation}\label{eq:theo:prevalent_spectrum:bound_h_f}
        \underline{h}_{\mu}(x) + h^{r}_{\nu} \leq h_{f_{a}}(x) \leq \overline{h}_{\mu}(x) + h^{r}_{\nu}.
    \end{equation}
    
    Let us begin with the upper bound $\sigma_{f_{a}} \leq \sigma_{\mu}(\cdot - h^{r}_{\nu})$. 
    For all $h \in\R_{+}$,
    \begin{align*}
        E_{f_{a}}(h) &\subset \{x \in \R^{d}~:~h_{f_{a}}(x) \leq h\} \subset \{x \in \R^{d}~:~\underline{h}_{\mu}(x) + h^{r}_{\nu} \leq h\} = \underline{E}_{\mu}^{\leq}(h - h^{r}_{\nu}), \\
        E_{f_{a}}(h) &\subset \{x \in \R^{d}~:~h_{f_{a}}(x) \geq h\} \subset \{x \in \R^{d}~:~\overline{h}_{\mu}(x) + h^{r}_{\nu} \geq h\} = \overline{E}_{\mu}^{\geq}(h - h^{r}_{\nu}).
    \end{align*}
    Items \ref{item:dim_E_leq} and \ref{item:dim_E_geq} of Proposition \ref{prop:prop_capa_SMF} give:
    \begin{itemize}[noitemsep]
        \item for every $h \leq \tau'_{\mu}(0) + h^{r}_{\nu}$, $\dimH\big(\underline{E}_{\mu}^{\leq}(h - h^{r}_{\nu})\big) = \tau_{\mu}^{*}(h - h^{r}_{\nu})$,
        \item for every $h \geq \tau'_{\mu}(0) + h^{r}_{\nu}$, $\dimH\big(\overline{E}_{\mu}^{\geq}(h - h^{r}_{\nu})\big) = \tau_{\mu}^{*}(h - h^{r}_{\nu})$.
    \end{itemize}
    Additionally, $\mu$ verifies the SMF, so for all $h \in\R_{+}$, $\tau_{\mu}^{*}(h) = \sigma_{\mu}(h)$. Hence, $\sigma_{f_{a}} \leq \sigma_{\mu}(\cdot - h^{r}_{\nu})$.

    To prove the lower bound $\sigma_{f_{a}} \geq \sigma_{\mu}(\cdot - h^{r}_{\nu})$, consider $h \in\R_{+}$ and $x\in E_{\mu}(h)$. 
    Recalling Definition \ref{defi:local_dimension_iso_set_and_spectrum}, $\underline{h}_{\mu}(x) = \overline{h}_{\mu}(x) = h_{\mu}(x)$, and \eqref{eq:theo:prevalent_spectrum:bound_h_f} implies that $h_{\mu}(x) = h_{f_{a}}(x) - h^{r}_{\nu}$.
    So, for every $h \in\R_{+}$, $E_{\mu}(h) \subset E_{f_{a}}(h + h^{r}_{\nu})$ and item \ref{item:sigma_is_tau_star} of Proposition \ref{prop:prop_capa_SMF} gives $\sigma_{f_{a}}(h) = \dimH(E_{f_{a}}(h)) \geq \dimH(E_{\mu}(h - h^{r}_{\nu})) = \sigma_{\mu}(h - h^{r}_{\nu})$.

    This holds for a given $r\in \R$. Now, given any countable sequence $(r_n)_{n\in \N}$, the set $\mathcal{F}^{r_n}$ is prevalent in $\widetilde{B}^{\xi}_{\infty,q}([0,1]^D) $, and so is their intersection $\bigcap_{n\in \N}\mathcal{F}^{r_n}$. This concludes the proof of Theorem \ref{theo:prev_spect_inh_besov_inf_q}, since for every $n\in\N$, for $\nu_{r_n}$-a.e $a\in(0,1]^{d'}$, $\sigma_{f_a}(h) = \sigma_{\mu}(h-h^{r_n}_{\nu})$.
\end{proof}

\subsection{A useful decomposition of \texorpdfstring{$\mathcal{F}^{r}$}{F of xi}}\label{subsec:decomposition_of_F_xi}

The objective is to show that, $r$ being fixed, $(\mathcal{F}^{r})^\complement = \widetilde{B}^{\xi}_{\infty,q}([0,1]^d) \backslash \mathcal{F}^{r}$ is shy.
For this, we include $(\mathcal{F}^{r})^\complement $ in a countable union of simpler ancillary sets.
For $n,m,p,M \in \N^{*}$ and $\gamma>0$, let 
\begin{equation*}
    \mathcal{O}_{n,m}^{p,\gamma,M} = \big\{
        f\in \widetilde{B}^{\mu}_{\infty,q}([0,1]^d)~:~\nu_r\big( \mathcal{A}_{n,m}^{p,\gamma,M}(f) \big) > 0 \big\}.
\end{equation*}
with
\begin{equation}\label{eq:A_nm_pgammaM}
    \mathcal{A}_{n,m}^{p,\gamma,M}(f) = \left\{ a\in A^{\nu_r}_{n,m}{\cap(0,1]^{d'}}:
    \begin{array}{c}
        \exists x \in [0,1]^d, \forall \lambda_j^d = (j,k,l)\in\Lambda^d\times L^d,\\
        \big|c^{f}_{\lambda^d_j}(a)\big| \leq M_1\ 2^{-j\gamma} \big(1+|2^j x-k|\big)^{\gamma}\\
        \forall j>M_2,~\mu(\lambda_j(x)) \geq 2^{-j(\gamma-h^{r}_{\nu}-1/p)}
    \end{array}\right\}.
\end{equation}
The sets $\mathcal{O}_{n,m}^{p,\gamma,M}$ contains those functions $f\in \widetilde{B}^{\xi}_{\infty,q}([0,1]^D)$ such that, for a set of $a\in(0,1]^{d'}$ of $\nu_r$-positive measure, for a given $x\in [0,1]^d$, $f_a \in C^{\gamma}(x)$ and simultaneously the $\mu^{(+h^{r}_{\nu})}$-mass of all small dyadic cubes $\lambda^d_j(x)$ is bounded by below. These two properties will allow us to control the ratio $\frac{|c_{\lambda^d}(a)|}{\mu^{(+h^{r}_{\nu})}(\lambda^{d}(x))}$.

\begin{prop}\label{prop:inclusion_complement_F_xi_in_O}
    We have $(\mathcal{F}^{r})^\complement \subset \bigcup_{p\in\N^*} \bigcap_{m\in\N^*} \bigcup_{n\in\N^*} \bigcup_{\gamma\in\Q_{+}} \bigcup_{M\in(\N^*)^2} \mathcal{O}_{n,m}^{p,\gamma,M}$.
\end{prop}
\begin{proof}
    We begin by observing that $(\mathcal{F}^{r})^{\complement}$ can be written as 
    \begin{multline*}
        (\mathcal{F}^{r})^{\complement} = \big\{ f\in \widetilde{B}^{\xi}_{\infty,q}([0,1]^D) : \forall A(f) \text{ of full $\nu_r$-measure}, \big. \\ \big.\exists a\in A(f),\exists x \in [0,1]^d, h_{f_{a}}(x) > \overline{h}_{\mu}(x) + h^{r}_{\nu} \big\}.
    \end{multline*}
    We write $(\mathcal{F}^{r})^{\complement} = \bigcup_{p\in\N^*} (\mathcal{F}_{p}^{r})^{\complement}$, where 
    \begin{multline*}
        \mathcal{F}_{p}^{r} = \big\{ f\in \widetilde{B}^{\xi}_{\infty,q}([0,1]^D) : \exists A(f) \text{ of full $\nu_r$-measure}, \big.\\ \big.\forall a\in A(f), \forall x \in [0,1]^d, h_{f_{a}}(x) - \big(\overline{h}_{\mu}(x) + h^{r}_{\nu}\big) \leq {2}/{p} \big\}.
    \end{multline*}
    Suppose that $f\in (\mathcal{F}^{r})^{\complement} $. There exists $p\in \N^*$ such that $f\notin \mathcal{F}_{p}^{r}$. The set $\mathcal{N}=\{a\in[0,1]^{d'} : \exists x\in [0,1]^{d},~h_{f_{a}}(x) - \big(\overline{h}_{\mu}(x) + h^{r}_{\nu}\big) > \frac{2}{p} \}$ has positive $\nu_r$-measure.
    Recall from Proposition \ref{prop:bound_meas_dyadics} that $\bigcap_{m\in\N^*} \bigcup_{n\in\N^*} A^{\nu_r}_{n,m}$ has full $\nu_r$-measure in $(0,1]^{d'}$.
    Considering
    \begin{equation*}
        \mathcal{N}_{n,m} := \{a\in A^{\nu_r}_{n,m}\cap(0,1]^{d'} : \exists x\in [0,1]^{d}, h_{f_{a}}(x) - \big(\overline{h}_{\mu}(x) + h^{r}_{\nu}\big) > 2/p \},
    \end{equation*}
    one has the following decomposition
    \begin{multline*}
        \Big ( \bigcap\limits_{m\in\N^*} \bigcup\limits_{n\in\N^*} \mathcal{N}_{n,m} \Big )^{\complement}
        =\Big ( \bigcap\limits_{m\in\N^*} \bigcup\limits_{n\in\N^*} A^{\nu_r}_{n,m} \Big )^{\complement} \\
        \bigcup \Big \{a\in \bigcap\limits_{m\in\N^*} \bigcup\limits_{n\in\N^*} A^{\nu_r}_{n,m} \cap (0,1]^{d'} : \forall x\in[0,1]^{d}, h_{f_{a}}(x) - \big(\overline{h}_{\mu}(x) + h^{r}_{\nu}\big) \leq {2}/{p}\Big \}.
    \end{multline*}
    Since $(\bigcap_{m\in\N^*} \bigcup_{n\in\N^*} A^{\nu_r}_{n,m})^{\complement}$ has zero $\nu_r$-measure, $f\notin \mathcal{F}_{p}^{r}$ implies that the set
    \begin{equation*}
        \Big \{a\in \bigcap\limits_{m\in\N^*} \bigcup\limits_{n\in\N^*} A^{\nu_r}_{n,m} \cap (0,1]^{d'} : \forall x\in [0,1]^{d},~h_{f_{a}}(x) - \big(\overline{h}_{\mu}(x) + h^{r}_{\nu}\big) \leq {2}/{p} \Big \}
    \end{equation*}
    cannot have full $\nu_r$-measure. Hence, if $f\notin \mathcal{F}_{p}^{r}$, then $\bigcap_{m\in\N^*} \bigcup_{n\in\N^*} \mathcal{N}_{n,m}$ must have positive $\nu_r$-measure. Thus, for every $m\in\N^*$, $\bigcup_{n\in\N^*} \mathcal{N}_{n,m}$ must have positive $\nu_r$-measure, and there exists $n_m \in\N^*$ such that $\mathcal{N}_{n_m,m}$ has positive $\nu_r$-measure.
    
    Suppose $a\in \mathcal{N}_{n_m,m}$, there exists $\gamma \in \Q_{+}$ and $x \in [0,1]^d$ such that $h_{f_a}(x) > \gamma$ and $\overline{h}_{\mu}(x) + h^{r}_{\nu} \leq \gamma- 2/p$. Using \eqref{eq:equi_hold_space_wlt_coef} of Proposition \ref{prop:equi_hold_space_wlt_coef}, if $h_{f_a}(x) > \gamma$, then $f_a \in C^{\gamma}(x)$ , which implies that
    \begin{equation*}
        \exists M_{1}\in\N, \forall \lambda_j = (j,k,l)\in\Lambda\times L^d, \big|c^{f}_{\lambda^d_j}(a)\big| \leq M_{1} 2^{-j\gamma} \big(1+|2^j x-k|\big)^{\gamma}.
    \end{equation*}
    Using Definition \ref{defi:local_dimension_iso_set_and_spectrum}, $\overline{h}_{\mu}(x) + h^{r}_{\nu} \leq \gamma - 2/p$ yields that $\sup_{j>M_2} \frac{\log \mu(\lambda_j(x))}{\log 2^{-j}} \leq \gamma-h^{r}_{\nu}-1/p$ for some integer $M_2$ large enough. This implies that
    \begin{equation*}
        \exists M_2\in\N, \forall j>M_2, \mu(\lambda_j(x)) \geq 2^{-j(\gamma-h^{r}_{\nu}-1/p)}.
    \end{equation*}
    This precisely means that $a$ belongs to the set $\mathcal{A}_{n_m,m}^{p,\gamma,M}(f)$ given in \eqref{eq:A_nm_pgammaM} with $M=(M_1,M_2)$. So $\mathcal{N}_{n_m,m} \subset \bigcup_{\gamma\in\Q_{+}} \bigcup_{M\in(\N^*)^2} \mathcal{A}_{n_m,m}^{p,\gamma,M}(f)$ and since $\nu_r(\mathcal{N}_{n,m})>0$, one of the $\mathcal{A}_{n_m,m}^{p,\gamma,M}(f)$ must have positive $\nu_r$-measure.
    One concludes that $f\in \mathcal{O}_{n_m,m}^{p,\gamma,M}$, for every integer $m$.
    
    To conclude, every $f\in (\mathcal{F}^{r})^{\complement}$ belongs to $\bigcap_{m\in \N} \bigcup_{n\in\N^*} \bigcup_{\gamma\in\Q_{+}} \bigcup_{M\in(\N^*)^2} \mathcal{O}_{n,m}^{p,\gamma,M}$ for some $p>0$. This proves the claim.
\end{proof}

\subsection{Analyticity of \texorpdfstring{$\mathcal{O}^{p,\gamma,M}_{n,m}$}{O p,gamma,M n,m}}\label{subsec:analyticity_O_nm}

To show that $\mathcal{F}^{r}$ is prevalent, it remains to prove the universal measurability of $\mathcal{O}^{p,\gamma,M}_{n,m}$ and that these sets are all shy. This is the purpose of the rest of the article.

We will actually prove that $\mathcal{O}^{p,\gamma,M}_{n,m}$ is an analytic set, which implies its universal measurability.
Analytic sets in Polish spaces are defined as continuous images of Borel sets. Unfortunately, as we work with $\widetilde{B}^{\xi}_{\infty,q}([0,1]^d)$ which is not separable, we cannot use the Polish spaces characterisation of analytic sets and have to consider a more general definition valid for any Hausdorff topological space $X$ endowed with a Borel $\sigma$-algebra $\mathcal{B}(X)$. 
More precisely, Choquet's capacitability theorem \cite[Theorem 30.1]{Choquet:1954:theory_of_capacities} with the following definition of analyticity \cite[Theorem 5.1]{Choquet:1954:theory_of_capacities} offers the right setting for our situation.
\begin{defi}\label{defi:analyticity_hausdorff_space}
    For a compact topological space $K$, the collection of its closed subsets is written $\mathcal{K}$. Let $(\mathcal{B}(X) \times \mathcal{K})_{\sigma \delta}$ be the collection of countable intersections of countable unions of sets that are Cartesian products of a Borel set in $X$ and a closed set in $K$.
    
    Call $\pi:X\times K \to X$ the canonical projection map $\pi(x,y)=x$.
    A set $A\subset X$ is analytic when there is a compact space $K$ and $\mathcal{T}\in(\mathcal{B}(X)\times\mathcal{K})_{\sigma \delta}$ such that $A= \pi(\mathcal{T})$.
\end{defi}
The famous Choquet's capacitability theorem is the following.
\begin{theo}\label{theo:analyticity_to_univ_measurable}
    An analytic set of a Hausdorff topological space is universally measurable.
\end{theo}
We finally prove the analyticity of $\mathcal{O}^{p,\gamma,M}_{n,m}$, and so its universal measurability.
\begin{prop}\label{prop:universally_measurable_O}
    Let $n,m,p\in\N^*$, $M\in(\N^{*})^2$ and $\gamma\in\Q_{+}$. The set $\mathcal{O}^{p,\gamma,M}_{n,m}$ is universally measurable in $\widetilde{B}^{\xi}_{\infty,q}([0,1]^D)$.
\end{prop}
\begin{proof}
    We are going to prove the analyticity of $\mathcal{O}^{p,\gamma,M}_{n,m}$ by taking the right sets $X$, $K$ and $\mathcal{T}$ in Definition \ref{defi:analyticity_hausdorff_space}, and then writing it as $\pi(\mathcal{T})$.
    {Denote $X:= (\widetilde{B}^{\xi}_{\infty,q}([0,1]^D) \times (0,1]^{d'},\mathscr{D}(\cdot,\cdot)+\|\cdot\|_{\infty})$ with $\mathscr{D}$ defined in \eqref{eq:metric_B_xi_tilde} and $K:=([0,1]^{d},\|\cdot\|_{\infty})$.} The proof follows some ideas of \cite{Aubry-Maman-Seuret:2013:Traces_besov_results}.
    \medskip

    \underline{Step 1:} For every $\lambda^d_j=(j,k,l)\in\Lambda^d \times L^d$, every $c^{f}_{\lambda^d_j}(a)$ \eqref{eq:dec_trace_2}, let us introduce the maps \\$\phi_{\lambda^d_j}, \phi : X\times K \to \R$ defined as
    \begin{align*}
        \phi_{\lambda^d_j} : (f,a,x) &\mapsto M_1 2^{-j\gamma} \big(1+|2^j x-k|\big)^{\gamma} - \big|c^{f}_{\lambda^d_j}(a)\big| \\
        \phi : (f,a,x) &\mapsto \inf_{\lambda\in\Lambda^d\times L^d} \phi_{\lambda} (f,a,x).
    \end{align*}
    We introduce the set of dyadic cubes
    \begin{align}
        \mathcal{D}^{d}_{\mu}(j,\gamma-h^{r}_{\nu}-1/p) &:= \big\{ \lambda_j\in\Lambda^{d}_{j} ~:~ \frac{\log\mu(\lambda_j)}{\log 2^{-j}} \leq \gamma - h^{r}_{\nu} - {1}/{p}\big\} \nonumber \\ 
        \textnormal{and}\quad \mathcal{D}^{\geq}_{\mu}(\gamma-h^{r}_{\nu}-1/p) &:= \bigcap_{j\geq M_2} \bigcup_{\lambda\in\mathcal{D}^{d}_{\mu}(j,\gamma-h^{r}_{\nu}-1/p)} \lambda. \label{eq:def_D_mu_sup} 
    \end{align}
    Also, let
    \begin{equation*}
        \widetilde{\phi} : f \mapsto \nu_{r}\big( \big\{ a\in A^{\nu_r}_{n,m}\cap (0,1]^{d'} : \exists x\in \mathcal{D}^{\geq}_{\mu}(\gamma-h^{r}_{\nu}-1/p), \ \phi(f,a,x) \geq 0 \big\} \big).
    \end{equation*}
    With these definitions, one sees that $\mathcal{O}^{p,\gamma,M}_{n,m}=\widetilde{\phi}^{-1}\big( (0,+\infty) \big)$.
 
    \begin{lemme}\label{lemma:continuity_trace_func}
        The map $(f,a)\in X \longmapsto f_a\in (C^0([0,1]^d),\|\cdot\|_{L^{\infty}})$ is continuous. 
    \end{lemme}
 \begin{proof}
        Let $\varepsilon>0$, and fix $f,g \in X$ such that $\mathscr{D}(f,g) < \varepsilon$. 
        By \eqref{eq:metric_B_xi_tilde}, for $n:= \partent{\max(1,s_1^{-1})}+1$, one has $2^{-n} \frac{\|f-g\|_{B^{\xi^{(-1/n)}}_{\infty,q}}}{1+\|f-g\|_{B^{\xi^{(-1/n)}}_{\infty,q}}} < \varepsilon.$ This impies $\|f-g\|_{B^{\xi^{(-1/n)}}_{\infty,q}} < 2^{n} \varepsilon$, hence $\|f-g\|_{L^{\infty}} < 2^{n} \varepsilon$.
        
        Then, the uniform continuity of $f$ on $[0,1]^D$ shows the existence of $\delta>0$ such that if $a,b\in [0,1]^{d'}$ verify $\|a-b\|_{\infty}\leq \delta$, then $\|f(\cdot,a)-f(\cdot,b)\|_{L^{\infty}}\leq \varepsilon$.
        One concludes by the triangular inequality that $\|f(\cdot,a)-g(\cdot,b)\|_{L^{\infty}} \leq \|f(\cdot,a)-f(\cdot,b)\|_{L^{\infty}} + \|f-g\|_{L^{\infty}} \leq (C 2^n + 1) \varepsilon$.
    \end{proof}
    \begin{lemme}
        Each $\phi_{\lambda^d_j}$ is continuous on $X\times K$. 
    \end{lemme}
    \begin{proof}
        Indeed, we deal with continuous functions $f$ and Lemma \ref{lemma:continuity_trace_func} shows that $(f,a) \mapsto f_a$ is continuous and $f_a \mapsto c^{f}_{\lambda^d_j}(a)$ is continuous as a scalar product between $f_a$ and a fixed wavelet $\psi_{\lambda^d_j}$. So as a composition of continuous functions, the mapping $(f,a,x) \mapsto c^{f}_{\lambda^d_j}(a)$ is continuous and so if $\phi_{\lambda^d_j}$ as a sum of continuous functions.
    \end{proof}
    
    {\underline{Step 2:} Let us introduce 
    \begin{equation}\label{eq:def-T}
         \mathcal{T} := \phi^{-1} ([0,+\infty)) \cap \big(X \times \mathcal{D}^{\geq}_{\mu}(\gamma-h^{r}_{\nu}-1/p) \big).
    \end{equation}
    
    \begin{prop}\label{lemma:T_in_BXK}
        The set $\mathcal{T}$ belongs to $(\mathcal{B}(X) \times \mathcal{K})_{\sigma \delta}$.
    \end{prop}
    \begin{proof}
        $\bullet$ One first shows that the set $ \phi^{-1} ([0,+\infty))\in (\mathcal{B}(X) \times \mathcal{K})_{\sigma \delta}$. Observe that, as composition of continuous mappings, the mapping 
        \begin{equation*}
            \begin{array}{ccccc}
                X & \longrightarrow & (C^0([0,1]^d),\|\cdot\|_{L^{\infty}}) & \longrightarrow & \R \\
                \big(f,a \big) & \longmapsto & f_a & \longmapsto & c_{\lambda^d_j}(a)
            \end{array}
        \end{equation*}
        is continuous. So, for $n\in\N$ and $m\in \Z$, the set
        \begin{equation*}
                F_{\lambda_j^d}(n,m) := \big \{(f,a)\in X :c^f_{\lambda^d_j}(a) \in 2^{-n}[m,m+1]\big\}.
        \end{equation*}
        belongs to $\mathcal{B}(X)$ as the inverse of a Borel set by a continuous mapping. Let us then define
        \begin{equation*}
            X_{\lambda_j^d}(n,m)
            :=\big\{ x\in[0,1]^d : \sup_{\big(f,a\big)\in F_{\lambda_j^d}(n,m)} \phi_{\lambda_j^d} \big(f,a,x \big) \geq 0 \big\}.
        \end{equation*}
        Observe that when $f$ ranges in $\widetilde{B}^{\xi}_{\infty,q}([0,1]^D)$ and $a$ in $(0,1]^{d'}$, all possible values of $c^f_{\lambda_j^d}(a) \in 2^{-n} [m, m+1]$ are reached. Hence, one can rewrite $X_{\lambda_j^d}(n,m)$ as
        \begin{align*}
            X_{\lambda_j^d}(n,m) 
            &= \Big\{ x\in[0,1]^d :\max_{c\in 2^{-n} [m, m+1]} M_1\,2^{-j\gamma} \big(1+|2^j x-k|\big)^{\gamma} - |c| \geq 0 \Big\}\\
            &= \big\{ x\in[0,1]^d :\exists\, c\in 2^{-n} [m, m+1],~ M_1\,2^{-j\gamma} \big(1+|2^j x-k|\big)^{\gamma} - |c| \geq 0 \big\}.
        \end{align*}
        Using the mappings $v: (x,c) \in\R^d\times\R \mapsto M_1\,2^{-j\gamma} \big(1+|2^j x-k|\big)^{\gamma} - |c| \in \R$ and the projection on the first coordinate $\widetilde{\pi}(x,c)=x$, one sees that                
        \begin{equation*}
            X_{\lambda_j^d}(n,m) 
            = \widetilde{\pi} \big( v^{-1}(\R_{+}) \cap ([0,1]^d\times 2^{-n} [m, m+1]) \big) \cap [0,1]^d.
        \end{equation*}
        Since $v$ and $\widetilde{\pi}$ are continuous, $X_{\lambda_j^d}(n,m)$ is thus compact.
        \begin{lemme}\label{lemma:phi_is_prod}
            One has $\phi_{\lambda_j^d}^{-1} ([0,+\infty)) = \bigcap_{n\in\N} \bigcup_{m\in \Z} F_{\lambda_j^d}(n,m) \times X_{\lambda_j^d}(n,m)$.
        \end{lemme}
        \begin{proof}
            First, let us prove that $\phi_{\lambda_j^d}^{-1} ([0,+\infty)) \subset \bigcap_{n\in\N} \bigcup_{m\in \Z} F_{\lambda_j^d}(n,m) \times X_{\lambda_j^d}(n,m)$.

            Let $(f,a,x)\in \phi_{\lambda_j^d}^{-1} ([0,+\infty))$. Then for every $n\in\N$, there is an integer $m\in\Z$ such that $c^f_{\lambda_j^d}(a) \in 2^{-n} [m, m+1]$ and $(f,a)\in F_{\lambda_j^d}(n,m)$. In addition, $\phi_{\lambda_j^d}\big(f,a,x \big) \geq 0$, so we have $\sup_{(g,b)\in F_{\lambda_j^d}(n,m)} \phi_{\lambda_j^d}\big(g,b,x \big) \geq 0$ and $x\in X_{\lambda_j^d}(n,m)$.
            Hence $(f,a,x) \in F_{\lambda_j^d}(n,m) \times X_{\lambda_j^d}(n,m)$.
            \medskip

            Next, we show the inclusion $\phi_{\lambda_j^d}^{-1} ([0,+\infty)) \supset \bigcap_{n\in\N} \bigcup_{m\in \Z} F_{\lambda_j^d}(n,m) \times X_{\lambda_j^d}(n,m)$.

            Let $(f,a,x)\in \bigcap_{n\in\N} \bigcup_{m\in \Z} F_{\lambda_j^d}(n,m) \times X_{\lambda_j^d}(n,m)$. For a fixed $n\in\N$, there is $m\in\Z$ such that $(f,a)\in F_{\lambda_j^d}(n,m)$ and $x\in X_{\lambda_j^d}(n,m)$, meaning that $c^f_{\lambda_j^d}(a) \in 2^{-n} [m, m+1]$ and $\sup_{(g,b)\in F_{\lambda_j^d}(n,m)} \phi_{\lambda_j^d}\big(g,b,x \big) \geq 0$. One deduces that
            \begin{align*}
                \phi_{\lambda_j^d}\big(f,a,x \big)  
                &\geq \sup_{(g,b)\in F_{\lambda_j^d}(n,m)} \phi_{\lambda_j^d}\big(g,b,x \big) - \Big|\sup_{(g,b)\in F_{\lambda_j^d}(n,m)} \phi_{\lambda_j^d}\big(g,b,x \big) - \phi_{\lambda_j^d}\big(f,a,x \big)\Big|\\
                &\geq - \sup_{(g,b)\in F_{\lambda_j^d}(n,m)} \big|\phi_{\lambda_j^d}\big(g,b,x \big) - \phi_{\lambda_j^d}\big(f,a,x \big)\big| \\
                &\geq - \sup_{(g,b)\in F_{\lambda_j^d}(n,m)} \big|\big|c^{g}_{\lambda_j^d}(b) \big|- \big|c^{f}_{\lambda_j^d}(a)\big|\big| 
                \geq -2^{-n}.
            \end{align*}
            Since this holds for all $n\in\N$, one deduces that $\phi_{\lambda_j^d}\big(f,a,x \big) \geq 0$, hence the result.
        \end{proof}
    Lemma \ref{lemma:phi_is_prod} ensures that $\phi_{\lambda^d_j}^{-1} ([0,+\infty)) \in (\mathcal{B}(X) \times \mathcal{K})_{\sigma \delta}$, which implies that  
    \begin{equation*}
        \phi^{-1} ([0,+\infty)) = \bigcap_{\lambda^d\in\Lambda^d \times L^d} \phi_{\lambda^d}^{-1} ([0,+\infty)) \in (\mathcal{B}(X) \times \mathcal{K})_{\sigma \delta}.
    \end{equation*}
    $\bullet$ To prove Proposition \ref{lemma:T_in_BXK}, it remains to show that in \eqref{eq:def-T}, the second term defining $\mathcal{T} $ also belongs to $(\mathcal{B}(X) \times \mathcal{K})_{\sigma \delta}$. By \eqref{eq:def_D_mu_sup}, the set $\mathcal{D}^{\geq}_{\mu}(\gamma-h^{r}_{\nu}-1/p)$ is closed in $[0,1]^d$ as countable intersection of closed sets. Thus $X \times \mathcal{D}^{\geq}_{\mu}(\gamma-h^{r}_{\nu}-1/p)$ is a countable intersection of countable unions of closed sets, and so belongs to $(\mathcal{B}(X) \times \mathcal{K})_{\sigma \delta}$.
    
    \smallskip
    
    \noindent$\bullet$
    Finally, $\mathcal{T}$ is the intersection of two sets belonging to $(\mathcal{B}(X) \times \mathcal{K})_{\sigma \delta}$, hence the conclusion.
    \end{proof}}

    \underline{Step 3:} By Definition \ref{defi:analyticity_hausdorff_space}, the canonical projection map $\pi:(x,y)\in X\times K \to x\in X$ gives that 
    \begin{equation*}
        \pi(\mathcal{T}) := \big\{ (f,a)\in X : \exists x\in \mathcal{D}^{\geq}_{\mu}(\gamma-h^{r}_{\nu}-1/p), (f,a,x)\in\phi^{-1} ([0,+\infty)) \big\}
    \end{equation*}
    is analytic in the space $(X,\mathcal{B}(X))$. 
 
    \medskip

    \underline{Step 4:} To conclude, notice that $\widetilde{\phi}:\widetilde{B}^{\xi}_{\infty,q}([0,1]^D) \longrightarrow [0,1] $ can be rewritten as 
    \begin{align*}
        \widetilde{\phi}(f)
        &= \nu_{r}\big( \big\{ a\in A^{\nu_r}_{n,m} \cap (0,1]^{d'} :\exists x\in \mathcal{D}^{\geq}_{\mu}(\gamma-h^{r}_{\nu}-1/p), \phi(f,a,x) \geq 0 \big\} \big) \\
        &= \int_{A^{\nu_r}_{n,m}} \mathds{1}_{\pi(\mathcal{T})} (f,a) d\nu_r(a).
    \end{align*}
    Since $\pi(\mathcal{T})$ is then $\beta\otimes\nu_r$-measurable, Fubini's theorem gives that $\widetilde{\phi}$ is $\beta$-measurable, for any complete Borel measure $\beta$ on $\widetilde{B}^{\xi}_{\infty,q}([0,1]^D)$. This yields that $\mathcal{O}^{p,\gamma,M}_{n,m}=\widetilde{\phi}^{-1}\big( (0,+\infty) \big)$ is also $\beta$-measurable, for any complete Borel measure $\beta$ on $\widetilde{B}^{\xi}_{\infty,q}([0,1]^D)$.
\end{proof}

\subsection{Shyness of \texorpdfstring{$\mathcal{O}^{p,\gamma,M}_{n,m}$}{O p,gamma,M n,m}}\label{subsec:shyness_O_nm}

The standard method to prove the shyness of a set is to focus on a finite dimensional \textit{probe space} $\mathcal{P}\subset\widetilde{B}^{\xi}_{\infty,q}([0,1]^D)$ and the Lebesgue measure restricted to $\mathcal{P}$ as a transverse measure.
 
 Recall the integer $N(d')$ and the sequence of multi-integer $(\overline{p_j}=(p_{j,1},\ldots,p_{j,d'}))_{j\in\N}$ obtained in Theorem \ref{theo:low_bound_1per_wlt}.
Let $m \geq p(d+2)$ and $d_1$ be the integer given by 
\begin{equation}\label{eq:def_d_1}
    d_1 = p(d+1).
\end{equation}
\begin{defi}\label{defi:saturating_func_on_multi_dim_subspace}
    The probe space $\mathcal{P}$ is the subspace of $\widetilde{B}^{\xi}_{\infty,q}([0,1]^D)$ spanned by the functions $\mathscr{G}^{(i)} := \sum_{\lambda^D \in\Lambda^D\times L^D} g^{(i)}_{\lambda^D} \psi_{\lambda^D}$, for $i\in \{1,\ldots,d_1\}$, whose wavelet coefficients $g^{(i)}_{\lambda^D}$ are defined in the following way: for every $\lambda^D=(j,k,l)\in \Lambda^D\times \{0,1\}^D$,
    \begin{equation*}
        g^{(i)}_{\lambda^D} = 
        \begin{cases}
            g_{\lambda^D} &\text{if } j\pmod{d_1 N(d')} \in \{0,\ldots,N(d')-1\}+(i-1)\times N(d'),\\
            0 &\text{else}.
        \end{cases}
    \end{equation*}
 \end{defi}
Observe that $\sum_{i=1}^{d_1} \mathscr{G}^{(i)} = \mathscr{G}_q$. The intuition is that $\mathscr{G}_q$ is sampled into $d_1-1$ functions $\mathscr{G}^{(i)}$ whose wavelet coefficients are non-zero only on a regular subsequence of generation $j$. 
More precisely, for $a\in[0,1]^{d'}$, by \eqref{eq:coef_saturating_func} and \eqref{eq:definition_trace_func_saturation}, the trace $\mathscr{G}^{(i)}_{a}$ can be expressed as
\begin{equation}\label{eq:definition_trace_func_saturation_d1_dimensional}
    \mathscr{G}^{(i)}_{a}(x) = \sum_{j\geq 0} \sum_{\lambda_j^{d}\in \Lambda^d_j\times L^d} e^{(i)}_{\lambda_j^d}(a) \psi_{\lambda_j^{d}}(x)
\end{equation}
where, recalling \eqref{eq:definition_trace_func_saturation}
, for every $\lambda^d=(j,k,l)\in\Lambda^d$,
\begin{equation}\label{eq:formula_coefficient_trace_saturating_func_multi_dim_subspace} 
    e^{(i)}_{\lambda_j^d}(a) :=
    \begin{cases}
        e_{\lambda_j^d}(a) &\text{if } j \!\!\! \mod {d_1 N(d')} \in \{0,\ldots,N(d')-1\}+(i-1) N(d'),\\
        0 &\text{else}.
    \end{cases}
\end{equation}

For any arbitrary $f\in \widetilde{B}^{\xi}_{\infty,q}([0,1]^D)$, define
\begin{equation}
    \begin{array}{ccc}
        \R^{d_1} & \to & f+\mathcal{P} \\
        \beta & \mapsto & f^{\beta} = f+ \sum_{i=1}^{d_1} \beta_i \cdot \mathscr{G}^{(i)}.
    \end{array}
\end{equation}

The trace of $f^{\beta}$ is denoted $f^{\beta}_{a}$, the wavelet coefficients of the trace and the associated wavelet leaders are respectively $c_{\lambda^d}^{\beta}(a)$ and $L_{\lambda^d}^{\beta}(a)$.
Similarly, $c_{\lambda^d}(a)$ and $L_{\lambda^d}(a)$ are the wavelet coefficient and the associated leader of $f_a$.
For $a\in (0,1]^{d'}$, let also
\begin{equation*}
    \mathcal{B}_a := \{\beta\in\R^{d_1} : \mathcal{A}_{n,m}^{p,\gamma,M}(f^{\beta})\}
\end{equation*}
 
\begin{lemme}\label{lemma:b_a_measurable}
    The mapping $\bigg\{ \begin{array}{ccc}
    [0,1]^{d'}\times \R^{d_1} & \longrightarrow & \R \\
    (a,\beta) & \longmapsto & \mathds{1}_{\mathcal{B}_a}(\beta)
    \end{array} \big.$ is measurable.
\end{lemme}
\begin{proof}
    Let $\Phi : (a,\beta,x) \mapsto \inf_{\lambda^d_j \in\Lambda^d\times L^d} \phi_{\lambda^d_j}(f^{\beta},a,x)$. The map $\Phi$ is Borel on $[0,1]^{d'}\times \R^{d_1} \times [0,1]^{d}$ as an infimum of countably many continuous functions. Writing explicitly $\mathcal{B}_a$ under the form
    \begin{equation*}
        \mathcal{B}_a := 
        \left\{\beta\in\R^{d_1} :
        \begin{array}{c}
            \exists x \in [0,1]^d, \forall \lambda^d_j = (j,k,l)\in\Lambda^d\times L^d, \\
            \big|{c^{\beta}_{\lambda^d_j}(a)} \big|\leq M_1~2^{-j\gamma} \big(1+|2^j x-k|\big)^{\gamma}\\
            \text{and} ~\forall j>M_2,~\mu(\lambda^d_j(x)) \geq 2^{-j(\gamma-h^{r}_{\nu}-1/p)}
        \end{array}\right\},
    \end{equation*}
    one sees that $\mathds{1}_{\mathcal{B}_a}(\beta) = \mathds{1}_{\mathcal{G}}(a,\beta)$, where
    \begin{equation*}
        \mathcal{G}= \big\{ (a,\beta)\in (0,1]^{d'}\times\R^{d_1} \ :\ \exists x\in \mathcal{D}^{\geq}_{\mu}(\gamma-h^{r}_{\nu}-1/p),\ \Phi(a,\beta,x)\geq 0 \big\}.
    \end{equation*}
    This set can be rewritten as
    \begin{equation}\label{eq:measurability_projection_borel_set}
        \mathcal{G}= \widetilde{\pi}\big( \Phi^{-1}([0,+\infty)) \bigcap \big( (0,1]^{d'}\times\R^{d_1} \times \mathcal{D}^{\geq}_{\mu}(\gamma-h^{r}_{\nu}-1/p) \big) \big),
    \end{equation}
    where $\widetilde{\pi}(a,\beta,x)=(a,\beta)$ is the canonical projection on the first two coordinates. The set between brackets in \eqref{eq:measurability_projection_borel_set} being a Borel set, $\mathcal{G}$ is analytic by Definition \ref{defi:analyticity_hausdorff_space}. By Theorem \ref{theo:analyticity_to_univ_measurable}, it is also Lebesgue-measurable. Since $\mathds{1}_{\mathcal{B}_a}(\beta) = \mathds{1}_{\mathcal{G}}(a,\beta)$, the conclusion follows.
\end{proof}

\begin{lemme}\label{lemma:B_a_0_lebesgue_measure}
    For $a\in A^{\nu_r}_{n,m}\cap (0,1]^{d'}$, the $d_1$-dimensional Lebesgue measure $\mathcal{L}_{d_1}$ of $\mathcal{B}_a$ is zero.
\end{lemme}
\begin{proof}
    For any $\lambda_0 := (j_0,k_0,l_0)\in\Lambda^d$, denoting
    \begin{equation*}
        \mathcal{B}_{a,\lambda_0} := 
        \left\{\beta\in\R^{d_1} :
        \begin{array}{c}
            \exists x \in \lambda_0, \forall \lambda^d_j = (j,k,l)\in\Lambda^d\times L^d, \\
             \big| {c^{\beta}_{\lambda^d_j}(a)} \big| \leq M_1~2^{-j\gamma} \big(1+|2^j x-k|\big)^{\gamma}\\
            \text{and} ~\forall j>M_2,~\mu(\lambda^d_j(x)) \geq 2^{-j(\gamma-h^{r}_{\nu}-1/p)}
        \end{array}\right\},
    \end{equation*}
    one has
    \begin{equation}\label{eq:lemma:B_a_0_lebesgue_measure:proof:decomposition_B_a}
        \mathcal{B}_{a} = \limsup_{j_0\to \infty} \bigcup_{(j_0,k_0,l_0)\in\Lambda_{j_0}^d} \mathcal{B}_{a,\lambda_0} = \bigcap_{J\in\N^*} \bigcup_{j_0\geq J} \bigcup_{(j_0,k_0,l_0)\in\Lambda_{j_0}^d} \mathcal{B}_{a,\lambda_0}.
    \end{equation}
    We show that $\mathcal{L}_{d_1}(\mathcal{B}_a) = 0$ by bounding above each $\mathcal{L}_{d_1}\big(\mathcal{B}_{a,\lambda_0}\big)$, and conclude with the Borel-Cantelli lemma.
    
    One can ensure that $j_0$ is large enough that $j_0 \geq M_2 + d_1 N(d')$, $j_0-d_1 N(d') \equiv 0 \pmod{d_1 N(d')}$ and $k_j(a)>0$ (recall \eqref{eq:def_kja}). 
    Suppose $\beta$ and $\widetilde{\beta}$ belong to a same $\mathcal{B}_{a,\lambda_0}$. 
    There exist $x_{\beta}, x_{\widetilde{\beta}}\in \lambda_0 \subset \R^d$ such that for every $j\in\N^*$, if $k_{j,\beta}$ and $k_{j,\widetilde{\beta}}$ is the multi-integer associated to $\lambda^d_{j}(x_{\beta})$, \begin{equation}\label{eq:majoration_d_j_beta}
         \big| {c^{\beta}_{\lambda_{j}(x_{\beta})}(a)} \big|\leq M_1 \ 2^{-j\gamma} \big(1+|2^j x_{\beta}-k_{j,\beta}|\big)^{\gamma} 
     \end{equation} 
    and similarly for $c^{\widetilde{\beta}}_{\lambda_{j}(x_{\widetilde{\beta}})}(a)$ with its associated $k_{j,\widetilde{\beta}}$.
    
    For every $j\in\{j_0-d_1 N(d'),\ldots,j_0\}$, one has $\lambda^d_{j}(x_{\beta}) = \lambda^d_{j}(x_{\widetilde{\beta}})$: call $\widetilde{\lambda}_{j}$ this cube, for short. 
    For every $i\in \{1,\ldots,d_1\}$, set $j^{(i)}\equiv (i-1)N(d') \pmod{d_1 N(d')}$ with $j^{(i)}\in\{j_0-d_1 N(d'),\ldots,j_0\}$ (remark that $j^{(1)}=j_0-d_1 N(d')$).
    We make use of Theorem \ref{theo:low_bound_1per_wlt}. For every $i\in \{1,\ldots,d_1\}$, for every $a\in\R^{d'}$, there exists an integer $J^{(i)}\in\{0,\ldots,N(d')-1\}$ such that $G_{j^{(i)}+J^{(i)}}^{d'}(a) \geq \alpha^{d'} >0$. One sees that $j^{(i)}+J^{(i)}\in\{j_0-d_1 N(d'),\ldots,j_0\}$.
    Using $x_{\beta}, x_{\widetilde{\beta}}\in \lambda_0$ and \eqref{eq:majoration_d_j_beta}, one deduces that there is a constant $C>0$ depending on $M_1, j_0, d_1, \psi, \gamma$ and $D$ such that, for every $j\in\{j_0-d_1 N(d'),\ldots,j_0\}$, $\big|c^{\beta}_{\widetilde{\lambda}_{j}}(a)\big|\leq C 2^{-j\gamma}$ and $\big|c^{\widetilde{\beta}}_{\widetilde{\lambda}_{j}}(a)\big|\leq C 2^{-j\gamma}$.
    From all this, one deduces that
    \begin{align*}
        &\big|(\beta_i-\widetilde{\beta}_i) \ e_{\widetilde{\lambda}_{j^{(i)}+J^{(i)}}} \!\! (a)\big|= \big|c_{\widetilde{\lambda}_{j^{(i)}+J^{(i)}}} \!\! (a) + \beta_i \ e_{\widetilde{\lambda}_{j^{(i)}+J^{(i)}}} \!\! (a) - \big (c_{\widetilde{\lambda}_{j^{(i)}+J^{(i)}}} \!\! (a) + \widetilde{\beta}_i \ e_{\widetilde{\lambda}_{j^{(i)}+J^{(i)}}} \!\! (a) \big)\big| \\
        &\leq \big|c^{\beta}_{\widetilde{\lambda}_{j^{(i)}+J^{(i)}}}(a)\big| + \big|c^{\widetilde{\beta}}_{\widetilde{\lambda}_{j^{(i)}+J^{(i)}}}(a)\big| \leq 2C\cdot 2^{-(j_0-d_1 N(d'))\gamma}.
    \end{align*}
    So, recalling \eqref{eq:definition_trace_func_saturation}, one has
    \begin{equation*}
        \big|\beta_i-\widetilde{\beta}_i\big| 
        \leq 2C j^{2/q} \frac{2^{-(j_0-d_1 N(d'))\gamma}}{\mu\big(\widetilde{\lambda}_{j^{(i)}+J^{(i)}}\big)\nu\big(\lambda^{d'}_{j^{(i)}+J^{(i)},k_{j^{(i)}+J^{(i)}}(a)}\big) G^{d'}_{j^{(i)}+J^{(i)}}(a) }. 
    \end{equation*}
    When $j\in\{j_0-d_1 N(d'),\ldots,j_0\}$, $j>M_2$, so $\mu(\widetilde{\lambda}_{j}) \geq 2^{-j(\gamma-h^{r}_{\nu}-1/p)}$, from which we deduce that $\mu\big(\widetilde{\lambda}_{j}\big) \geq 2^{-j_0(\gamma-h^{r}_{\nu}-1/p)}$.
    As $a\in A^{\nu_r}_{n,m}$, this gives 
    \begin{equation*}
        \nu\big(\lambda^{d'}_{j^{(i)}+J^{(i)},k_{j^{(i)}+J^{(i)}}(a)}\big) \geq 2^{-(j^{(i)}+J^{(i)})(h^{r}_{\nu}+1/m)} \geq 2^{-j_0(h^{r}_{\nu}+1/m)}.
    \end{equation*}
    So for every $i\in \{1,\ldots,d_1\}$, recall that $G_{j^{(i)}+J^{(i)}}^{d'}(a) \geq \alpha^{d'}$, and one has
    \begin{equation*}
        \big|\beta_i-\widetilde{\beta}_i\big|
        \leq 2C j^{2/q} \frac{2^{-(j_0-d_1 N(d'))\gamma}}{2^{-j_0(\gamma - h^{r}_{\nu} - 1/p)} ~2^{-j_0(h^{r}_{\nu}+1/m)} \alpha^{d'} } 
        \leq C j_0^{2/q} 2^{-j_0(1/p - 1/m)}.
    \end{equation*}
    The distance between two elements of $\mathcal{B}_{a,\lambda_0}$ is at most $ j_0^{2/q} 2^{-j_0(1/p - 1/m)}$ up to a constant.
    Hence $\mathcal{L}_{d_1}(\mathcal{B}_{a,\lambda_0}) \leq C\ j_0^{2d/q} 2^{-j_0\cdot d_1(1/p - 1/m)}$.
    Summing over all the $2^{d j_0}$ dyadic cubes for $k_0 \in \Z^d_{j_0}$ at scale $j_0$ gives $\mathcal{L}_{d_1}\big( \bigcup_{k_0 \in \Z^d_{j_0}} \mathcal{B}_{a,\lambda_0} \big) \leq C j_0^{2d/q} 2^{j_0(d+d_1/m-d_1/p)}$.
    Because of our choice \eqref{eq:def_d_1} and $d_1/m \leq \frac{p+1}{p+2} \leq 1$, $d+d_1/m-d_1/p<0$. The Borel-Cantelli lemma implies $\mathcal{L}_{d_1}\big( \mathcal{B}_{a} \big) = \mathcal{L}_{d_1}\Big( \limsup_{j_0\to \infty} \bigcup_{k_0 \in \Z^d_{j_0}} \mathcal{B}_{a,\lambda_0} \Big) = 0$, which proves the claim.
\end{proof}

\begin{prop}\label{prop:beta_in_Xi_measure_0}
    For any $f\in \widetilde{B}^{\xi}_{\infty,q}([0,1]^D)$, the set $\big\{ \beta\in\R^{d_1} :  f^{\beta} \in \mathcal{O}_{n,m}^{p,\gamma,M} \big\}$ has $\mathcal{L}_{d_1}$-Lebesgue measure 0.
\end{prop}
\begin{proof}
    By Lemma \ref{lemma:B_a_0_lebesgue_measure}, $\int_{a\in A_{n,m}^{\nu_r}\cap(0,1]^{d'}}\int_{\beta\in\R^{d_1}} \mathds{1}_{\mathcal{B}_a}(\beta) \ d\beta d\nu_r(a) = 0.$ 
    Applying Fubini's theorem (measurability being guaranteed by Lemma \ref{lemma:b_a_measurable}), 
    \begin{equation*}
        \int_{\beta\in\R^{d_1}} \int_{a\in A_{n,m}^{\nu_r}\cap(0,1]^{d'}} \mathds{1}_{\mathcal{B}_a}(\beta) \ d\nu_r(a) d\beta = 0.
    \end{equation*}
    In other words, for Lebesgue-almost any $\beta\in\R^{d_1}$, there exists a set $A(f^{\beta})\subset A_{n,m}^{\nu_r}$ with same $\nu_r$-measure as $ A_{n,m}^{\nu_r}$, such that $A(f^{\beta})\cap \mathcal{A}_{n,m}^{p,\gamma,M}(f)=\emptyset $. Recalling that $\mathcal{A}_{n,m}^{p,\gamma,M}(f)\subset A_{n,m}^{\nu_r}$, this implies that $\nu_r(\mathcal{A}_{n,m}^{p,\gamma,M}(f^{\beta}))=0$, and $f^{\beta}\notin \mathcal{O}_{n,m}^{p,\gamma,M}$ which yields the result by taking complement.
\end{proof}
Finally, we are now ready to prove the shyness of the sets $\mathcal{O}_{n,m}^{p,\gamma,M}$.
\begin{prop}
    For every $n,m\in \N$ with $m \geq p(d+2)$, $M=(M_1,M_2)\in(\N^*)^2$ and $\gamma>h + h^{r}_{\nu}$, $\mathcal{O}_{n,m}^{p,\gamma,M}$ is shy.
\end{prop}
\begin{proof}
    Let $\vartheta = G_{\#}\mathcal{L}_{d_1}$ be the push-forward measure of the $\mathcal{L}_{d_1}$ onto $\mathcal{P}$ by 
    \begin{equation}
        G:\left\{\begin{array}{ccc}
            [0,1]^{d_1} & \to & \mathcal{P} \\
            \beta & \mapsto & \sum_{i=1}^{d_1} \beta_i \cdot \mathscr{G}^{(i)}.
        \end{array}\right.
    \end{equation}
    This measure is supported by a compact in $\mathcal{P} \subset \widetilde{B}^{\xi}_{\infty,q}([0,1]^D)$.
    Fix any $f\in \widetilde{B}^{\xi}_{\infty,q}([0,1]^D)$. Using Proposition \ref{prop:beta_in_Xi_measure_0}, for $\vartheta$-almost all $F\in\mathcal{P}$, $f+F \notin \mathcal{O}_{n,m}^{p,\gamma,M}$. Hence for any $f\in \widetilde{B}^{\xi}_{\infty,q}([0,1]^D)$, the set $\{f+\mathcal{O}_{n,m}^{p,\gamma,M}\}$ has $\vartheta$-measure equal to 0, i.e., $\vartheta(\{f+\mathcal{O}_{n,m}^{p,\gamma,M}\})=0$.
    Since it is true for any $f\in \widetilde{B}^{\xi}_{\infty,q}([0,1]^D)$, the set $\mathcal{O}_{n,m}^{p,\gamma,M}$ is shy.
\end{proof}
Using Proposition \ref{prop:inclusion_complement_F_xi_in_O}, the set $(\mathcal{F}^{r})^\complement$ is shy as it is a subset of countable unions and intersections of shy sets. This yields that $\mathcal{F}^{r}$ is prevalent, proving Proposition \ref{prop:prevalent_set_in_B_xi_inf_q}. 

\section{Further investigations}

A first natural extension consists in investigating the case where $p < \infty$, for which new and interesting phenomena may arise, as was observed in \cite{Barral-Seuret:2023:Besov_Space_Part_2}.
One would then like to explore situations where $\xi$ is not a product of measures, or, as a preliminary step, to consider capacities that are not Gibbs capacities associated with $\mathscr{E}_d$. For instance, a natural extension consists in taking for $\xi$ either 2- or 3-dimensional Mandelbrot cascades or Castaing cascade models, and to see whether the same conclusion hold. However, this scenario will be significantly more intricate and will likely depend on the behavior of the marginals of $\xi$ on $\R^{d}$ and $\R^{d'}$.

Finally, in all cases, it would be valuable to study traces not only on affine subspaces but also on general submanifolds.

\bibliographystyle{abbrv}
{\small
\bibliography{biblio}}

\begin{thebibliography}{10}

\bibitem{Abry-Ciuciu-Dumeur-Jaffard-Saes:2024:Multifract_analysis_neuroscience}
P.~Abry, P.~Ciuciu, M.~Dumeur, S.~Jaffard, and G.~Sa{\"e}s.
\newblock {Multifractal analysis based on the weak scaling exponent and
  applications to MEG recordings in neuroscience}.
\newblock Preprint 2024.

\bibitem{Ansorena-Blasco:1995:weighted_besov_spaces}
J.~L. Ansorena and O.~Blasco.
\newblock {Characterization of Weighted Besov Spaces}.
\newblock {\em Mathematische Nachrichten}, 171(1):5--17, 1995.

\bibitem{Aubry-Bastin-Dispa:2007:prevalence_sobolev}
J.-M. Aubry, F.~Bastin, and S.~Dispa.
\newblock {Prevalence of multifractal functions in {$S^{V}$} spaces}.
\newblock {\em Journal of Fourier analysis and applications}, 13:175--185,
  2007.

\bibitem{Aubry-Maman-Seuret:2013:Traces_besov_results}
J.-M. Aubry, D.~Maman, and S.~Seuret.
\newblock {Local behavior of traces of Besov functions: Prevalent results}.
\newblock {\em Journal of Functional Analysis}, 264(3):631--660, 2013.

\bibitem{Balka-Darji-Elekes:2016:dimension_graph_continous_maps}
R.~Balka, U.~Darji, and M.~Elekes.
\newblock {Hausdorff and packing dimension of fibers and graphs of prevalent
  continuous maps}.
\newblock {\em Advances in Mathematics}, 293:221--274, 2016.

\bibitem{Barral:2015:inverse-problems}
J.~Barral.
\newblock {Inverse problems in multifractal analysis of measures}.
\newblock {\em Annales Scientifiques de l'{\'E}cole Normale Sup{\'e}rieure},
  48:1457--1510, 2015.

\bibitem{Barral-Ben_Nasr-Peyriere:2003:multifract-formalism}
J.~Barral, F.~Ben~Nasr, and J.~Peyriere.
\newblock {Comparing multifractal formalisms: The neighboring boxes condition}.
\newblock {\em Asian Journal of Mathematics}, 7(2):149--166, 2003.

\bibitem{Barral-Seuret:2023:Besov_Space_Part_2}
J.~Barral and S.~Seuret.
\newblock {The Frisch-Parisi conjecture II: Besov spaces in multifractal
  environment, and a full solution}.
\newblock {\em J. Math. Pures Appl.}, 175:281--329, 2023.

\bibitem{Bayart-Heurteaux:2013:Hausdorff_Dimension_Graphs_Prevalent_Continuous_Functions}
F.~Bayart and Y.~Heurteaux.
\newblock {\em {On the Hausdorff dimension of graphs of prevalent continuous
  functions on compact sets}}, pages 25--34.
\newblock Birkh{\"a}user Boston, 2013.

\bibitem{Besoy-Haroske-Triebel:2022:traces_weighted_fct_spaces}
B.~F. Besoy, D.~D. Haroske, and H.~Triebel.
\newblock Traces of some weighted function spaces and related non-standard real
  interpolation of {B}esov spaces.
\newblock {\em Math. Nach.}, 295(9):1669--1689, 2022.

\bibitem{Brown_Michon_Peyriere:1992:Multifractal-analysis-measures}
G.~Brown, G.~Michon, and J.~Peyri{\`e}re.
\newblock {On the multifractal analysis of measures}.
\newblock {\em Journal of Statistical Physics}, 66(3--4):775--790, 1992.

\bibitem{Caetano-Haroske:2016:traces_besov_fractal}
A.~M. Caetano and D.~D. Haroske.
\newblock Traces of {B}esov spaces on fractal $h$-sets and dichotomy results.
\newblock {\em Studia Mathematica}, page 1–31, 2016.

\bibitem{Chamizo-Ubis:2014:Multifractal_behavior_polynomial_Fourier_Series}
F.~Chamizo and A.~Ubis.
\newblock {Multifractal behavior of polynomial Fourier series}.
\newblock {\em Advances in Mathematics}, 250:1--34, 2014.

\bibitem{Choquet:1954:theory_of_capacities}
G.~Choquet.
\newblock {Theory of capacities}.
\newblock {\em Annales de l'institut Fourier}, 5:131--295, 1954.

\bibitem{Christensen:1972:sets_Haar_measure_zero}
J.~P.~R. Christensen.
\newblock {On sets of Haar measure zero in abelian polish groups}.
\newblock {\em Israel Journal of Mathematics}, 13(3):255--260, 1972.

\bibitem{Daubechies:1988:compactly_supported_wavelets}
I.~Daubechies.
\newblock {Orthonormal bases of compactly supported wavelets}.
\newblock {\em Comm. Pure and App. Math.}, 41(7):909--996, 1988.

\bibitem{Daubechies-Lagarias:1991:regularity_multiresolution_analysis}
I.~Daubechies and J.~C. Lagarias.
\newblock {Two-Scale Difference Equations. I. Existence and Global Regularity
  of Solutions}.
\newblock {\em SIAM Journal on Mathematical Analysis}, 22(5):1388--1410, 1991.

\bibitem{Diening-Harjulehto-Hasto-Ruzicka:2011:variable_exponent_lebesgue_sobolev}
L.~Diening, P.~Harjulehto, P.~H{\"a}st{\"o}, and M.~Ruzicka.
\newblock {\em {Lebesgue and Sobolev Spaces with Variable Exponents}}, volume
  2017 of {\em Lecture Notes in Mathematics}.
\newblock Springer, 2011.

\bibitem{Farkas-Leopold:2006:characterisations_generalised_smoothness_Besov}
W.~Farkas and H.-G. Leopold.
\newblock {Characterizations of function spaces of generalised smoothness}.
\newblock {\em Annali di Matematica Pura ed Applicata}, 185:1--62, 2006.

\bibitem{Feng:2007:Gibbs_properties_conformal_measures}
D.-J. Feng.
\newblock {Gibbs properties of self-conformal measures and the multifractal
  formalism}.
\newblock {\em Ergodic Theory and Dynamical Systems}, 27(3):787--812, 2007.

\bibitem{Fraysse:2007:prevalent_validity_multifractal_formalism}
A.~Fraysse.
\newblock {Generic Validity of the Multifractal Formalism}.
\newblock {\em SIAM Journal on Mathematical Analysis}, 39(2):593--607, 2007.

\bibitem{Fraysse-Jaffard:2006:Almost_all_function_sobolev}
A.~Fraysse and S.~Jaffard.
\newblock {How smooth is almost every function in a Sobolev space?}
\newblock {\em Revista Matem{\'a}tica Iberoamericana}, 22(2):663--682, 2006.

\bibitem{Fraysse-Jaffard-Kahane:2005:propriete_analyse}
A.~Fraysse, S.~Jaffard, and J.-P. Kahane.
\newblock {Quelques propri{\'e}t{\'e}s g{\'e}n{\'e}riques en analyse}.
\newblock {\em Comptes Rendus Mathematique}, 340(9):645--651, 2005.

\bibitem{Frisch-Parisi:1985:Multifractal-turbulence}
U.~Frisch and G.~Parisi.
\newblock {On the singularity structure of fully developed turbulence}.
\newblock In {\em Turbulence and Predictability in Geophysical Fluid Dynamics
  and Climate Dynamics}, volume~88 of {\em {Proc. Int. School of Physic Enrico
  Fermi}}, pages 84--88. North-Holland, 1985.

\bibitem{Gaczkowski-Gorka-Pons:2016:variable_exponent_sobolev}
M.~Gaczkowski, P.~G{\'o}rka, and D.~J. Pons.
\newblock {Sobolev spaces with variable exponents on complete manifolds}.
\newblock {\em Journal of Functional Analysis}, 270(4):1379--1415, 2016.

\bibitem{Halsey:1987:fractal-measure-singularities}
T.~C. Halsey, M.~H. Jensen, L.~P. Kadanoff, I.~Procaccia, and B.~I. Shraiman.
\newblock {Fractal measures and their singularities: The characterization of
  strange sets}.
\newblock {\em Nuclear Physics B - Proceedings Supplements}, 2:501--511, 1987.

\bibitem{Heurteaux:2007:Dimension-measure}
Y.~Heurteaux.
\newblock {Dimension of measures: The probabilistic approach}.
\newblock {\em Publicacions Matem{\`a}tiques}, 51(2):243--290, 2007.

\bibitem{Hunt:1994:prevalence_continuous_functions}
B.~R. Hunt.
\newblock {The prevalence of continuous nowhere differentiable functions}.
\newblock {\em Proceedings of the American Mathematical Society},
  122(3):711--717, 1994.

\bibitem{Hunt-Sauer-Yorke:1992:prevalence}
B.~R. Hunt, T.~Sauer, and J.~A. Yorke.
\newblock {Prevalence: a translation-invariant ``almost every'' on
  infinite-dimensional spaces}.
\newblock {\em Bull. AMS}, 27(2):217--238, 1992.

\bibitem{Jaerisch-Sumi:2020:multi-frac-formalism-gibbs-measures}
J.~Jaerisch and H.~Sumi.
\newblock {Multifractal formalism for generalised local dimension spectra of
  Gibbs measures on the real line}.
\newblock {\em J. Math. Anal. Appl.}, 491(2):124246, 2020.

\bibitem{Jaffard:1989:Holder_exponent_wavelet_coef}
S.~Jaffard.
\newblock {Exposants de H{\"o}lder en des points donn{\'e}s et coefficients
  d'ondelettes}.
\newblock {\em CRAS}, 308:79--81, 1989.

\bibitem{Jaffard:1995:traces_negative_dim}
S.~Jaffard.
\newblock {Th{\'e}or{\`e}mes de traces et dimensions n{\'e}gatives}.
\newblock {\em CRAS}, 320:409--413, 1995.

\bibitem{Jaffard:1996:Riemann_function_singularities}
S.~Jaffard.
\newblock {The spectrum of singularities of Riemann's function}.
\newblock {\em Revista Matem{\'a}tica Iberoamericana}, 12(2):441--460, 1996.

\bibitem{Jaffard:1997:multifractal_formalism_all_functions}
S.~Jaffard.
\newblock {Multifractal formalism for functions part I: Results valid for all
  functions}.
\newblock {\em SIAM Journal on Mathematical Analysis}, 28(4):944--970, 1997.

\bibitem{Jaffard:2000:Frisch-Parisi-conjecture}
S.~Jaffard.
\newblock {On the Frisch--Parisi conjecture}.
\newblock {\em J. Math. Pures Appl.}, 79(6):525--552, 2000.

\bibitem{Jaffard:2004:Wavelet_techniques}
S.~Jaffard.
\newblock {Wavelet techniques in multifractal analysis}.
\newblock {\em Proc. Symp. Pure Math.}, 72, 2004.

\bibitem{Jaffard:2007:Wavelet_leaders}
S.~Jaffard, B.~Lashermes, and P.~Abry.
\newblock {Wavelet leaders in multifractal analysis}.
\newblock In {\em {Wavelet Analysis and Applications}}, pages 201--246.
  Birkh{\"a}user, 2007.

\bibitem{Jaffard-Meyer:2000:regularity_critical_besov_spaces}
S.~Jaffard and Y.~Meyer.
\newblock {On the pointwise regularity of functions in critical Besov spaces}.
\newblock {\em Journal of Functional Analysis}, 175(2):415--434, 2000.

\bibitem{Billat-Jaffard-Nasr-Saes:2023:Marathon-physiological-analysis}
S.~Jaffard, G.~Sa{\"e}s, W.~Ben~Nasr, F.~Palacin, and V.~Billat.
\newblock {A review of univariate and multivariate multifractal analysis
  illustrated by the analysis of marathon runners Physiological Data}.
\newblock In {\em {Analysis, Applications, and Computations}}, pages 3--60.
  Springer, 2023.

\bibitem{Kreit-Nicolay:2013:characterisation_generalised_Holder-Zygmund_spaces}
D.~Kreit and S.~Nicolay.
\newblock {Characterizations of the elements of generalized H{\"o}lder--Zygmund
  spaces by means of their representation}.
\newblock {\em J. Approx. Th.}, 172:23--36, 2013.

\bibitem{LevyVehel-Vojak:1998:Choquet-capacities}
J.~L{\'e}vy~V{\'e}hel and R.~Vojak.
\newblock {Multifractal Analysis of Choquet Capacities}.
\newblock {\em Advances in Applied Mathematics}, 20(1):1--43, 1998.

\bibitem{Loosveldt-Nicolay:2019:equivalent_definition_generalised_smoothness_besov}
L.~Loosveldt and S.~Nicolay.
\newblock {Some equivalent definitions of Besov spaces of generalized
  smoothness}.
\newblock {\em Mathematische Nachrichten}, 292(10):2262--2282, 2019.

\bibitem{Maman:2013:Genericity_Thesis}
D.~Maman.
\newblock {\em {G{\'e}n{\'e}ricit{\'e} et pr{\'e}valence des propri{\'e}t{\'e}s
  multifractales de traces de fonctions}}.
\newblock Phd thesis, Université Paris-Est Créteil, 2013.

\bibitem{Mattila99}
P.~Mattila.
\newblock {\em Geometry of Sets and Measures in Euclidean Spaces: Fractals and
  Rectifiability}.
\newblock Cambridge Studies in Advanced Mathematics, 1999.

\bibitem{Meyer:1990:Ondelette_operateur}
Y.~Meyer.
\newblock {\em {Ondelettes et op{\'e}rateurs I: Ondelettes}}.
\newblock Hermann, 1990.

\bibitem{Moura:2007:characterisation_generalised_smoothness_besov}
S.~D. Moura.
\newblock {On some characterizations of Besov spaces of generalized
  smoothness}.
\newblock {\em Mathematische Nachrichten}, 280(9-10):1190--1199, 2007.

\bibitem{Olsen:1995:Multifractal-formalism-measure}
L.~Olsen.
\newblock {A multifractal formalism}.
\newblock {\em Advances in Mathematics}, 116(1):82--196, 1995.

\bibitem{Olsen:2010:multifractal_dimension_prevalent_measure}
L.~Olsen.
\newblock {Fractal and multifractal dimensions of prevalent measures}.
\newblock {\em Indiana Univ. Math. J.}, 59:661--690, 2010.

\bibitem{Olsen:2010:Lq_dimension_prevalent_measure}
L.~Olsen.
\newblock {Prevalent {$L^q$}-dimensions of measures}.
\newblock {\em Mathematical Proceedings of the Cambridge Philosophical
  Society}, 149(3):553--571, 2010.

\bibitem{Parry-pollicott:1990:zeta_function}
W.~Parry and M.~Pollicott.
\newblock {\em {Zeta functions and the periodic orbit structure of hyperbolic
  dynamics}}, volume 187--188 of {\em Ast{\'e}risque}.
\newblock Soci{\'e}t{\'e} Math{\'e}matique de France, 1990.

\bibitem{Patzschke:1997:self-conformal-measures}
N.~Patzschke.
\newblock {Self-conformal multifractal measures}.
\newblock {\em Adv. Appl. Math.}, 19(4):486--513, 1997.

\bibitem{Pesin:1997:dimension_theory}
Y.~Pesin.
\newblock {\em {Dimension theory in dynamical systems contemporary views and
  applications}}.
\newblock University of Chicago Press, 1997.

\bibitem{Pesin-Weiss:1997:multi-frac-anal-gibbs-measures}
Y.~Pesin and H.~Weiss.
\newblock {The multifractal analysis of Gibbs measures: Motivation,
  mathematical Foundation, and examples}.
\newblock {\em Chaos}, pages 89--106, 1997.

\bibitem{Rible:2025:Prevalence_inhomogeneous_Besov}
Q.~Rible.
\newblock {Inhomogeneous Sobolev and Besov Spaces: Embeddings and prevalent
  smoothness}.
\newblock Preprint, 2025.

\bibitem{Rible-Seuret-1}
Q.~Rible and S.~Seuret.
\newblock A non-vanishing property for tensor products of wavelets.
\newblock {\em Preprint}, 2026.

\bibitem{Seuret:2018:Inhomogeneous-random-covering-Markov-Shifts}
S.~Seuret.
\newblock {Inhomogeneous random coverings of topological Markov shifts}.
\newblock {\em Mathematical Proceedings of the Cambridge Philosophical
  Society}, 165(2):341--357, 2018.

\bibitem{Seuret-Ubis:2017:Riemann_series}
S.~Seuret and A.~Ubis.
\newblock {Local $L^2$-regularity of {Riemann{\textquoteright}s} {Fourier}
  series}.
\newblock {\em Annales de l'Institut Fourier}, 67(5):2237--2264, 2017.

\bibitem{Shmerkin:2005:multifractal_formalism_self-similar_measures}
P.~Shmerkin.
\newblock {A Modified Multifractal Formalism for a Class of Self-similar
  Measures with Overlap}.
\newblock {\em Asian Journal of Mathematics}, 9(3):323 -- 348, 2005.

\bibitem{Shmerkin-Solomyak:2016:absolute_continuity_measures}
P.~Shmerkin and B.~Solomyak.
\newblock {Absolute continuity of self-similar measures, their projections and
  convolutions}.
\newblock {\em Trans. AMS}, 9368(7):5125 -- 5151, 2016.

\bibitem{Triebel:1983:theory_function_spaces_1}
H.~Triebel.
\newblock {\em {Theory of Function Spaces}}.
\newblock Modern Birkh{\"a}user Classics. Birkh{\"a}user, 1983.

\bibitem{Abry-Jaffard-Wendt:2009:wavelet_leader_texture_classification}
H.~Wendt, P.~Abry, S.~Jaffard, H.~Ji, and Z.~Shen.
\newblock Wavelet {L}eader multifractal analysis for texture classification.
\newblock In {\em 16th IEEE Int. Conf. Image Proc.}, pages 3829--3832, 2009.

\bibitem{Abry-Jaffard-Roux-Wendt:2009:Wavelet_Bootstrap}
H.~Wendt, S.~G. Roux, P.~Abry, and S.~Jaffard.
\newblock {W}avelet leaders and bootstrap for multifractal analysis of images.
\newblock {\em Signal Processing}, 89(6):1100--1114, 2009.

\end{thebibliography}

\end{document}